\documentclass[a4paper, 10pt]{amsart} % for a short document

\usepackage[a4paper,text={6.2in,9.5in},centering]{geometry}
\usepackage[T1]{fontenc}
\usepackage[utf8]{inputenc}
\usepackage[english]{babel}
\usepackage[pdfauthor={Andrea di Lorenzo, Lorenzo Mantovani}, %
                pdftitle={EquivariantChowWitt}%
            dvips,colorlinks=true]{hyperref}
\hypersetup{
    bookmarksnumbered=true,
    linkcolor=blue,
    citecolor=red,
}

\usepackage{amsmath,amsthm,stmaryrd}
\usepackage{amssymb,amsfonts}
\usepackage{booktabs}

\usepackage[mathscr]{euscript}
\usepackage{mathcomp}
\usepackage{tikz}
\usepackage{tikz-cd}
\usetikzlibrary{matrix,arrows,decorations.pathmorphing}
\usepackage[all,cmtip]{xy}
\usepackage{enumitem}
\usepackage{xcolor}
\definecolor{dark-blue}{rgb}{0,0,0.5}

\usepackage{verbatim}
\usepackage[
backend=bibtex,
style=alphabetic,
sorting=nyt
]{biblatex}
\usepackage[switch,mathlines,right,left,pagewise]{lineno}
\usepackage{cleveref}

\newcommand{\op}{\operatorname}

\newcommand{\bb}{\mathbb}
\renewcommand{\bf}{\mathbf}
\newcommand{\scr}{\mathscr}
\newcommand{\cal}{\mathcal}
\newcommand{\fk}{\mathfrak}

\renewcommand{\rm}{\mathrm}

\DeclareMathOperator\Hom{Hom}

\DeclareMathOperator\spec{Spec}

\newcommand{\ra}{\longrightarrow}

\newcommand{\hra}{}% teste si deja defini
\DeclareRobustCommand{\hra}{\lhook\joinrel\relbar\joinrel\rightarrow}
\newcommand{\hla}{}% teste si deja defini
\DeclareRobustCommand{\hla}{\leftarrow\joinrel\relbar\joinrel\rhook}
\DeclareMathOperator{\CH}{CH}
\DeclareMathOperator{\Ch}{Ch}
\DeclareMathOperator{\CHW}{\widetilde{{\rm CH}}}
\newcommand{\D}{\rm D}

\newcommand{\Gm}{\mathbb{G}_m}
\newcommand{\GW}{{\rm {GW}}}
\newcommand{\Ker}{\rm{Ker}}
\newcommand{\K}{\mathrm{K}}
\newcommand{\KM}{{\rm K}^{\rm M}}
\newcommand{\KMW}{{\rm K}^{\rm{MW}}}
\newcommand{\Km}{{\rm k}^{\rm M}}
\newcommand{\Lcal}{\mathcal{L}}
\newcommand{\Mbar}{\overline{\mathcal{M}}}
\newcommand{\Mcal}{\mathcal{M}}
\newcommand{\Ocal}{\mathcal{O}}
\newcommand{\rk}{\mathrm{rk}}
\newcommand{\Pb}{P}
\newcommand{\W}{{\rm W}}
\newcommand{\I}{\rm I}
\newcommand{\ZZ}{\mathbb{Z}}

\newcommand{\kc}[1]{\langle #1 \rangle}
\renewcommand{\th}{\operatorname{Th}}

\newcommand{\HH}{\rm H}
\newcommand{\T}{\rm T}
\newcommand{\E}{\rm V}
\newcommand{\U}{\rm U}

\theoremstyle{plain}
\newtheorem{thm}[subsubsection]{Theorem}
\newtheorem*{thm-no-num}{Theorem}
\newtheorem{prop}[subsubsection]{Proposition}
\newtheorem{lemma}[subsubsection]{Lemma}

\theoremstyle{definition}
\newtheorem{defin}[subsubsection]{Definition}

\theoremstyle{remark}
\newtheorem{rmk}[subsubsection]{Remark}

\title[Equivariant Chow-Witt groups and moduli stacks of elliptic curves]{Equivariant Chow-Witt groups \\ and moduli stacks of elliptic curves}
\author{Andrea Di Lorenzo}
\address{Department of Mathematics, Humboldt-Universit\"at zu Berlin , Berlin, Germany}
\email{andrea.dilorenzo@hu-berlin.de}
\author{Lorenzo Mantovani}
\address{Department of Mathematics, Johannes Gutemberg-Universit\"at Mainz, Mainz, Germany}
\email{lorenzo.mantovani@uni-mainz.de}
\begin{document}
\begin{abstract}
    We introduce equivariant Chow-Witt groups in order to define Chow-Witt groups of quotient stacks. We compute the Chow-Witt ring of the moduli stack of stable (resp. smooth) elliptic curves, providing a geometric interpretation of the new generators. Along the way, we also determine the Chow-Witt ring of the classifying stack of $\mu_{2n}$.
\end{abstract}
\maketitle
\tableofcontents

\section*{Introduction} % (fold)
Intersection theory on moduli of curves is a central area of study in algebraic geometry: in the last forty years, there has been a huge amount of works in this field. Among the different research directions, an interesting one consists in determining the structure of the integral Chow ring of moduli stacks of curves, a task which has been accomplished only in few cases (e.g. \cite{EG}, \cite{Vis}, \cite{Lar}, \cite{DFV}).

In \cite{BM,FGCW} the authors introduced more refined cycle groups, called \emph{Chow-Witt groups}, initiating the so-called \emph{quadratic} intersection theory. Roughly speaking, if the basic objects of study in intersection theory are cycles on a variety, i.e. formal sums of subvarieties with integral or rational coefficients, in Chow-Witt theory we consider cycles with coefficients in certain Grothendieck-Witt rings of quadratic forms.

The motivation for studying these new invariants is that they are more sensitive to the arithmetic of the ground field. For instance, quadratic intersection theory can be used to extend some classical results of enumerative geometry to general base fields, albeit at the price of "counting with quadratic forms" rather than with integers. For instance in \cite{KW} the authors generalize results of Severi asserting that the difference between the number of hyperbolic lines and elliptic lines on a real cubic surface is always equal to three. 
In a slightly different direction, Chow-Witt groups have turned out to be useful in the study of characteristic classes of bundles on varieties. Work in this direction has been carried out extensively in \cite{HW} and \cite{WendtGrass}, where characteristic classes of $\mathrm{SL}_n$-, $\mathrm{SP}_n$-, and $\mathrm{GL}_n$-bundles have been studied via an explicit description of the Chow-Witt ring of the  classifying stack of the groups. Similar questions for orthogonal groups are under investigation in \cite{MNW}.

In this framework we direct our attention to moduli stacks of curves, looking for invariants of quadratic nature. The goal of this paper is to make some small steps in this direction, focusing on the moduli stack $\Mbar_{1,1}$ (resp. $\Mcal_{1,1}$) of stable (resp. smooth) marked curves of genus one. Our main results consist in defining Chow-Witt ring of quotient stacks and in determining a presentation of the Chow-Witt rings of $\Mbar_{1,1}$ and $\Mcal_{1,1}$ in terms of generators and relations.
\begin{thm-no-num}
Let $k$ be a perfect field of characteristic different from $2$ and $3$, and let $\GW(k)$ denote the Grothendieck-Witt ring of $k$, with fundamental ideal $I$. We have isomorphisms of $GW(k)$-algebras:
\[ \CHW^*(\Mbar_{1,1},\bullet)\simeq\GW(k)[\T,\E,\HH]/(I\cdot\T,I\cdot\HH,\HH^2-2h,h\E,\HH\E,\E^2,24\T^2,12\HH\T^2-\E\T)\]
\[\CHW^*(\Mcal_{1,1},\bullet)\simeq\GW(k)[\T,\D,\HH]/(I\cdot\T,I\cdot\HH,\HH^2-2h,h\D,\HH\D,\D^2+2\D,6\T\HH,12\T-\D\T).\]
\end{thm-no-num}

For these computations Chow-Witt classes are indexed by codimension $\ast \in \bb N$ and by twisting $\bullet \in \rm{Pic}/2$. Since the grading is a delicate matter, we direct the reader to \ref{subs:grading} and to the beginning of Section \ref{sec:CW Bmu2n} for more details on our conventions for definitions and for computations respectively. A treatment with more details is given in Section 1.3 of \cite{FasLectures} and in Appendix A of \cite{feld1}.

Here is a brief explanation of the significance of the multiplicative generators that appear above, but we encourage the reader to refer to \Cref{thm:CW ring of Mbar} and \Cref{thm:CW ring of Mcal} for a precise description, and \ref{subsec:inter} for a geometric description. The element $\T$ is an element in degree $(1, \cal E)$ corresponding to the Euler class of the dual of the Hodge line bundle $\cal E$ on $\Mbar_{1,1}$.

The element $\HH$ lives in degree $(0,\cal E)$ and is a sort of "twisted" version of hyperbolic plane, and its presence is somehow not a surprise. More interestingly, the generator $\E$ is an element in degree $(1,\cal E^{\otimes -12})$, it is supported on the divisor of nodal curves $\Delta_0$, and it is related to the discriminant of the universal family of nodal curves. The elements $\T,\HH$ restrict to $\Mcal_{1,1}$ maintaining the same meaning, whereas $\D$ corresponds to a class of degree $(0,\cal O)$ associated to the quadratic form on $\cal E^{\otimes 6}$ induced by the discriminant of the universal elliptic curve.

To define Chow-Witt groups of quotient stacks, we follow closely the ideas of \cite{totaro_chowring} and \cite{EG}. These ideas have already been adapted to work with Chow-Witt theory in \cite{HW,WendtGrass}, but we need some further adjustment in order to work with singular stacks and to keep track of twists. Since our argument is rather abstract, we take the chance for introducing a good definition of $G$-equivariant Chow groups for $G$-varieties with coefficients in a Milnor-Witt cycle module in the sense of Feld (cf Section 4 of \cite{feld1}).

At the moment we restrict ourselves with $G$-varieties $X$ satisfying some technical assumption (cf. \ref{subs:assumptions_on_X}) implying that that the resulting quotients $[X/G]$ may be approximated by open schemes in suitable vector bundles over $[X/G]$. Most likely, one can make this theory work in higher generality, but to maintain the present paper as self-contained as possible we preferred to work under these hypotheses, that hold for all the $G$-varieties appearing in this paper.

Our \Cref{def:equivariant groups} of equivariant Chow groups with coefficients for a $G$-variety $X$ only depends on the associated quotient $[X/G]$ (see \Cref{prop:indep_of_choices_in_def_of_equiv_hom}), giving thus a reasonable cycle theory with coefficients for (possibly singular) quotient stacks that can be nicely approximated by schemes. 

To understand the multiplicative structure of $\Mbar_{1,1}$ and $\Mcal_{1,1}$ we also need some information on the Chow-Witt ring of the classifying stacks $\cal B\mu_{2n}$. The same approach used in \cite{MNW} for $n=1$ can be extended for other values of $n$. In addition we provide a description of the multiplicative structure.
\begin{thm-no-num}
  Let $k$ be a perfect field of characteristic different from $2n$. We have an isomorphism of $\GW(k)$-algebras 
  \[ \CHW^*(\scr B\mu_{2n},\bullet)\simeq\GW(k)[\T,\HH,\U]/(I\cdot\T,I\cdot\HH,h\U,\HH\U, n\T\HH, \HH^2-2h,\U^2+2\U,\T\U- 2n\T).\]
\end{thm-no-num}
The classes $\T,\HH$ are found in degrees $(1,\cal U)$ and $(0,\cal U)$ respectively, where $\cal U$ denotes the universal line bundle on $\scr B \mu_{2n}$, and they represent the Euler class $e(\cal U^\vee)$ and a variation of the hyperbolic plane respectively. The class $\U$ lives in degree $(0,\cal O)$ and is induced by the tautological quadratic form on $\cal U^{\otimes n}$.
In future works, we aim further developing the machinery of equivariant quadratic intersection theory (e.g. by developing localization formulas) with the goal in mind to tackle more advanced questions about the structure of Chow-Witt rings of moduli stacks of curves, and related enumerative problems.

\subsection*{Outline of the paper}

In \Cref{sec:preliminaries_on_chow_witt_groups} we recall some foundational aspects of Chow-Witt theory and, more generally, of homology groups with coefficients in a Milnor-Witt cycle modules. We then extend these notions to the equivariant setting, and thus to quotient stacks.

In \Cref{sec:preliminaries} we recall some properties of Chow-Witt groups that will be used in the subsequent computations.

In \Cref{sec:stacks_of_elliptic_curves} we collect some geometric facts on the classifying stacks $\scr B\Gm$ and $\scr B\mu_{2n}$ and on the moduli stacks $\Mbar_{1,1}$ and $\Mcal_{1,1}$. The geometry of these stacks of curves is tightly related to that of some vector bundles on $\scr B \Gm$ and to $\scr B\mu_{2n}$. It is this tight relation that allows us to cleanly carry out our approach.

In \Cref{sec:CW Bmu2n}, after recapping what is known on the Chow-Witt ring of $\scr B\Gm$, we compute the Chow-Witt ring of $\scr B\mu_{2n}$: we first describe its additive structure (\Cref{prop:CW groups Bmu2n}) and subsequently its multiplicative structure (\Cref{thm:CW ring of Bmu2n}).

In \Cref{sec:Chow-Witt_ring_of_Mbar} we compute the Chow-Witt rings of $\Mbar_{1,1}$ and $\Mcal_{1,1}$. We warm up by computing the $\I^\ast$-cohomology of $\Mbar_{1,1}$ (\Cref{prop:Ij cohomology Mbar}), which we need as input of the main computation. Then we determine the additive structure of Chow-Witt groups (\Cref{thm:additive structure Mbar} and \Cref{thm:additive structure Mcal}) and after that we conclude with the multiplicative part (\Cref{thm:CW ring of Mbar} and \Cref{thm:CW ring of Mcal}). Finally we give a geometric description of the new multiplicative generators appearing in the descriptions above (cf.  \ref{subsec:inter}).

In \Cref{sec:poor} we give a gentle introduction to Chow-Witt groups and some of their main properties. This Section is aimed for those readers who have little or no background in algebraic K-theory and motivic cohomology theories, but perhaps they are more acquainted with the usual theory of Chow groups. %First we sketch the construction of Chow groups both as groups of cycles modulo rational equivalence and as homology groups with coefficients in Milnor K-theory. In a similar fashion, we then give a brief account of Chow-Witt groups as groups of \emph{cycles with coefficients in quadratic forms} and subsequently as homology groups with coefficients in Milnor-Witt K-theory. We end the Section by introducing the key concept of \emph{twisted} Chow-Witt groups, which constitutes one of the main differences between the quadratic theory and the classical one.

\subsection*{Conventions and notation}
All schemes are assumed to be of finite type over a perfect field $k$ of characteristic $\neq 2$. All stacks are stacks on the site of schemes over $k$, endowed with the fppf topology.
We recall the following standard notation:
\smallskip
\begin{center} 
\begin{tabular}{lp{0.5\textwidth}}
\toprule
$\GW(-)$ & the Grothendieck-Witt ring\\
$\W(-)$ & the Witt ring \\
$\I(-)$ & the fundamental ideal both in $\GW(-)$ and $\W(-)$\\
$\KM_\ast(-)$ & the Milnor K-theory ring\\
$\KMW_\ast(-)$ & the Milnor-Witt K-theory ring\\
$\Km_\ast(-)$ & the mod $2$ Milnor K-theory ring.\\
\bottomrule
\end{tabular}
\end{center}
\smallskip
When no field is specified in the symbols above, we are implicitly thinking of corresponding object for the ground field. If $\K_\ast$ is a Milnor-Witt cycle module and $X$ is a variety (or a quotient stack), the notation
\[H_{i,j}(X,\K_\ast,\cal L)\]
always refers to the homology groups of a suitable cycle complex with coefficients in $\K_\ast$. Upper indices are used only for smooth varieties (resp. stacks) when one can actually harmlessly confuse the homology groups of such cycle complexes with the Nisnevich cohomology groups of unramified sheaves associated with $\K_\ast$. 
%For a scheme or quotient stack $X$ and ${\K}_\ast\in \left\{\KMW_\ast,\KM_\ast, \Km_\ast, \I^\ast\right\}$, the notation $H^\ast(X,{\rm K}_\ast,\bullet)$ stands for the direct sum of the cohomology groups $H^i(X,{\rm K}_\ast,\Lcal)$, for every $i\geq 0$ and every element $\Lcal$ in ${\rm {Pic}}(X)\otimes\ZZ/2$.

We will use the following standard notation for Cartier divisors: if $\{U_i\}_{i\in I}$ is an affine covering of a scheme $X$, the notation $\{(U_i,f_i) \}_{i\in I}$ stands for the Cartier divisor $D\subset X$ whose equation in each affine open subset $U_i$ is $f_i=0$.

For a locally free sheaf $F$ on a variety $X$, we will denote $\bb V(F)$ the associated vector bundle on $X$. We adopt here the convention $\bb V(F):=\spec{\rm{Sym}(F^{\vee})}$, where $F^{\vee}$ is the dual of $F$.

\subsection*{Acknowledgments}
This project started in December 2019, when the first named author was visiting Andrew Kresch at the University of Zurich. We thank Andrew for this and for several others useful conversations on the topic. We wish to thank N. Feld and J. Fasel for answering some questions on Milnor-Witt cycle modules. In addition the second named author wishes to sincerely thank M. Wendt for the huge influence on his understanding of the subject of this paper. The second named author also thanks J. Ayoub for clarifying a number of doubts on the six functors formalism. Finally, we thank the referee for several suggestions that greatly improved the readability of the paper.

The second named author was supported by the \emph{Swiss National Science Foundation} (SNF), project 200020\_178729.
% section introduction (end)

\section{Equivariant Chow-Witt groups} % (fold)
 \label{sec:preliminaries_on_chow_witt_groups}

   The ultimate goal of this section is to introduce a sensible notion of Chow-Witt groups with coefficients for quotient stacks. In \ref{sub:classical_theory} we briefly recall some notation on cycle complexes with coefficients in a Milnor-Witt cycle modules twisted by virtual vector bundles. We compare the homology groups of such complexes with motivic Borel-Moore homology of the spcetrum associated with the cycle module. For smooth equidimensional varieties we introduce a cohomological notation too, indexing cycles by codimension. In \ref{sub:equivariant_chow-witt_groups} we extend the notion of homology and cohomology of cycle complexes to the equivariant setting via the now standard technique introduced by Totaro in \cite{totaro_chowring} and further developed in \cite{EG} and in \cite{HW}. We conclude with an observation on how the equivariant Borel-Moore homology theory introduced here is actually an invariant of the associated quotient stack.
   
   In the second part of the section we will fix a smooth affine algebraic group $G$ over $k$.
    
  %This is achieved via a two-step process, obtained by adopting the same strategy of \cite[Section 2]{EG}: first, given a scheme $X$ endowed with the action of an algebraic group $G$ satisfying some mild assumptions, we define the \emph{equivariant} Chow-Witt groups $\CHW_G^i(X,\bullet)$ of the $G$-scheme $X$ (\Cref{def:equivariant groups}). Second, given a $G$-scheme $X$ and an $H$-scheme $Y$, we prove that if $[X/G]\simeq [Y/H]$ then the associated equivariant Chow-Witt groups are isomorphic (\Cref{prop:invariance for quotients}). This allows us to define the Chow-Witt groups of a quotient stack $\scr X=[X/G]$ as the $G$-equivariant Chow-Witt groups of $X$, which we proved in the second step to not depend on the chosen presentation of $\scr X$ as a quotient.
  
  %To be precise, what we actually define are some slightly more general gadgets, namely homology groups of quotient stacks with coefficients in a generalized cycle module. By taking the sheaf $\KMW_*$ of Milnor-Witt K-theory, we get the aforementioned Chow-Witt groups.
  
  %The Section is organized as follows: we start  with a recap of the Chow-Witt theory, with a focus on the notion of the \emph{twist by a vector bundle}, which constitutes one of the major technical subtleties here. Next, we explain in \Cref{sub:equivariant_chow-witt_groups} how to define equivariant homology groups via the so called \emph{equivariant scheme approximations}, already introduced in \cite[Section 2]{EG}. These equivariant groups are actually an invariant of the associated quotient stack.
  \subsection{Classical Theory} % (fold)
   \label{sub:classical_theory} 
    
    \subsubsection{}
    \label{subs:grading}
      
      Chow-Witt groups are naturally indexed on graded line bundles and more generally on virtual vector bundles; we recall these notions first. 

      Let $X$ be a variety of finite type over $k$. The graded Picard groupoid $\cal G(X)$ is the category of pairs $(\cal L,a)$ where $\cal L $ is a line bundle on $X$ and $a\in \op H^0(X,\bb Z)$; the set of morphisms of $\cal G(X)$ from $(\cal L,a)$ to $(\cal M,b)$ is that of $\cal O_X$-linear isomorphisms $\varphi: \cal L \ra \cal M$ if $a=b$ and $\emptyset$ otherwise. The groupoid $\cal G(X)$ is made into a symmetric monoidal category via $(\cal L,a)\otimes (\cal M,b):=(\cal L\otimes \cal M,a+b)$ and the symmetry isomorphism $(\cal L,a)\otimes (\cal M,b) \ra (\cal M,b) \otimes (\cal L,a)$ is defined to be $l\otimes m \mapsto (-1)^{ab}m\otimes l$. 

      Let us consider now the Quillen K-theory space $K(X)=\Omega\op{BQVect}(X)$ of $X$, and let us denote by $\underline K(X)$ its fundamental groupoid. In 4.2 of \cite{deligne:det_coh} Deligne calls $\underline K(X)$ the category of virtual objects of the exact category $\op{Vect}(X)$. We will denote by $+$ the monoidal structure induced on $\underline K(X)$ from the direct sum of vector bundles, and by $\underline 0$ the neutral element. Similarly for $r \in \bb Z$, the symbol $\underline r$ will denote the virtual vector bundle of rank $r$. The category $\underline K(X)$  is used to index the twists of Chow-Witt groups with coefficients in the work of Feld. The grading by $\underline K(X)$ is related to the grading by $\cal G(X)$ via the graded determinant $D:\underline K(X) \ra \cal G(X),v \mapsto (\det(v),\rk(v))$, which turns out to be a symmetric monoidal functor (cf. 4.13 of Deligne's work).

      If $\scr X=[X/G]$ is a quotient stack one gives a similar definition of $\underline K(\scr X)$ starting from the exact category of vector bundles $\op{Vect}(\scr X)$ on $\scr X$, or equivalently from the exact category of $G$-vector bundles $\op{Vect}_G(\scr X)$ on $X$. 

      Let $\scr E$ be an exact category (in the sense of Section 2 of \cite{quillen:hakt}) which is a full subcategory of an abelian category $\scr A$ so that the inclusion functor be exact. By combining Theorem 1.11.2 and Theorem 1.11.7 of \cite{thomason_trobaugh} we have a canonical homotopy equivalence of spaces between the Quillen $K$-theory space $K(\scr E)$ and the Waldhausen $K$-theory space of $\scr E^{\sim}$, where $\scr E^{\sim}$ is the category of bounded chain complexes in $\scr E$ with a suitably defined Waldhausen structure (cf. 1.11.6 where the players of Theorem 1.11.7 are properly introduced). In our context, this implies that a bounded complex of vector bundles on a scheme (or algebraic stack) $X$ gives a well defined object of $\underline K(X)$. To be even more general, one can replace $\cal E^\sim$ with its associated $\infty$-category, and obtaining an equivalent $K$-theory space (cf. section 7 of \cite{bgt}). In particular the cotangent complex $L_f$ of a morphism $f: X \ra Y$ between smooth schemes over a field $k$, has always a canonically defined object $L_f \in \underline K(X)$.

    \subsubsection{}
    \label{subs:rs_complexes}
      Let $X$ be a $k$-variety, let $(\cal L,a)$ be a graded line bundle on $X$ and $j$ an integer. For every point $x\in X$ denote by $k(x)$ the residue field at $x$, and by $\Omega_{k(x)/x}$ the $k(x)$-vector space of relative differential $1$-forms. The Rost-Schmid complex $C_\bullet(X,\KMW_\ast,j,\cal L,a)$ with coefficients in Milnor-Witt K-theory is the chain complex with $i$-th term 
      \begin{equation}
        \label{eqn:KMW-cyccle_complex}
          C_i(X,\KMW_\ast, j,\cal L,a)=\bigoplus _{x \in X_{(i)}}\KMW_{i+j}(k(x),D(\Omega_{k (x)/k})\otimes \cal (L_x,a)).
      \end{equation}
      Here $\KMW_p(k(x),D(\Omega_{k (x)/k})\otimes \cal (L_x,a))$ denotes the degree $p$ Milnor-Witt K-theory of the residue field at $x$, twisted by the one dimensional graded $k(x)$-vector space $D(\Omega_{k (x)/k})\otimes \cal (L_x,a)$. We redirect the reader to \cite[Section 1.1]{FasLectures} for a detailed definition of twisted Milnor-Witt K-theory, or to \Cref{sec:poor} for a brief account. 

      The differential $d_i: C_i(X,\KMW_\ast,j,\cal L,a) \ra C_{i-1}(X,\KMW_\ast,j,\cal L,a)$ is constructed using a combination of the residue maps for discrete valuation rings and the geometric transfer maps in Milnor-Witt K-theory: the construction can be found in Section 2.1 of \cite{FasLectures}. %Note however that in comparision with Section 2.1 of \cite{FasLectures} our Rost-Schmid complexes are shifted by $a$. The need for this shift is clear, for instance, if one observes the statement of Theorem 2.15 of \cite{FasLectures}.
      With our definitions $C_\bullet(X,\KMW_\ast,j,\cal L,a)$ is a well defined chain complex, which we regard as concentrated in range $[d,0]$. 
    
    \begin{comment}
      \subsubsection{}
        \label{subsub:graded_twists}
      Observe that when we have an isomorphism of line bundles $\phi: \cal L \ra \cal L'$ there is a naturally induced isomorphism of Rost-Schmid complexes 
      \[C(\phi):C_\bullet(X,\KMW_\ast,j,\cal L)\overset{\simeq}{\longrightarrow} C_\bullet(X,\KMW_\ast,j,\cal L').\] As explained at the beginning of Sections 1.5 and 2.1 of \cite{FasLectures}, we should use these sisomorphisms to make the Rost-Schmid complexes functorial on the Graded Picard Grupoid $\cal G(X)$ of $X$. Recall that $\cal G(X)$ is the category of pairs $(\cal L,a)$ where $\cal L $ is a line bundle on $X$ and $a\in \op H^0(X,\bb Z)$; the morphisms of $\cal G(X)$ from $(\cal L,a)$ to $(\cal M,b)$ are $\rm{Isom}(\cal L,\cal M)$ if $a=b$ and $\empty$ otherwise. $\cal G(X)$ is made into a symmetric monoidal category via $(\cal L,a)\otimes (\cal M,b):=(\cal L\otimes \cal M,a+b)$ and the symmetry isomorphism $(\cal L,a)\otimes (\cal M,b) \ra (\cal M,b) \otimes (\cal L,a)$ is defined to be $l\otimes m \mapsto (-1)^{ab}m\otimes l$. We then define 
      \[C_\bullet (X,\KMW_\ast,j,\cal L,a):=C_\bullet(X,\KMW_\ast,j,\cal L)[a].\]
   \end{comment}
    \begin{defin}
      \label{defin:cw_tw_by_lb}
      The Chow-Witt group of $i$-dimensional cycles on $X$ twisted by $\cal (L,a) \in \cal G(X)$ is 
      \[\CHW_i(X,\cal L,a)=H_i(C_\bullet(X,\KMW_\ast,-i,\cal L,a)).\] Similarly we define
      \[H_{i,j}(X,\KMW_\ast,\cal L,a)=H_i\big (C_\bullet(X,\KMW_\ast,-j,\cal L,a) \big ).\]
    \end{defin} 
    The definition of Chow-Witt groups above, for $a=1$, coincides with the one given in \ref{sub:chow-witt app}.
    
    \subsubsection{}
      \label{subsub:rs_complex_twisted_by_vb}
      In fact for having a natural comparison with other motivic theories one should really consider twists by $v \in \underline K(X)$, defining
      %\[C_\bullet(X,j,\langle v\rangle):=C_{\bullet}\big( (X,j-\rm{rk}(v),\det(v)^\ast\big )[-\rm{rk}(v)].\]
      \[C_\bullet(X,\KMW_\ast,j,\kc v):=C_{\bullet}\big( X,\KMW_\ast,j,D(v)\big ).\]
    
    \begin{defin}
      \label{defin:CW}
      \label{defin:cw_twisted_by_kc}
      The Chow-Witt group of $i$-dimensional cycles on $X$ twisted by $v \in \underline K(X)$ is
      \[\CHW_i(X,\kc v):=H_i(C_\bullet(X,\KMW_\ast,-i,\kc v))=H_i\big (C_\bullet(X,\KMW_\ast,-i,D(v))\big ).\] Similarly
      \[H_{i,j}(X,\KMW_\ast,\kc v)=H_i\big (C_\bullet(X,\KMW_\ast,-j,\kc v) \big )=H_i\big (C_\bullet(X,\KMW_\ast,-j,D(v))\big ).\]
      
      %\[H_i(C_\bullet(X,j,V))=H_{s+\rm{rk}(V)}(C_{\bullet}(X,-\rm{rk}(V)-n,\det(V^*)))\]
      %When $X$ is smooth over K and equidimensional of dimension $d$, we similarly define 
      %\[\CHW^i(X,\cal L)=H^i(C^\bullet(X,i,\cal L))=\]
    \end{defin} 
    \subsubsection{}
    Let $\underline n$ denote a virtual vector bundle of rank $n$. We have isomorphisms 
    \begin{align*}
      \CHW_i(X,\kc v) & \simeq \CHW_i(X,\kc{v+\underline n}) \\
      H_{i,j}(X,\KMW_\ast,\kc v) & \simeq H_{i,j}(X,\KMW_\ast,\kc{v+n}),
    \end{align*}
    since modifying the rank of $v$ does not influence the chain complex $C_\bullet(X, \KMW_\ast,j,\kc v)$. However the isomorphism $\underline n+v\simeq v +\underline n$ induces an isomorphism on the homology of the Rost-schmid complexes, whose sign depeneds on $n$. 
    
    \subsubsection{}
        \label{subsub:gen_cyc_mods}
        Milnor-Witt K-theory is a particular example of Milnor-Witt cycle module in the sense of \cite{feld1} (cf. Theorem 3.20 of \cite{feld1} for a proof). Given any Milnor-Witt cycle module $\K_\ast$, Feld assigns to every $k$-variety $X$ and every $v\in \underline K(X)$ a Milnor-Witt cycle complex (Definition 5.4 of \cite{feld1})  
        \[(C^{F}_\bullet(X,\K_\ast,v),d^F)\] concentrated in homological degrees $[d,0]$, and corresponding Chow-Witt groups with coefficients (\cite[Definition 7.2]{feld1})
        \[A_i(X,\KMW_\ast,v):=H_i(C^{F}_\bullet(X,\K_\ast,v)).\]
        Let us denote by $\underline n$ the trivial virtual vector bundle of rank $n$. The complex we have introduced in \ref{subsub:rs_complex_twisted_by_vb} compares with Feld's definition as follows:
        %\[C^{F}_\bullet(X,\KMW_\ast, \underline n+v)[\rm{rk}(v)]=C_\bullet(X,\KMW_\ast,n+\rm{rk}(v),\det(v))[\rm{rk}(v)]=C_\bullet(X,\KMW_\ast,n,\kc v),\]
        \[C^{F}_\bullet(X,\KMW_\ast, v)=C_\bullet(X,\KMW_\ast,\rm{rk}(v),D(v)) \quad d^F_i=d_i,\]
        and in particular the two cycle theories compare via canonical isomorphism
        \[A_{i}(X,\KMW_\ast,v) \simeq H_{i,-\rk (v)}(X,\KMW_\ast,\kc v).\]
          \begin{comment}
             \textcolor{red}{check!
        \begin{align*}
          A_{i}(X,\KMW_\ast,v) & \simeq H_{i,-\rk (v)}(X,\KMW_\ast,\kc v), \\
          \op H_{i,j}(X,\KMW_\ast,\kc v) & \simeq  A_i(X,\KMW_\ast,v-\underline {\rk(v)}-\underline j), \\
          \op H_{i,j}(X,\KMW_\ast,\cal L,a) &  \simeq A_i(X,\KMW_\ast,\cal L-\underline 1-\underline j),
        \end{align*}
        Given any Milnor-Witt cycle module $\K_\ast$ we might then by analogy define
        \[\op H_{i,j}(X,\K_\ast,\kc v) = A_i(X,\K_\ast,v-\underline {\rk(v)}-\underline j).\]
        }
          \end{comment}
        \begin{comment}
           With a little abuse of notation, we use
        \[C_\bullet(X,\K_\ast,n,\kc v) \; \textrm{ and } \;  H_{i,n}(X,\K_\ast,\kc v)\]
        to denote 
        \[C^{F}_\bullet(X,\K_\ast, \underline n+v)[\rm{rk}(v)]\;\textrm{ and }\; A_{i-\rm{rk}(v)}(X,\K_\ast,v-\underline n)\]
        respectively. This notational shift is justified by the following result.
        \end{comment}

        Recall that in \cite{feld2}, the author constructs an equivalence between the category of Milnor-Witt cycle modules and the heart of the motivic stable category $\scr{SH}(k)^\heartsuit$. In Section 3.2 of \cite{feld2} the author indeed defines a functor $E\mapsto \widehat{E}$ that associates with every $E\in \scr{SH}(k)$ the functor
        \[\widehat {E}: (L,v) \mapsto [\bf 1_{\spec L}, \Sigma^{-\rk(v)}\th(v)]_{L},\]
        where $k\subseteq L$ is any finitely generated field extension, $v\in \underline K(L)$, $\th(-)$ denotes the Thom space construction, and $[-,-]_{L}$ denotes the $\Hom$ groups of the homotopy cateogory of $\scr{SH}(L)$. The author then proves (cf. Theorem 3.3.3 of \cite{feld2}) that $\widehat {E}$ has canonically the structure of a Milnor-Witt cycle module, and that $E\mapsto \widehat{E}$ is an equivalence when restricted to $\scr{SH}(k)^\heartsuit$ (cf. Section 4 of \cite{feld2}).

        Recall that every Milnor-Witt cycle module $\K_\ast$ yelds a Borel-More homology theory on $k$-varieties of finite type by simply taking homology groups of the complex $C^F_\ast(X,\K_\ast,v)$. On the other hand, with every object $E \in \scr{SH}(k)$ we can associate a Borel-More homology theory for $k$-varieties of finite type by setting:
        \[\op E_{i,j}^{BM}(X/k,\kc v):=[\Sigma^{i,j}\th(v),p_X^! E]_{X},\]
        where $i$ and $j$ are integers, $v\in \underline K(X)$, $\th(v)$ is the Thom spectrum of $v$ and $p_X:X \ra \spec k$ denotes the the structure morphism. More generally we will denote by $p_T$ the structure morphism of any $k$-scheme $T$.
    
    \begin{prop}
      Let $X$ be any separated $k$-variety of finite type, $v\in \underline K(X)$ a virtual bundle. Moreover let $E\in \scr{SH}(k)^\heartsuit$ and $\K_\ast=\widehat E$ be the corresponding Milnor-Witt homotopy module. Then we have a natural identifications
      %\[H_{i,j}(X,\kc v)\simeq H^{BM}_{i+j,j}(X,\kc v),\]
      %\[H_{i,j}(X,\kc{-v})\simeq H^{BM}_{i+j,j}(X/k,\kc v,  \KMW_\ast),\]
      \begin{equation*}
      % \op H_{i+\rk(v),n+\rk(v)}(X,\K_\ast,\kc{-v})=
       A_{i+\rk(v)}(X,\K_\ast,-v-\underline n) \simeq \op E^{BM}_{i+n,n}(X/k,\kc v).
      \end{equation*}
      \begin{comment}
        \textcolor{blue}{In particular
      \begin{equation*}
        H^{prima}_{i,n}(X,\kc v)=H_{i-\rk(v),n-\rk(v)}(X,\kc v)=A_{i-\rk(v)}(C,v-\underline n )=E^{BM}_{i+n,n}(X,-v)
      \end{equation*}}
      \end{comment}

      %In particular when $\K_\ast=\KMW_\ast$ we have
      %\[H_{i,n}(X,\KMW_\ast,\kc{-v})\simeq E^{BM}_{i+n,n}(X/k,\kc v).\]
      %When $X$ is a smooth K-variety the canonical map...
      %\[H^i(X,\cal \KMW_{j}(\cal L))\simeq H^i(\widetilde C^i(C,j,\cal L)).\]  
      %\textcolor{dark-blue}{Finish and find reference}            
    \end{prop}
    \begin{proof}
        We use the niveau spectral sequence on $X$ converging to the cohomology of the Thom space $\th(v)$ represented by the motivic spectrum $p^! E \in \scr{SH}(X)$, (cf. Definition 3.1.5 of \cite{BD}). In terms of Borel-Moore homology the spectral sequence reads:
        \[E^1_{p,q} = \bigoplus _{x \in X_{(p)}} \; \rm{colim}_{U\subseteq \overline {\{x\}}} \op E^{\rm{BM}}_{p+q+n,n}(U, \kc{v_{|U}})\Rightarrow \op E^{\rm{BM}}_{p+q+n,n}(X,\kc v),\]
        with differential $d^k_{p,q}: E^k_{p,q} \ra E^k_{p-k,q+k-1}$, and where colimits range over the collection of open subsets of the closure of $\{x\}$. We want to show that the complexes $E^1_{_\ast,q}$ can be canonically identified with $C^F_\ast(X,\widehat E',-v-\underline n)$, where $E':=\Sigma^{-\rm{rk}(v)-q}E$.

        Unrolling definitions the spectral sequence reads:
        \[E^1_{p,q}=\bigoplus _{x \in X_{(p)}} \; \rm{colim}_{U\subseteq \overline {\{x\}}} [\Sigma^{p+q+n,n}\th (v_{|U}), p_U^!E]_U \Rightarrow [\Sigma^{p+q+n,n}\th(v),p_X^! E]_{X}.\]
        For every $x\in X$ the colimits appearing on the page can be restricted to smooth open subschemes U of the closure of $x$. Combining this with the continuity property of cohomology, we canonically identify
        \[\rm{colim}_{U\subseteq \overline {\{x\}}} [\Sigma^{p+q+n,n}\th (v_{|U}), p_U^!E]_U \simeq [\Sigma^{p+q+n,n}\th (v_x),  p_x^\ast E \otimes \th(\Omega^1_{k(x)/k})]_x,\]
        where $v_x$ denotes the pull-back of $v$ to $k(x)$. If we set $r:=\rm{rk}(v)$ then $E'=\Sigma^{-r-q}E$, and we obtain the following chain of canonical isomorphism
        \begin{equation*}
          \begin{split}
            [\Sigma^{p+q+n,n}\th (v_{|x}), p_x^\ast E \otimes \th(\Omega^1_{k(x)/k}) ]_x  & \simeq [ \bf 1_{k(x)}, p_x^\ast E \otimes \Sigma ^{-p-q+n}\th(\Omega^1_{k(x)/k}-v_x -\underline n) ]_x \\
            & \simeq [ \bf 1_{k(x)}, p_x^\ast E \otimes \Sigma ^{-q-r}\Sigma ^{-p+r+n}\th(\Omega^1_{k(x)/k}-v_x -\underline n) ]_x \\
            & \simeq [ \bf 1_{k(x)}, p_x^\ast E' \otimes \Sigma ^{-p+r+n}\th(\Omega^1_{k(x)/k}-v_x -\underline n) ]_x \\
            & = \widehat{E'}(k(x), \Omega^1_{k(x)/k}-v_x-\underline n)\\
            & = \widehat{E'}(x,-v_x-\underline n).
          \end{split}
        \end{equation*}
        The last two equalities holds literally by definition (cf. the beginning of Section 3.2 of \cite{feld2} and the beginning of Section 4 of \cite{feld1} for the second last and last equalities respectively). Thus
        \[E^1_{p,q} \simeq \bigoplus_{x \in X_(p)} \widehat{E'}(x, \Omega^1_{k(x)/k}-v_x-\underline n) = C_p^{F}(X,\widehat{E'},-v-\underline n),\] 
        where the last equality is the definition of the cycle complex of a Milnor-Witt cycle module (cf. Definition 5.4 of \cite{feld1}).

        The next step consists in identifying the differentials, i.e. proving that the diagram 
        \begin{equation*}
          \begin{tikzcd}
          E^1_{p,q} \arrow[r,"d^1_{p,q}"] \arrow[d,"\simeq"'] & E^1_{p-1,q} \arrow [d,"\simeq"]\\
          C_p^{F}(X,\widehat{E'},-v-\underline n) \arrow[r,"d^F"] & C_{p-1}^{F}(X,\widehat{E'},-v-\underline n)
        \end{tikzcd}
        \end{equation*}
        commutes. The proof of this fact is very similar to the proof of Theorem 3.3.2 of \cite{feld2}, so we don't repeat it here, but we rather suggest the required modifications to Feld's argument. In the proof of 3.3.2 the author is working with the niveau spectral sequence constructed out of $E$-cohomology with support, whereas here we have used $E$-Borel-More homology. Replacing every occurence of $E$-cohomology with $E$-Borel-More homology relative to the base field $k$, and every occurrence of $E$-cohomology with support with $E$-Borel-More homology of the support does the trick. The ingredients statements 2.3.7-10 of \cite{feld2}, work for $E$-Borel-More homology directly. Finally the smoothness assumtion of the Theorem 3.3.2 can be comppletely dropped, since this assumption is only used to replace cohomology of the ambient variety $X$ with support in a closed subscheme $T$ with the actual cohomology of the subscheme. This identification comes for free with Borel-More homology, for any finite type variety $X$.

        We are now ready to finish the argument. We have 
        \begin{equation*}
          \begin{split}
            [\Sigma^{p+q+n,n}\th (v_{|x}), \th(\Omega^1_{k(x)/k})\otimes p_x^\ast E]_x  & \simeq [\Sigma^{p+q+n+2r,n+r},\Sigma^{2p,p}p_x^\ast E]\\
          % E^{p-q-2r-n,p-r-n}(\spec k(x))\\
            & \simeq \underline \pi_{q+r}(E)_{p-r-n}(\spec k(x)),\\
          \end{split}
        \end{equation*}
        where $\underline \pi_k(E)_j$ denotes the $k$-th homotopy module of the spectrum $E$ in weight $j$. Since $E\in \scr{SH}(k)^{\heartsuit}$, we have that $\underline \pi_k(E)_j=0$ if $k\not = 0$ independently of $j$.  The niveau spectral sequence thus collapses at page $E^2$, yielding a canonical isomorphism 
        \[\op E^{BM}_{i+n,n}(X/k,\kc v)=H_{i+r}(C^{F}(X,\K_\ast,-v-\underline n)).\]
        \end{proof}

    \subsubsection{}
        When $p: X \ra \spec k$ is a smooth $k$-variety, Ayoub's purity equivalence 
        \[p^!(-)\simeq\th(\Omega^1_{X/k})\otimes p^\ast(-)\] in $\scr{SH}(X)$ induces an isomorphism in any motivic cohomology theory $E$
        %\[H_{i,j}(X,\kc{-\Omega^1_X+v})\simeq H^{BM}_{i+j,j}(X/K,\kc{\Omega^1_{X}-v}, \KMW_\ast)\simeq H^{-i-j,-j}(X,v, \KMW_\ast) .\] 
        \[\op E^{BM}_{i+j,j}(X/K,\kc{\Omega^1_{X/k}-v})\simeq \op E^{-i-j,-j}(X,v).\]
        This motivates the following Definition.
    \begin{defin}[c.f \cite{FasLectures} Section 2.1]
      \label{defin:coh_RS_cpx}
      Let $X$ be a smooth equidimensional $k$-variety of dimension $d$, and $\K_\ast$ a Milnor-Witt cycle module. For a graded line bundle $(\cal L,a)$ and a virtual vector bundle $v \in \underline K(X)$ we define 
      %\[H^{i,j}(X,\kc V):=H_{-i,-j}(X,\kc{\Omega^1_X-V})\]
    %  \textcolor{blue}{
      \begin{align*}
        C^\bullet(X,\KMW_\ast,j,\cal L,a) & :=C_{d-\bullet}(X,\KMW_\ast,j-d, D(\Omega^1_X)^{-1}\otimes(\cal L,a)) 
         %C^\bullet(X,\KMW_\ast,j, \kc v) & := C_{d-\bullet}(X,\KMW_\ast,j-d, D(\Omega^1_X)^{-1}\otimes D(v))
      \end{align*}
        regarded as a cochain complexes concentrated in cohomological degrees $[0,d]$. Moreover we set 
      \begin{equation*}
        H^{i,j}(X,\KMW_\ast,\cal L,a) :=H^i(C^\bullet(X,\KMW_\ast,j,\cal L,a)) = H_{d-i,d-j}(X,\KMW_\ast,D(\Omega^1_X)^{-1} \otimes (\cal L,a)),
      \end{equation*}
      %\begin{equation*}
      %  H^{i,j}(X,\KMW_\ast,\kc{v}) :=H^i(C^\bullet(X,\KMW_\ast,j,\kc{v})) = H_{d-i,d-j}(X,\KMW_\ast,D(-\Omega^1_X)+v)),
      %\end{equation*}
      and the twisted cohomological Chow-Witt groups
      \begin{equation*}
        \CHW^i(X,\cal L,a) :=H^{i,i}(X,\KMW_\ast,\cal L,a ))=\CHW_{d-i}(X,D(\Omega^1_X)^{-1} \otimes (\cal L,a))
       \end{equation*}
       %\begin{equation*}
       % \CHW^i(X,v) :=H^{i,i}(X,\KMW_\ast,v))=\CHW_{d-i}(X,D(-\Omega^1_X+v)).
       %\end{equation*}
       \begin{comment}
         and the $v$-twisted cohomological Chow-Witt groups
         \begin{equation*}
           \CHW^i(X,\kc v):=\CHW^{i}(X,D(\Omega^1_X+v)).
         \end{equation*}
       \end{comment}
        %\[C_F^\bullet(X,\K_\ast,j,\kc v)=C^F_{d-\bullet}(X,\K_\ast,j-d, \kc{-\Omega^1_X+v})\]%=C_{d-r-\bullet}(X,\K_\ast,j-d+r,\omega^\vee\otimes\det v)\]
       % regarded as a cochain complex concentrated in cohomological degrees $[0,d]$, 
        %and consequently
        \begin{comment}
           Similarly we set
        \[H^{i,j}(X,\K_\ast,\kc v):=H_{d-i,d-j}(X,\K_\ast,\kc{-\Omega^1_X+v}).\]
        With this definition we have that 
        \[H^{i,j}(X,\K_\ast,\kc v)=A_{d-i}(X,\K_\ast,(-\Omega^1_X +v -\underline{\rk(v)}+\underline j)=:A^i(X,\K_\ast,(-\Omega^1_X +v -\underline{\rk(v)}+\underline j),\]
        according to Feld's definition of cohomological Chow-Witt groups with coefficients (cf. Definition 7.2 of \cite{feld1}). With our definitions we tautologically recover the formalism of Proposition 7.6 of Feld's work.
        \end{comment}
        Following Feld's Definition 7.2 of \cite{feld1}, for $v\in \underline K(X)$ we set
        \begin{equation*}
          A^i(X,\K_\ast,v):=A_{d-i}(X,\K_\ast,v).
        \end{equation*}
        %\[H^{i,j}(X,\K_\ast,\kc v)=A_{d-i}(X,\K_\ast,(-\Omega^1_X +v -\underline{\rk(v)}+\underline j)=:A^i(X,\K_\ast,(-\Omega^1_X +v -\underline{\rk(v)}+\underline j),\]
        \end{defin}

        %we define cohomological Chow-Witt groups as
        %\[\CHW^i(X,\kc v):=H^{i,i}(X,\KMW_\ast,\kc v)=\CHW^{i}(X,D(\Omega^1_X+v)).\]
        %}
        \subsubsection{}
        \label{subsub:BlochFormula}
        %With the above notation, if $E$ denotes the spectrum corresponding to $\K_\ast$, we have
        %\[H^{i,j}(X,\K_\ast,\kc v)=H_{-i,-j}(X,\K_\ast,\kc{-\Omega^1_X+v})\simeq E^{i+j,j}(X,\kc v).\]
        After Corollary 8.5 of \cite{feld1}, for every Milnor-Witt homotopy module $\K_\ast$  and every smooth variety $X$ we have a canonical isomorphism  
        \[A^i(X,\K_\ast,v)\simeq \op H^i_{Zar}(X,\cal A_X^i(\K_\ast,v)),\]
        where $\cal A_X^i(\K_\ast,v)$ denotes the Zariski sheafification of the presheaf $U\mapsto A^0(U,\K_\ast,v)$. In particular, for $\K_\ast=\KMW_\ast$ we have
        \[H^{i,j}(X,\KMW_\ast,\cal L,a)=\op H^i_{Zar}(X,\cal K^{\rm{MW}}_{j,X}(\cal L,a)),\]
        where $\cal K^{\rm{MW}}_{j,X}(\cal L,a)$ denotes the Zariski sheafification of $U\mapsto H^{0,j}(U,\KMW_\ast,\cal L,a)$.

        \begin{comment}
           $v\sim \underline r$ in the category $\underline K(X)$ gives an isomorphism 
        \[E^{i+j,j}(X,\K_\ast,\kc v)\simeq E^{i+j+2r,j+r}(X)\simeq H^{i+r}_{\rm{Zar}}(X,\bf K_{j+r})\]
        where the right hand side term denotes the Zariski cohomology of the unramified sheaf $\bf K_{j+r}$ associated with the presheaf  $H^{0,j+r}(-,\K_\ast,\kc{0})$.

        \end{comment}

  \subsection{Equivariant Theory} % (fold)
   \label{sub:equivariant_chow-witt_groups}
        %Let $G$ be an affine algebraic group over $k$. In this Section we introduce equivariant quadratic intersection theory for $k$-varieties endowed with a $G$-action.
    \subsubsection{}
     \label{subs:scheme_approx_BG}
      Let $G$ be a smooth affine algebraic group (of finite type) over $k$. Recall (\cite[Remark 1.4]{totaro_chowring}) that given an integer $c$ we can always find a pair $(V,U)$ where $V$ is a representation of $G$, $U$ is an open subset of the vector bundle $\bb V(V)\ra\spec k$ associated to $V$, and such that $(V,U)$ satisfies the following properties:
    \begin{enumerate}
      \item the complement $S=\bb V(V)\smallsetminus U$ has codimension $>c$;
      \item $U$ is $G$-stable and the induced action of $G$ on $U$ is free;
      \item the quotient $U/G$ is a variety over $k$.
    \end{enumerate}
      In this setting $U/G$ is called an \emph{equivariant scheme approximation of $\scr B G$ in codimension $\leq c$}. With a little but common abuse of notation, we will denote with the same name the pair $(V,U)$.
      
      This condition essentially means that if $X$ is a $k$-variety, then every cycle $Z$ on $X$ with $\rm{codim}_X(Z)\leq c$ pulls back to a cycle on $X\times \bb V(V)$ not entirely contained in $X\times S$. (The same observation holds for cycles on $X$ of dimension $\geq \dim(X)-c>\dim(S)+\dim(X)-\rm{rk}(V)$.)
    
    \subsubsection{}
     \label{subs:assumptions_on_X}
      Let $X$ be a finite type variety over $k$ endowed with an action of $G$, and assume $G$ satisfies the avoe assumptions in \ref{subs:scheme_approx_BG}. Suppose that one of the following holds:
      \begin{enumerate}
          \item The reduced scheme $X_{red}$ is quasi-projective and the action of $G$ is linearized with respect to some quasi-projective embedding.
          \item The group $G$ is connected and $X_{red}$ embeds equivariantly as a closed subscheme in a normal variety.
          \item The group $G$ is special.
      \end{enumerate}
      Then for every equivariant scheme approximation $(V,U)$ of $\scr B G$ in codimension $\leq c$, the quotient of $X\times U$ modulo the diagonal action is a scheme (\cite[Proposition 23]{EG}). This quotient is denoted by $X\times_G U$.
      
      Recall that given $X$ and $G$ as above, we can define the stack $[X/G]$ whose $S$-points, for $S$ a scheme, are $G$-torsors $P\to S$ together with a $G$-equivariant map $P\to X$. The stack $[X/G]$ is in general an Artin stack (see \cite[04TK]{stacks-project}) and is usually referred to as \emph{quotient stack}. The scheme $X\times_G U$ is called an \emph{equivariant scheme approximation} of $[X/G]$ in codimension $\leq c$.

    \subsubsection{}
     \label{subs:G_twists}
      Let $X$ be any $G$-variety and let $(V,U)$ be an equivariant scheme approximation of $\scr BG$. Since the action of $G$ on $U$ (resp. on $X\times U$) is free, the quotient map $\pi:U\ra U/G$ (resp. $\pi: X\times U\ra X\times_G U$) is a $G$-torsor.
     In particular, a $G$-line bundle $\cal L$ (resp. a point $e$ of the $G$-equivariant K-theory space) of $X$ determines via pull-back a line bundle (resp. a point of the  K-theory space) on $X\times_GU$, which we will keep denoting by $\cal L$ (resp. $e$). The following commutative diagram should help to give a clear picture of the geometric setup.
     \begin{equation}
     \label{eqn:stack_twisting_diagram}
      \begin{tikzcd}[column sep=10pt,row sep=10pt]
           & & & \bb V(V) \ar[rr]  \ar[dd] & & \spec k  \ar[dd] \\
          X\times U \ar[rr, hook, "i'"] \ar[dd,"\pi'"] & & X\times \bb V(V) \ar[rr,"p'" near end ,crossing over]  \ar[dd, "\pi'"] \ar[ur]& & X \ar[ur] \ar[dd, "\pi" near start] & \\
          & & & {[\bb V(V)/G]} \ar[rr] & & \scr BG \\
          X\times_G U \ar[rr, hook,"i"] & & {[X\times\bb V(V)/G]} \ar[rr,"p"] \ar[ur]& & {[X/G]}\ar[ur] \ar[from=uu,crossing over]& 
          %X\times U \ar[rr, hook] & & X\times \bb V(V) \ar[rr] & & X & \\
      \end{tikzcd}
    \end{equation}
    Note that all squares are cartesian, the higher lever of the diagram is $G$-equivariant, and that one can pass from the higher to the lower level by taking the stack-theoretic quotient modulo $G$.
    
    There is a locally free sheaf on on $X\times_GU$, defined by taking the $0$-th Zariski cohomology sheaf of the relative cotangent complex $L_{p\circ i}$. The reader can refer to Chapter 17 of \cite{lmb} for more details on the relative cotangent complex of a $1$-morphism of stacks. Since $p\circ i$ is smooth and representable by schemes, $L_{p\circ i}$ is concentrated in degree zero (cf. Lemma 17.5.8 of loc. cit.). Since forming the cotangent complex commutes with flat base chance (Theorem 17.3.(4) of loc. cit.) and $L_{X\times U/X}$ is quasi-isomorphic to $\Omega^1_{X\times U/X}$ we deduce that $\pi'^\ast L_{p\circ i}$ is quasi-isomorphic to the pull-back of $V^\vee$ to $X\times U$. This shows that $L_{p\circ i}$ is canonically represented by the locally free sheaf concentrated in degree $0$ induced by the $G$-locally free sheaf $V^\vee$ on $X\times U$.

    \begin{comment}
        There is a naturally defined locally free sheaf on $X\times_GU$, obtained by pulling back the $G$-representation $V$ to $X\times U$. The quotient of this pull-back by $G$ is a locally free sheaf on $U\times _GX$, sheaf that we keep denoting $V$. As a locally free sheaf on $X\times_G U$, we can naturally identify $V$ with the restriction of the sheaf of sections of the relative tangent bundle of $p: [X\times \bb V(V)/G] \ra [X/G]$. 

     In other words we claim that we have obtained a locally free sheaf on $X\times_G U$ which is dual to $\Omega_{X\times_G U/[X/G]}$, where $\Omega_{X\times_G U/[X/G]}$ denotes the $0$-th cohomology sheaf of the relative cotangent complex $L_{p\circ i}$ of the morphism $p\circ i$. We now briefly explain our claim, but first we refer the reader to Chapter 17 of \cite{lmb} for more details on the relative cotangent complex of a $1$-morphism of stacks. Since $p\circ i$ is smooth and representable by schemes, $L_{p\circ i}$ is concentrated in degree zero (cf. Lemma 17.5.8 of \cite{lmb}). Moreover on $X\times U$ we have a chain of canonical isomorphism 
     \[\pi'^\ast \Omega_{X\times_G U/[X/G]} \simeq \pi'^\ast L_{p\circ i}\simeq L_{p'\circ i'}\simeq \Omega^1_{p'\circ i'}\]
     deduced from Theorem 17.3.(4) of \cite{lmb} and Lemma 91.9.1 of the stacks project. One concludes by checking that the above isomorphism is compatible with the descent data.
      \end{comment}
    We refrain from using the heavy notation $L_{X\times_G U/[X/G]}$ in what follows, and abusively stick instead with the symbol $V^\vee$, thinking of it as "the bundle induced by $V^\vee$". We associate to $V^\vee$ its graded determinant $D(V^\vee)\in \cal G(X\times_GU)$ and its class $V^\vee \in \underline K(X\times_GU)$.

      % All our statements are easily verified by comparing the geometry of $[X/G]$ with $G$-equivariant geometry over $X$. 

    \begin{lemma}
      \label{lemma:general_indep_of_low_dim}
      Let $X$ be variety over $k$, let $e\in \underline K(X)$, and let $Z\subseteq X$ be a closed subscheme of dimension $<i-1$. Then for every Milnor-Witt cycle module $\K_\ast$ the restriction map 
      $A_p(X,\K_\ast,e) \ra A_p(X\setminus Z,\K_\ast,e)$
      %$H_{p,j}(X,\K_\ast,\kc{e}) \ra H_{p,j}(X\smallsetminus Z,\K_\ast,\kc{e})$ 
      is an isomorphism in all degrees $p\geq i$ and in all twists $e \in \underline K(X)$.
    \end{lemma}
    \begin{proof}
      It follows directly from the localization sequence (cf. $7.4$ of \cite{feld1}), combined with the fact that the complex $C^F_\bullet(Z,\K_\ast,\kc e)$ is concentrated in degrees $[\dim(Z),0]$ and $\dim(Z)<p-1$. 
    \end{proof}
    \begin{lemma}
      \label{lemma:homotopy_invariance}
      Let $X$ be a variety over $k$ and let $\pi:\bb V(\cal E)\ra X$ be a rank $r$ vector bundle on $X$. Then $\pi$ induces an isomorphism
      %\[\pi^\ast: H_{p,j}(X,\K_\ast,\kc e)\ra H_{p,j}(\bb V(\cal E),\K_\ast,\kc {-\Omega_{\bb V(\cal E)/X}+e}),\]
      \[A_p(X,\K_\ast,e) \ra A_{p+r}(\bb V(\cal E), \K_\ast,-\Omega^1_{\bb V(\cal E)/X}+e )\]
      for all choices of $p$ and $e$. Furthermore we have a canonical isomorphism $\Omega^1_{\bb V(\cal E)/X}\simeq \pi^\ast \cal E^\vee$.
    \end{lemma}
    \begin{proof}
      This is Theorem 9.4 of \cite{feld1}. Alternatively, we can deduce it from the general properties of Borel-Moore homology (cf 2.1.3 of \cite{djk}). The second part of the statement is well known. 
    \end{proof}

    \begin{defin}\label{def:equivariant groups}
      Let $G$ be a smooth affine algebraic group over $k$ of dimension $g$ and let $X$ be a $G$-variety over $k$ satisfying any of the assumptions of \ref{subs:assumptions_on_X}. For every graded line bundle $(\cal L ,a)$ on $[X/G]$ and every pair of integers $i$ and $j$ we set
      \[\CHW^G_i(X,\cal L,a):=\CHW_{i+l-g}(X\times_GU,D(V^\vee)^{-1}\otimes (\cal L,a))\]
      and similarly
      \[H^G_{i,j}(X,\KMW_\ast,\cal L,a):=H_{i+l-g,j+l-g}(X\times_GU,\KMW_\ast, D(V^\vee)^{-1}\otimes(\cal L,a)),\]
      where $(V, U)$ is any scheme approximation of $\scr B G$ in codimension $\leq \dim(X)-i+1$ with $\dim_k(V)=l$; the objects $\cal L$ and $D(V^\vee)$ are those induced on $X\times _GU$ as described in \ref{subs:G_twists}. 

      Analogously for every Milnor-Witt cycle module $\K_\ast$, every $e \in \underline K([X/G])$ and $i\in \bb Z$ we set
      \[A^G_i(X,\K_\ast,e):=A_{i+l-g}(X\times_G U,\K_\ast, -V^\vee +\underline g +e).\]
      \begin{comment}
        Analogously, for every $e \in \underline K([X/G])$ and every pair of integers $i$ and $j$ we set
      \[\CHW^G_i(X,\kc e):=\CHW_{i-g+l}(X\times_GU,\kc{-V^\vee+e}).\]
      More generally, if $\K_\ast$ is a Milnor-Witt cycle module, we set
      \textcolor{blue}{\[H^G_{i,j}(X,\K_\ast,\kc e):=H_{i-g+l,j-g+l}(X\times_GU,\K_\ast,\kc{-V^\vee+e}),\]}
      \[A^G_i(X,\K_\ast,e):=A_{i-g+l}(X\times_G U,\K_\ast, \underline g -V^\vee +e)\]
      \end{comment}
      where $(V, U)$ is any equivariant scheme approximation of $\scr B G$ in codimension $\leq \dim(X)-i+1$.%while $e$ and $V$ are induced as described in \ref{subs:G_twists}.
    \end{defin}
    
    For $X$ smooth, we can also define equivariant \emph{cohomology} groups, for which a duality akin to that of Definition \ref{defin:coh_RS_cpx} holds (cf. \Cref{rmk:equivariant duality} below). Recall that since $G$ is smooth, the stack $[X/G]$ is smooth if and only if $X$ is. Thus when $X$ is smooth and $(V,U)$ is any scheme approximation, $X\times_GU$ is automatically a smooth scheme, since it is open in a vector bundle $[\bb V(V)\times X/G]$ over $[X/G]$. In particular for $X=\spec k$ the approximation $U/G$ is always smooth. For this reason, when we deal with smooth stacks we don't have to worry about providing an explicit smooth equivariant scheme approximation, because the smoothness of the approximation will be automatic.
    
    \begin{defin}\label{def:equivariant groups 2}
    Let $G$ be a smooth affine algebraic group over $k$ of dimension $g$ and let $X$ be a \emph{smooth} and equidimensional $G$-variety over $k$ satisfying any of the assumptions of \ref{subs:assumptions_on_X}. For every $(\cal L,a) \in \cal G([X/G])$ and every pair of integers $i$ and $j$ we set
    \[\CHW_G^i(X,\cal L,a):=\CHW^i(X\times_GU,\cal L,a)\]
     and similarly
    \[H_G^{i,j}(X,\KMW_\ast,\cal L,a):=H^{i,j}(X\times_GU,\KMW_\ast,\cal L,a),\]
     where $(V, U)$ is any scheme approximation of $\scr B G$ in codimension $\leq \dim(X)-i+1$ and $(\cal L,a)$ is induced on $X\times _GU$ as described in \ref{subs:G_twists}. 
      
     Analogously, for every Milnor-Witt cycle module $\K_\ast$, every $e \in \underline K([X/G])$ and every integer $i$ we set
     \[A^i_G(X,\K_\ast, e):=A^i(X\times_G U,\K_\ast,-V^\vee+ e),\]
     where $(V, U)$ is any equivariant scheme approximation of $\scr B G$ in codimension $\leq \dim(X)-i+1$, while $e$ is induced as described in \ref{subs:G_twists}.
    \begin{comment}
       \[\CHW_G^i(X,\kc e):=\CHW^{i}(X\times_GU,\kc{e}).\]
      More generally, if $\K_\ast$ is a Milnor-Witt cycle module, we set
      \textcolor{blue}{\[H_G^{i,j}(X,\K_\ast,\kc e):=H^{i,j}(X\times_GU,\K_\ast,\kc{e}),\]}
      where $(V, U)$ is any equivariant scheme approximation of $\scr B G$ in codimension $\leq \dim(X)-i+1$, while $e$ and $V$ are induced as described in \ref{subs:G_twists}.
    \end{comment}
    \end{defin}
    
    \begin{rmk}\label{rmk:equivariant duality}
    Let $X$ be a smooth $G$-scheme of dimension $d$. In particular $[X/G]$ is a smooth $k$-stack and its cotangent complex $L_{[X/G]/k}$ is a perfect complex. Actually $L_{[X/G]/k}$ is quasi-isomorphic to a two-term complex of $G$-vector bundles on $X$
    \[\Omega^1_{X/k} \ra \fk g^\vee\otimes_k \cal O_X,\]
    where $\fk g$ is the Lie algebra of $G$ with the adjoint action (cf. page 586 of \cite{behrend}). In light of \ref{subs:grading} $L_{[X/G]/k}$ gives a well defined object of $\underline K([X/G])$. 

    Let $(V,U)$ be an equivariant scheme approximation of $\scr B G$ in codimension $\leq i+1$. In view of diagram \eqref{eqn:stack_twisting_diagram}, on $X\times_G U$ we have an exact triangle of perfect complexes 
    \[(p\circ i)^\ast L_{[X/G]/k} \ra L_{X\times_GU/k} \ra L_{X\times_GU/[X/G]},\]
    so that $(p\circ i)^\ast L_{[X/G]/k}$ is naturally the virtual vector bundle $\Omega^1_{X\times_GU/k} - \Omega^1_{X\times_GU/[X/G]}$ on $X\times_G U$. It follows that for every virtual vector bundle $e$ on $[X/G]$ we have the chain of canonical isomorphism:
    \begin{equation*}
      \begin{split}
      \CHW^i_G(X,\cal L,a ) & \simeq \CHW^i(X\times_G U, \cal L,a) \\
      & \simeq \CHW_{d-i+l-g}(X\times_G U, D(-\Omega^1_{X\times U/k}) \otimes (\cal L,a)) \\
      & \simeq \CHW_{d-i+l-g}(X\times_G U, D(-\Omega^1_{X\times_GU/[X/G]}-(p\circ i)^\ast L_{[X/G]/k})\otimes (\cal L,a)) \\
      & \simeq \CHW_{d-i}^G(X, D(-L_{[X/G]/k})\otimes(\cal L,a)).
    \end{split}
    \end{equation*}
    %\[\CHW^i_G(X,\kc e)=\CHW^i(X\times_G U, \kc e)=\CHW_{n-i+l-g}(X\times_G U, \kc {-\Omega^1_{X\times U/k}+e})=\]
    %  \[=\CHW_{n-i+l-g}(X\times_G U, \kc {-\Omega^1_{X\times_GU/[X/G]}-(p\circ i)^\ast L_{[X/G]/k}+e})=\CHW_{n-i}^G(X, \kc{-L_{[X/G]/k}+e}).\]

    This shows that the usual duality between Borel-Moore homology and cohomology is extended to the equivariant setting. The same argument shows that we also have
    \[ H^{i,j}_G(X,\KMW_\ast,\cal L,a)\simeq H_{d-i,d-j}^G(X,\KMW_\ast,D(-L_{[X/G]/k})\otimes(\cal L,a)).\]
    A similar statement holds for the homology of any Milnor-Witt cycle module.
    \end{rmk}

    \subsubsection{} % (fold)
      As for Chow groups, all the Definitions above does not depend on any choice. The proof, which we sketch below, is basically the same of \cite[Proposition 1]{EG}.

% subsection  (end)
    \begin{prop}
      \label{prop:indep_of_choices_in_def_of_equiv_hom}
      Let $\K_\ast$ be a Milnor-Witt cycle module over $k$. The  definition of $A^G_{i}(X,\K_\ast, e)$ does not depend on $(V,U)$, as long as $(V,U)$ is an equivariant scheme approximation of $\scr BG$ in codimension $\leq \dim(X)-i+1$. In particular, under the same assumtions, the same holds for the groups $H^G_{i,j}(X,\KMW_\ast,\cal L,a)$. For $X$ smooth, the same statement holds for the respective cohomology groups.
    \end{prop}

    \begin{proof}
      The statement for $H^G_{i,j}(X,\KMW_\ast,\cal L,a)$ follows from the main statement by choosing a rank zero $e$ with determinant isomorphic to $\cal L$. %Also the statement for $H_G^{i,j}(X,\K_\ast,\kc e)$ follows from the main one, thanks to duality (see \Cref{rmk:equivariant duality}).

      The first point consists in proving that, given a representation $V$, our definition does not depend on the choice of $U$. Assume that $(V,U)$ and $(V,U')$ are both approximations in codimension $\leq \dim(X)-i+1$. Up to further restricting to $(V,U\cap U')$ we can immediately reduce to the case where $U' \subseteq U$. 
      
      Let us also denote by $Z$ and $Z'$ the complements in $V$ of $U$ and $U'$ respectively. We can apply Lemma \ref{lemma:general_indep_of_low_dim} to the decomposition $X\times _GU' \subseteq X\times_GU \supseteq X\times_G(U\cap Z')$. By assumption we have 
      \[\rm{rk}(V)-\dim(Z')> \dim(X)-i+1\] 
      which implies that 
      \[\dim(X\times(U\cap Z'))\leq \dim(X)+\dim(Z')<i-\rm{rk}(-V^\vee+e)-1, \]
      which in turn gives 
      \[\dim(X\times_G(U\cap Z'))<i-g+\rm{rk}(V^\vee)-1.\]
      Therefore, using \ref{lemma:general_indep_of_low_dim}, we can conclude that up to a canonical isomorphism the definition of $A_i^G(X, \K_\ast ,e)$ does not depend od the choice of $U$.
      %that the restriction along $U'\subseteq U$ induces an isomorphism on 
      %\[A_{p}(X,\K_\ast,e)$\]
      %$H^G_{p,j}(X,\K_\ast,\kc e)$ for all $p\geq i$.

      Now we check that \Cref{def:equivariant groups} does not depend on the choice of $V$. If $(V,U_V)$ and $(W,U_W)$ are two approximations in codimension $\leq \dim(X)-i+1$, we can consider the approximation given by the open subscheme $U_{V\oplus W}:=\bb V(V)\times U_W\cap U_V\times\bb V(W)$ of $\bb V(V\oplus W)$, so that $(V\oplus W,U_{V\oplus W})$ in an approximation of $\scr B G$ in codimension $\leq \dim(X)-i+1$.
      
      We then consider the quotient of the $G$-equivariant diagram 
      \[\xymatrix{ X\times U_V & X\times U_V \times\bb(W) \ar[l]_{q'} & X\times U_{V\oplus W} \ar@{_{(}->}[l]_{j_V}\ar@{^{(}->}[r]^{j_W} & X\times \bb V(V)\times U_W \ar[r]^{p'}& X\times U_W}\]
      by the (free) diagonal action of $G$ on each term. Here the external maps $p'$ and $q'$ are vector bundles induced by $p:\bb V(V) \ra \ast$ and $q:\bb V(W) \ra \ast$ under pull-back, while the internal maps are the tautological open embeddings associated with the definition of $U_{V\oplus W}$. 
      
      On $\bb V(V\oplus W)$ we have natural $G$-isomorphisms of locally free sheaves,
      \[\xymatrix{\Omega_{q'} \oplus q'^\ast\Omega_{p} & \Omega_{\bb V(W\oplus V)}\ar[l]_{
      \simeq} \ar[r]^{\tau} &  \Omega_{\bb V(V\oplus W)} \ar[r]^(0.45){\simeq} & \Omega_{p'} \oplus p'^\ast\Omega_{q} }\]
      inducing isomorphisms
      \begin{equation}
            \label{eqn:paths_on_VW}
          -W^\vee-V^\vee+e\simeq -(W+V)^\vee+e \simeq -(V+W)^\vee+e\simeq -V^\vee -W^\vee +e 
      \end{equation}
      in $\underline K(X\times_G U_{V\oplus W})$. 
      
      Let now $l=\rk(V)$ and $m=\rk(W)$. By \ref{lemma:homotopy_invariance} and \ref{lemma:general_indep_of_low_dim} pulling back along $q'\circ j_V$ induces an isomorphism
      \begin{equation*}
          \begin{tikzcd}
              A_{i+l-g} (X\times_G U_V,\K_\ast, -V^\vee + \underline g+e) \arrow[r] & A_{i+l+m-g}(X\times_G U_{V\oplus W}, \K_\ast, -W^\vee-V^\vee+\underline g+e).
          \end{tikzcd}
      \end{equation*}
      %A_{i-g+\rm{rk}(V^\vee)} (X\times_G U_V,\K_\ast, -V^\vee+\underline g+e})\ar[r] & H_{i-g+\rk(V^\vee)+\rk(W^\vee)}(X\times_G U_{V\oplus W}, \K_\ast,\kc{-W^\vee-V^\vee+\underline g+e)}.}\]
      Combining with the isomorphism induced by \eqref{eqn:paths_on_VW} and the inverse of the pull-back map along $p'\circ j_W$ concludes the argument.
    \end{proof}
    
    \subsubsection{}
    \label{subsub:ext_homology_to_eq_setting}
    The preceding proof, although formal, is very relevant for our paper. A similar technique can indeed be used to extend most properties of ordinary Chow-Witt groups of varieties, to the equivariant setting, provided that the property we are interested in be "compatible" with pulling back along a smooth map. This is completely analogous to what happens for Chow groups, both in terms of properties that can be extended to the equivariant setting, and both in the proof technique (see \cite[Section 2.3]{EG}). Here is how things work. Definitions, constructions and properties are first given in terms of a chosen scheme approximation $(V,U)$ where the classical theory with classical constructions can be used. One then has to check the independence of the choice of approximation. If $(V',U')$ is another scheme approximation, the two can be compared by passing to an open of $V\oplus V'$ via the zig-zag
    \[V \leftarrow V\oplus V' \rightarrow V'.\]
    For doing so the only operations that are required are pulling-back along vector bundle projections, and restricting along open subschemes (with complements of high enough codimension). 

    Following this strategy one can extend the cycle theory $H_{\ast_1,\ast_2}(-,\KMW_\ast,\bullet)$ to the equivariant setting of group actions (satisfying the assumptions of \ref{subs:scheme_approx_BG} and \ref{subs:assumptions_on_X}) and equivariant maps, or equivalently to the setting of quotient stack of finite type with representable morphisms. We gather here those properties of the resulting equivariant cycle theory that will be used in the rest of the paper.

    \begin{thm}
    \label{thm:equiv prop_verison2}
    Let $k$ be a perfect field of characteristic not $2$, and let $G$ be a smooth affine algebraic group of finite type over $k$. For every $G$-scheme $X$ of finite type over $k$ satisfying one of the assumptions of \ref{subs:assumptions_on_X}, and every pair of integers $i,j \in \bb Z$ we have a well defined group 
    \[ H_{i,j}^G(X,\KMW_*,\cal L,a),\]
    which is functorial in $(\cal L,a)$ and with the following extra structure.
     \begin{enumerate}
      \item \label{list:pGit_module} The functor $\bigoplus_{i,j} H_{i,j}^G(X,\KMW_*,-):\cal G(X) \ra \scr{AB}$ is a left module over the monoidal functor $\KMW_\ast(k,-): \cal G(k) \ra \scr{AB}$, where $\cal G$ has the monoidal structure introduced in \ref{subs:grading}, and $\scr{AB}$ has the usual monoidal structure.

      \item \label{list:pGit_pf}
      For every proper $G$-equivariant map $f:X\ra Y$ there is an induced push-forward homomorphism 
      $$ f_*:H_{i,j}^G(X,\KMW_*,f^*\Lcal,a)\ra H_{i,j}^G(Y,\KMW_*,\Lcal,a) $$
      which is functorial in $f$ and in $\cal L$, and is compatible with the module structure of \eqref{list:pGit_module}.
      \item \label{list:pGit_pb}
      For every $G$-equivariant map $f:X\ra Y$ that is smooth of relative dimension $d$, or a regular embedding of dimension $d$, there is an induced pullback homomorphism 
      \[f^*:H_{i,j}^G(Y,\KMW_\ast,\Lcal,a)\ra H_{i+d,j+d}^G(X,\KMW_\ast,D(L_{f})^{-1}\otimes f^*(\Lcal,a))\] 
      which is functorial in $f$, in $\cal L$, and is compatible with the module structure. In particular we have pullbacks along any map when $X$ and $Y$ are smooth, by factoring $f$ though its graph.

      \item \label{list:pGit_ls}
      Let $s:Y\hookrightarrow X$ be an equivariant closed embedding and let $r:X\smallsetminus Y\hookrightarrow X$ be the open embedding of the complement of $Y$ in $X$. There is a long exact sequence   
        \begin{equation*} 
          \begin{tikzcd}
            \ra  H_{i,j}^G(Y,\KMW_\ast,s^*\Lcal,a) \rar{s_*} &  H_{i,j}^G(X,\KMW_\ast, \Lcal,a) \rar
            \ar[draw=none]{d}[name=X, anchor=center]{}
            & H_{i,j}^G(X\smallsetminus Y,\KMW_\ast,r^\ast \Lcal,a) \ar[rounded corners,
                to path={ -- ([xshift=2ex]\tikztostart.east)
                    |- (X.center) \tikztonodes
                    -| ([xshift=-2ex]\tikztotarget.west)
                    -- (\tikztotarget)}]{dll}[at end]{} \\      
            \rar{s_*} H_{i-1,j}^G(Y,\KMW_\ast,s^*\Lcal,a) & H_{i-1,j}^G(X,\KMW_\ast,\Lcal,a) \rar & H_{i-1,j}^G(X\smallsetminus Y,\KMW_\ast,r^\ast \Lcal,a) \ra
          \end{tikzcd}
        \end{equation*}
      where the boundary morphisms are compatible with the module structure. The sequence is functorial in $X$ with respect to pulling back along $G$-equivariant maps $f: X' \ra X$ when $f$ is smooth of constant relative dimension, or when $f$ is a regular embedding of constant codimension and is Tor-independent with $r$.

      \item \label{lsit:pGit_hi}
      For every locally free sheaf $\cal E$ of rank $r$ on $X$, the projection $p:\bb V (\cal E) \ra X$ induces an isomorphism
      \[p^\ast: H_{i,j}^G(X,\KMW_*,\Lcal,a)\overset{\simeq}{\ra} H_{i+r,j+r}^G(\bb V(\cal E),\KMW_*,D(\cal E^\vee)^{-1} \otimes (\Lcal,a)).\]
      
      \item \label{list:pGit_ec} 
      For every locally free sheaf $\cal E$ of rank $r$ on $X$ there is a Euler class morphism
      \[e(\cal E):H_{i,j}^G(X,\KMW_*,\Lcal,a)\ra H_{i-r,j-r}^G(X,\KMW_*,D(\cal E^\vee)\otimes (\Lcal,a)), \]
      functorial in $(\cal L,a)$, compatible with pullbacks, satisfying a Whitney sum formula, and compatible with the module structure.
     
      \item \label{list:pGit_sq} 
      For every cartesian square of $G$-schemes
      \begin{equation*}
        \begin{tikzcd}
          X' \arrow[r,"f'"'] \arrow[d,"g'"] & Y' \arrow[d,"g"']\\
          X \arrow[r,"f"] & Y
        \end{tikzcd}
      \end{equation*}
      we have $g^\ast \circ f_\ast=f'_\ast\circ  g'^\ast$ when $f$ is proper and $g$ is smooth, or when $f$ is proper, $g$ is a local complete intersection and $f$ and $g$ are Tor-independent. 
      
      \item \label{list:pGit_ringstr}
      When $X$ is smooth the intersection product makes the functor $\bigoplus_{i,j} H^{i,j}_G(Y,\KMW_\ast,-): \cal G(X) \ra\scr{AB}$ into an algebra over the monoidal functor $\KMW_\ast(k,-): \cal G(k) \ra \scr{AB}$. When $f:X \ra Y$ is a an equivariant map between smooth schemes the map \eqref{list:pGit_pb} respects the intersection product.

      \item \label{list:pGit_projform}
      For every proper $G$-equivariant map $f: X \ra Y$ between smooth $G$-schemes we have equalities
      \[f_\ast(f^\ast(x)y)=xf_\ast(y) \quad f_\ast(x f^\ast(y))=f_\ast(x)y.\]
     \end{enumerate}   
    \end{thm}

    \begin{proof}
        The non-equivariant properties can all be found in Section 2 and 3 of \cite{FasLectures}. Let us show how to prove (2), all the other statements of the theorem can be proved using exactly the same argument, which is taken from \cite[\S 2.3]{EG}.
        
        Pick an equivariant scheme approximation $(V,U)$ of $\cal B G$ in codimension $\leq \max\left\{\dim(X),\dim(Y)\right\}-i$, for some fixed integer $i$. Observe then that $X\times_G U$ and $Y\times_G U$ are both schemes; moreover, the proper morphism $f\times {\rm id}_U$ descends to a proper morphism $f_G:X\times_G U \ra Y\times_G U$ (see \cite[Proposition 2]{EG}). Also the $G$-equivariant line bundle on $Y$ descends to a line bundle $\cal L_G$ on $Y\times_G U$, and it is straightforward to check that $(f^*\cal L)_G = f_G^*(\cal L_G)$. We have then a well defined homomorphism 
        \[(f_G)_*: H_{i+l-g,j+l-g}(X,\KMW_*,f_G^* (D(V^{\vee})^{-1}\otimes (\cal L_G,a)) \ra  H_{i+l-g,j+l-g}(Y,\KMW_*, D(V^{\vee})^{-1}\otimes (\cal L_G,a)) \]
        and by definition the target (resp. the domain) is the equivariant homology group of $Y$ (resp. of $X$) of bidegree $(i,j)$ and twist $(\cal L,a)$ (resp. twist $f^*(\cal L,a)$).
    \end{proof}

    \begin{rmk}
        A version of Theorem \ref{thm:equiv prop_verison2} holds, for the same reasons mentioned in \ref{subsub:ext_homology_to_eq_setting}, for the homology of any Milnor-Witt cycle module with the appropriate shifts and twists. We don't state the theorem here in detail since there would be no gain for the reader. 
    \end{rmk}

        \subsection{Chow-Witt groups of quotient stacks}
        The next Proposition shows that equivariant Chow-Witt groups (and the other equivariant homology/cohomology groups) can be used to define Chow-Witt groups of quotient stacks $\scr X$. If $\scr X$ happens to be represetable by a scheme, the definitions in \ref{sub:chow-witt stack} below, coincide with the classical definitions.
        
    \begin{prop}\label{prop:invariance for quotients}
      Let $G,H$ be smooth affine algebraic groups over $k$, and let $X$ and $Y$ be respectively a $G$-variety and an $H$-variety satisfying any of the assumptions \ref{subs:assumptions_on_X}. Assume that there is an equivalence of stacks $\phi: [X/G]\ra [Y/H]$. Then the induced map $\phi_\ast$ on homology is an isomorphism. 
      
      If the stack $[X/G]$ is smooth or, equivalently, the scheme $X$ is smooth, then the same statement holds for the induced map $\phi^\ast$ on cohomology.
    \end{prop}
    
    \begin{proof}
    Similar to its analogue in Chow theory: cf. for instance Proposition 16 of \cite{EG}.
    \end{proof}
    \subsubsection{}
    \label{sub:chow-witt stack}
      Let $\scr X\simeq [X/G]$ with $G$ satisfying the assumptions of \ref{subs:scheme_approx_BG}, and with $X$ satisfying one of the assumptions of \ref{subs:assumptions_on_X}. For any $e\in \underline K(\scr X)$, any $(\cal L,a) \in \cal G(\scr X)$, and integers $i$, $j$, we define
      %\[ H_{i,j}(\scr X,\KMW_\ast,\cal L,a )  := H^G_{i+g,j+g}(X,\KMW_\ast,\cal L,a) \quad \textrm{ and } \quad  H_{i,j}(\scr X,\KMW_\ast,\kc e)  := H^G_{i+g,j+g}(X,\KMW_\ast,\kc e)\]
      \begin{equation*}
          H_{i,j}(\scr X,\KMW_\ast,\cal L,a ) := H^G_{i+g,j+g}(X,\KMW_\ast,\cal L,a)
      \end{equation*}
      %\begin{equation*}
      %    \begin{split}
      %        H_{i,j}(\scr X,\KMW_\ast,\cal L,a ) & := H^G_{i+g,j+g}(X,\KMW_\ast,\cal L,a)\\
      %        H_{i,j}(\scr X,\KMW_\ast,\kc e) & := H^G_{i+g,j+g}(X,\KMW_\ast,\kc e)
      %    \end{split}
      %\end{equation*}
      By \Cref{prop:invariance for quotients} this Definition does not depend on the presentation of $\scr X$ as a quotient. 
      Further specializing to the case $i=j$, we come to a definition of Chow-Witt groups for quotient stacks:
      \begin{equation*}
          \CHW_i(\scr X,\cal L,a) := \CHW_{i+g}^G(X,\cal L,a)
      \end{equation*}
      %\begin{equation*}
      %    \begin{split}
      %        \CHW_i(\scr X,\cal L,a) & := \CHW_{i+g}^G(X,\cal L,a)\\
      %        \CHW_i(\scr X,\kc e) & := \CHW_{i+g}^G(X,\kc e)
      %    \end{split}
      %\end{equation*}
      For $\scr X$ smooth and equidimensional of dimension $d$, we also obtain a definition of $\KMW_\ast$-cohomology
      \begin{equation}
         H^{i,j}(\scr X,\KMW_\ast,\cal L,a ) := H_G^{i,j}(X,\KMW_\ast,\cal L,a)
      \end{equation}
      %\begin{equation*}
      %    \begin{split}
      %        H^{i,j}(\scr X,\KMW_\ast,\cal L,a ) & := H_G^{i,j}(X,\KMW_\ast,\cal L,a)\\
      %        H^{i,j}(\scr X,\KMW_\ast,\kc e) & := H_G^{i,j}(X,\KMW_\ast,\kc e)
      %    \end{split}
      %\end{equation*}
      and of cohomological Chow-Witt groups
      \begin{equation}
        \CHW^i(\scr X,\cal L,a) := \CHW^{i}_G(X,\cal L,a)  
      \end{equation}
      %\begin{equation*}
      %    \begin{split}
      %        \CHW^i(\scr X,\cal L,a) & := \CHW^{i}_G(X,\cal L,a) \\
      %        \CHW^i(\scr X,\kc e) & := \CHW^{g}_G(X,\kc e)
      %    \end{split}
      %\end{equation*}
      Similarly for a Milnor-Witt cycle module $\K_\ast$ one defines 
      \[A_i(\scr X,\K_\ast, e):=A^G_{i+g}(X, \K_\ast,-\underline g +e ),\]
      and for a smooth and equidimensional $\scr X$
      \[A^i(\scr X, \K_\ast,e):=A_G^i(X,\K_\ast, +e).\]
      %\[ \CHW^i(\scr X,\cal L,a) := \CHW^i_G( X,\cal L,a).\]
      For $v \in \underline K(\scr X)$ we can compare the homological and the cohomological Chow-Witt groups as follows:
      \[ \CHW^i(\scr X,\cal L,a ) \simeq \CHW_{d-i}(\scr X,D(-L_{\scr X/k}) \otimes (\cal L,a) ).\]
      %For $\scr X$ a scheme, the object we define coincide with the classic ones. \textcolor{red}{smoothness of approx vs smoothness of the quotient}

      %Finally, the properties of Chow-Witt group sketched in Subsection \ref{sec:properties} hold for quotient stacks as well (just as in \cite{EG}).

    \begin{thm}
    \label{thm:equiv prop_verison2_stacks}
      All the statements of Theorem \ref{thm:equiv prop_verison2} hold unchanged if we replace $G$-varieties of finite type over $k$ by the associated quotient stacks of finite type over $k$, equivariant graded line bundles by graded line bundles on quotient stacks, and equivariant maps by representable morphisms.
    \end{thm}
    \begin{proof}
    If $f:\scr X \ra \scr Y\simeq [Y/G]$ is representable, we have that the pullback of the $G$-torsor $Y\ra [Y/G]$ along $f$ is a scheme $X$ together with a map $X\ra\scr X$ that makes it into a $G$-torsor over $\scr X$. In other terms, we have that $\scr X\simeq [X/G]$ and $f$ is induced by a $G$-equivariant morphism $X\ra Y$. Using this trick one reduces to the equivariant setting \ref{thm:equiv prop_verison2} both the construction of the functorialities and the checking of their properties.
    \end{proof}

\section{Some useful facts on Chow-Witt groups}
  \label{sec:preliminaries}
    We collect some basic results on Chow-Witt groups and related theories that we will use in the next sections. %For the most basic properties of Chow-Witt groups, we direct the reader to the brief account contained in \ref{sec:properties} or to the more detailed treatment \cite[Section 2]{FasLectures}.

    In this Section and in the remainder of the paper, every time we make a statement concerning Chow-Witt groups of quotient stacks, we always assume that the quotient stack has a presentation of the form $[X/G]$ where the pair $(X,G)$ satisfies at least one of the assumptions in \ref{subs:assumptions_on_X}.

  \subsection{Some Milnor-Witt cycle modules}
    In addition to the Milnor-Witt cycle module $\KMW_\ast$ we have the so called "powers of the fundamental ideal" $\I^\ast$. Here $\I^r(E)\subseteq \W(E)$ denotes the $r$-th power of the fundamental ideal of the Witt ring of the field $E$, with the convention that negative powers be defined as $\W(E)$. The elements $u\in E^\times$ act on $\W(E)$ via multiplication by the rank-one form $\kc u$, and they act on any $E$-vector space via multiplication by $u$. For every fintely generated field extension $E/k$, the Milnor-Witt cycle module $\I^\ast: \underline K(E) \ra \scr{AB}$ assigns 
    \[v \mapsto \I^{\rk(v)}(E)\otimes _{\bb Z[E\setminus \{0\}]} \bb Z[\det(v)\setminus\{0\}].\]
    One can turn the above assignment into a Milnor-Witt cycle module by following the steps of Theorem 4.13 of \cite{feld1}, or alternatively identifying the above assignment with the image of $\KMW_\ast$ in the cycle module $\KMW_\ast[\eta^{-1}]$ (cf. Remark 3.12 of \cite{morel:A1at} and \cite{MR2099118}). 

    Further examples of Milnor-Witt cycle modules are those of Milnor $K$-theory $\KM_\ast$, and of Milnor $K$-theory modulo $2$, denoted by $\Km_\ast$. These are actually cycle modules according to Rost's definition (Section 2 of \cite{rost}), and in particular they are Milnor-Witt cycle modules (cf. section 12 of \cite{feld1}). 

    We have canonical maps
    \begin{equation*}
    \begin{tikzcd}
    \I^\ast \arrow[r, "q"] & \Km_\ast & \KM_\ast \arrow[l,"p"']
    \end{tikzcd}
    \end{equation*}
    where $q$ is induced by mapping the Pfister form $\kc{\kc{u}}$ to the symbol $\{u\}$, and $p$ is simply the reduction modulo two. We also have canonical maps
    \begin{equation*}
    \begin{tikzcd}
    \I^\ast & \KMW_\ast \arrow[l, "p'"'] \arrow[r,"q'"] & \KM_\ast 
    \end{tikzcd}
    \end{equation*}
    where $p'$ is induced by mapping $[u] \mapsto -\kc{\kc{u}}$, and $q'$ by $\eta\mapsto 0$. 

    To the square of Milnor-Witt cycle modules
    \begin{equation}
    \label{eq:diag morel}
    \begin{tikzcd}
      \KMW_\ast\arrow[d, "p'"'] \arrow[r,"q'"] & \KM_\ast \ar[d,"p"]\\
      \I^\ast \arrow[r,"q"''] & \Km_\ast 
    \end{tikzcd}
    \end{equation}
     we associate the corresponding square of their cycle complexes associated to a variety $X$ of finite type. For convenience of notation we set
     \[C_\bullet(X,\K_\ast,j,\cal L,a) \]
     to be the complex defined analogously to \eqref{eqn:KMW-cyccle_complex} for $\K_\ast \in \{\I^\ast,\KM_\ast, \Km_\ast\}$. We are going to use bigraded and twisted homology of these complexes, as we did for $\KMW_\ast$, and exactly with the same notation as have introduced in \ref{defin:CW} and in \ref{defin:coh_RS_cpx}. We don't repeat all these definitions here. It is obvious but important to remark that all the results of Section 1, and in particular Theorems \ref{thm:equiv prop_verison2} and \ref{thm:equiv prop_verison2_stacks}, hold for $\K_\ast \in \{\I^\ast,\KM_\ast, \Km_\ast\}$ even with this new notation. By design the twising don not matter on Rost cycle modules, so we directly omit them form notation regarding $\KM_\ast$ and $\Km_\ast$.

     In a sense this definition is redundant, since we could use the notation $A_i(-)$ and $A^i(-)$ for our purposes. However with the aim of performing some concrete computations, we find it easier to stick to the twisting by graded line bundles rather then by virtual vector bundles.

  \subsubsection{}
    \label{sub:shopping_list_of_cycle_modules}
    In Theorem 5.3 of \cite{MR2099118} Morel proves that the square \eqref{eq:diag morel}
    is cartesian. Observe that all the arrows in this square are surjective. In particular there is an induced homotopy cartesian square of the associated Rost-Schmid complexes. Indeed for the projective model structure on bounded above chain complexes (introduced in Chapter 2, pages 4.11-12 of \cite{quillen:ha}), degree-wise surjections are fibrations. If $X$ is a $k$-variety of finite type we have thus an exact triangle of Rost-Schmid complexes
    \[C_\bullet(X,\KMW_\ast,j,\cal L,a) \ra C_\bullet(X,\I^\ast,j,\cal L,a)\oplus C_\bullet(X,\KM_\ast,j) \ra C_\bullet(X,\Km_\ast,j)\]

    A bit of diagram chasing, together with the functoriality of cycle complexes, easily gives the following result.
    
    \begin{prop}\label{prop:fundamental diagram}
      Let $X$ be either a scheme or a quotient stack of finite type over $k$. The square \eqref{eq:diag morel} induces a canonical exact sequence
      \begin{equation}
      \label{eqn:fundseq}
      \begin{tikzcd}
        0 \arrow[r] & K_i(X,\cal L,a)\arrow[r] & \CHW_i(X,\cal L,a) \arrow[r,"{(p',q')}"] & P_i(X,\cal L,a)\arrow[r] & 0
      \end{tikzcd}
      \end{equation}
      where $P_i(X,\cal L)$ is the pull-back of 
      \begin{equation}
        \begin{tikzcd}
              H_{i,i}(X,\I^\ast,\cal L,a)\arrow[r,"q"] & \CH_i(X)\otimes\bb Z/2 & \CH_i(X),\arrow[l,"p"']
        \end{tikzcd}
      \end{equation}
      %\[\xymatrix{H_{i,i}(X,\I^\ast,\cal L,a)\ar[r]^q & \CH_i(X)\otimes\bb Z/2 & \CH_i(X),\ar[l]_{p}}\]
      and 
      \[K_i(X,\Lcal,a)\simeq\rm{coker}\Big( H_{i+1,i}(X,\I^\ast,\Lcal,a) \oplus H_{i+1,i}(X,\KM_\ast) \overset{p+q}{\ra}  H_{i+1,i}(X,\Km_\ast)\Big ). \]
       Alternatively we can canonically describe the term $K_i(X,\cal L,a)$ as a quotient of $\CH_i(X)[2]$.
      \begin{comment}
        Alternatively we can canonically describe the term $K_i(X,\cal L,a)$ as 
      \[K_i(X,\Lcal,a)\simeq\CH_i(X)[2]/\partial H_{i+1,i}(X,\I^\ast,\Lcal,a)\cap\CH_i(X)[2],\]
      where the boundary morphism $\partial$ is the one induced by the short exact sequence of Milnor-Witt cycle modules
      \[ 0\ra \I^\ast \ra \KMW_\ast \ra \KM_\ast \ra 0. \]
      \end{comment}       
      The sequence \eqref{eqn:fundseq} is functorial for pull-backs along regular embeddings and along smooth maps, and for push-forwards along proper maps. When $X$ is smooth the map $(p',q')$ at the right hand side of \eqref{eqn:fundseq} respects the intersection product.
        
    \end{prop}
    \begin{proof}When $X$ is a variety the proof proceeds as in Proposition 2.11 of \cite{HW}, using the above exact triangle. The functoriality of \eqref{eqn:fundseq} is a direct consequence of the functoriality of Rost-Schmid complexes and the compatibility of the maps in \eqref{eq:diag morel} with pull-backs and push-forwards. The compatibility with  intersection products in the smooth case follows from the fact that the maps in \eqref{eq:diag morel} are ring homomorphisms. We only need to prove the statement for quotient stacks $\scr X=[X/G]$. For this, choose an equivariant scheme approximation $X\times_G U$ of $[X/G]$: then the square \eqref{eq:diag morel} induces a short exact sequence for $X\times_G U$, and the following statements hold as well. As all the groups involved are defined via cohomology of Milnor and Milnor-Witt cycle modules, they are also well defined for $[X/G]$ via equivariant scheme approximation, thanks to \Cref{prop:invariance for quotients} and thanks to the pull-back functoriality of \eqref{eqn:fundseq} for smooth maps on varieties. We can then just define the maps at the level of equivariant scheme approximations, and the statement for quotient stacks follows.
        \end{proof}
    In particular, when the group $K_i(X,\cal L,a)$ is zero, the Chow-Witt group $\CHW_i(X,\cal L,a)$ can be regarded as the subgroup of the product
    \[ \CH_i(X) \times H_{i,i}(X,\I^{\ast},\cal L,a), \]
    formed by those pairs whose entries have the same image in $\CH_i(X)\otimes \bb Z/2$. This description can be quite useful: for instance, most of the computations in \cite{HW} are based on this fact.

    \subsection{On Chow-Witt groups of complements of zero sections }
    \subsubsection{}
    Recall from \cite[Section 2.5]{FasLectures} that when $X$ is an algebraic variety and $\cal E$ is a rank $r$ vector bundle on $X$, one can define the so called \emph{Euler homomorphism}, denoted by
    \[e(\cal E)\cap-: H_{i,j}(X,\KMW_\ast,\cal L,a) \ra H_{i-r,j-r}(X,\KMW_\ast, D(\cal E^\vee) \otimes \cal L,a).\]
    When the pullback along $\pi:\bb V(\cal E)\ra X$ is an isomorphism, the Euler homomorphism coincides with the composition of the push-forward along the zero section $s_0: X \ra \bb V(\cal E)$ with the inverse of $\pi^*$. When $X$ is smooth, $e(\cal E)\cap 1 \in \CHW^r(X,D(\cal E^\vee))$ is called \emph{Euler class} of $\cal E$. By combining the projection formula, homotopy invariance, and the fact that pulling-back is a ring homomorphism, one sees that $e(\cal E)\cap-$ actually amounts to taking the intersection product with the Euler class $e(\cal E)$. 

    If $ {\K}_\ast$ is any generalized cycle module, there is a similarly defined operation $e^{\K_\ast}(\cal E)\cap-$ in $ {\rm K}_\ast$-homology. If $X$ is smooth and ${\rm K}_\ast$ is a commutative monoid in Milnor-Witt cycle modules (with respect to the monoidal structure induced on Milnor-Witt cycle modules through the equivalence of Section 4 of \cite{feld2}), as it happens for all the examples of \ref{sub:shopping_list_of_cycle_modules}, the image of the $\KMW_\ast$-theoretic Euler class coincides with the $ \K_\ast$-theoretic Euler class. For instance, in the case $ {\rm K}_\ast\in \{ \KM_\ast, \Km_\ast\}$, the $\K_\ast$-theoretic Euler class is the usual top Chern class. The same observations hold for quotient stacks.
    
    \begin{comment}
       \subsubsection{}
    All these definitions can be further extended to quotient stacks in the usual way: if $\pi:\scr V\ra\scr X=[X/G]$ is a vector bundle over a quotient stack, then we have $\scr V=[E/G]$, where $E$ is a $G$-equivariant vector bundle over $X$, and if $(V,U)$ is an equivariant scheme approximation of $\scr B G$, then $E\times_G U$ is an equivariant scheme approximation of $\scr V$ as well as a vector bundle over $X\times_G U$ (see the proof of \Cref{thm:equiv prop_verison2_stacks}).
    
    We can then define the Euler class of $\mathscr{V}$ as the mapping induced by the Euler class of $E\times_G U$, regarded as a vector bundle over $X\times_G U$. Similarly, if $\scr X$ is smooth (which by \Cref{rmk:smooth} implies that $X\times_G U$ is smooth), we can also define the Euler class of $\scr V$.
    
    Let $s_0:\scr X\to \scr V$ be the zero-section, and suppose that $\scr X$ is smooth. Then the induced homomorphism
    \[ H^{i,j}(\scr X, \K_\ast, \Lcal,a) \overset{s_{0*}}\ra H^{i+r,j+r}(\scr V, \K_\ast, \pi^*\Lcal,a) \overset{(\pi^*)^{-1}}\ra H^{i+r,j+r}(\scr V, \K_\ast, \pi^*\Lcal,a) \]
    coincides with the mapping induced by the Euler class of $\scr V$. The proof of this fact for smooth quotient stacks is again based on picking equivariant scheme approximations of the objects involved.
    \end{comment}

    \subsubsection{}
    \label{sub:can_loc_sequence_on_a_bundle}
    Assume now that $\cal E$ is a vector bundle of rank $r$ on a smooth variety or a smooth quotient stack $X$, and that $\op K_\ast \in \{\KMW_\ast,\I^\ast, \KM_\ast,\Km_\ast\}$. By combining the projection formula with homotopy invariance, the localization sequence associated to the decomposition $X \overset{s_0}{\hra} \bb V(\cal E)\overset{j}{\hla} \bb V(\cal E)\smallsetminus s_0(X)$ is identified with 
    \[\xymatrix@C=2em{\cdots \ar[r] & H^{i-r,j-r}(X,\K_\ast,D(\cal E)\otimes (\cal L,a)) \ar[r]^(.6){e(\cal E)\cdot} & H^{i,j}(X,\K_\ast,\cal L,a) \ar[r]^(.4){(j\circ p)^\ast} & H^{i,j}(\bb V(\cal E)\smallsetminus s_0(X),\K_\ast,\cal L,a) \ar[r]^(0.8){\partial} & \cdots}\]
    The map $(j\circ p)^\ast$, considered on the whole bi-graded and twisted cohomology ring, is a ring map with respect to the intersection product, and from the exactness of the above sequence we conclude that $\ker((j\circ p)^\ast)$ is the principal graded ideal generated by $e(\cal E)$. In particular we have a canonical extension
    \[\xymatrix@C=2em{0 \ar[r] & H^{\ast_1,\ast_2}(X, \KMW_\ast,\bullet)/(e(\cal E)) \ar[r]^(0.47){(j\circ p)^\ast}&  H^{\ast_1,\ast_2}(\bb V(\cal E)\smallsetminus s_0(X), \KMW_\ast,\bullet) \ar[r]^(.65){\partial} & \ker(e(\cal E)\cdot) \ar[r] & 0. }\]
    \begin{comment}
      If we take $\K_\ast=\KMW_*$, the sequence above reads as
    \[\xymatrix@C=1.6em{\cdots \ar[r] & H^{i-r}(X,\cal L\otimes \det(\cal E),\KMW_{j-r}) \ar[r]^(.6){e(\cal E)\cdot} & H^{i}(X,\cal L,\KMW_j) \ar[r]^(.4){(j\circ p)^\ast} & H^{i}(\bb V(\cal E)\smallsetminus s_0(X),\cal L,\KMW_j) \ar[r]^(0.8){\partial} & \cdots}.\]
    Finally, if we restrict to the Chow range $i=j$, we get
    \[\xymatrix@C=0.9em{
    \CHW^{i-r}(X,\cal L\otimes \det(\cal E)) \ar[r]^(.63){e\cdot} & \CHW^i(X,\cal L) \ar[r] & \CHW^i(\bb V(\cal E)\smallsetminus s_0(X),\cal L) \ar[r]^(0.47){\partial} &  H^i(X, \cal L\otimes \det(\cal E), \KMW_{i-1})\ar[r] & \cdots
    }\]
    \end{comment}
    If we take $\K_\ast=\KMW_*$ and $i=j$, the sequence above reads as
    \[\xymatrix@C=0.9em{
    \CHW^{i-r}(X,D(\cal E)\otimes (\cal L,a)) \ar[r]^(.63){e\cdot} & \CHW^i(X,\cal L,a) \ar[r] & \CHW^i(\bb V(\cal E)\smallsetminus s_0(X),\cal L,a) \ar[r]^(0.47){\partial} &  H^{i,i-1}(X,\KMW_\ast, \cal D(\cal E)\otimes(\cal L,a))\ar[r] & \cdots
    }\]
    
\subsection{A useful Lemma}
\subsubsection{} 
    The following result can be used to produce a natural class in the Chow-Witt group of codimension $0$ cycles.
    
    \begin{lemma}\label{lemma:splitting CW0}
    Let $X$ be either a smooth scheme or a smooth quotient stack, endowed with a line bundle $\Lcal$, satisfying the assumptions of \ref{subs:assumptions_on_X}. Let $s$ be a global section of $\Lcal^{\otimes 2}$ whose vanishing locus $D\subset X$ is smooth and non-empty. Set $U:=X\smallsetminus D$. Then:
    \begin{enumerate}
        \item the non-degenerate quadratic form on $\Lcal_U$ defined by
    \[q:\Lcal_{|U}\otimes\Lcal_{|U}\ra\cal O_U,\quad a\otimes b\longmapsto (a \otimes b)/s \] 
    determines an element $q_{gen}$ in $\CHW^0(X\smallsetminus D)$.
    
    \item We have $h(q_{gen}-1)=0$ and $\partial(q_{gen}-1)=\eta\otimes \overline{f}^{\vee}$, where
    \[ \partial:\CHW^0(X\smallsetminus D)\ra H^0(D,\KMW_{-1},\cal O_D(D))\]
    is the boundary map induced by the closed embedding $D\hookrightarrow X$, and $f$ is a local equation of $D$.
    \end{enumerate} 
    \end{lemma}
    
    \begin{proof}
    The version of the Lemma for smooth quotient stacks follows from the one for schemes, by reducing the statement to a statement on the equivariant scheme approximation of the stack.

    Let us sketch the argument a bit more in detail: suppose that the Lemma is true in the case of schemes and let $\scr X=[X/G]$ be a quotient stack satisfying the hypotheses of the Lemma. In particular, if $\scr D$ is a divisor which is the vanishing locus of a section of $\Lcal$, we have $\scr D=[D/G]$ for some $G$-invariant divisor $D\subset X$, obtained as the vanishing locus of a section of an equivariant line bundle $L$. 
    
    We can pick an equivariant scheme approximation $(V,W)$ of $\scr B G$ such that the codimension of the complement of $W$ in $V$ is greater than $2$. Therefore, we have by definition that \[\CHW^0(\scr X\smallsetminus \scr D)=\CHW^0_G(X\smallsetminus D)=\CHW^0((X\smallsetminus D)\times_G W).\]
    The smooth scheme $(X\smallsetminus D)\times_G W$ together with the line bundle $\cal L\times_G W$ (whose vanishing locus is precisely $D\times_G W$) satisfy the hypotheses of the Lemma, so we have a well defined class in $\CHW^0((X\smallsetminus D)\times_G W)$ with all the desired properties, hence a class in $\CHW^0(\scr X\smallsetminus \scr D)$.
    
    Let us then assume that $X$ is a scheme, and pick an affine open subscheme $V\subset X$ whose intersection with $D$ is non-empty and such that there exists a trivialization
    \[\cal O_V \simeq \Lcal_V,\quad 1 \longmapsto t. \]
    Then by construction we have $s=ft^{\otimes2}$, where $f$ is an equation for $D\cap V$. Moreover, after trivializing $\Lcal_V$ in this way, the quadratic form $q$ that we defined is given by
    \[ \cal O_{U\cap V}\otimes \cal O_{U\cap V} \longrightarrow \cal O_{U\cap V},\quad g\otimes g'\longmapsto \frac{gg'}{f} .  \]
    The quadratic form $\langle f^{-1} \rangle$, regarded as an element of $\GW(k(X))\simeq \KMW_0(k(X))$, has residue $0$ along all the codimension $1$ points of $X\smallsetminus D$. In other terms, $\langle f^{-1}\rangle$ determines an element of $\CHW^0(X\smallsetminus D)$. 
    
    Observe that this element does not depend on the choice of an affine open subscheme $V$ and of a trivialization $\cal O_{V}\simeq\Lcal_{V}$: indeed, if we picked a different subcheme and a different trivialization, the resulting quadratic form would differ from the previous one by a square and hence it would determine the same element in $\GW(k(X))$.
    
    We have $h(q_{gen}-1)=h\cdot\eta[f^{-1}]=0$. The image of $q_{gen}-1=\eta[f^{-1}]$ along the boundary morphism
    \[ \CHW^0(X\smallsetminus D)\ra H^0(D,\KMW_{-1},\cal O_D(D)) \]
    is equal by construction to the residue of this element at the point $\spec{k(D)}$ of codimension $1$ (\cite[Subsection 2.2]{FasLectures}). Observe that $f$ is a local parameter for the valuation $\nu:k(X)\ra \ZZ\cup\{\infty\}$ induced by the Cartier divisor $D\subset X$. Applying the formula for residues (\cite[Subsection 2.1]{FasLectures}) we obtain
    \[ \partial_\nu(\eta[f^{-1}])=\partial_\nu^{f}(\eta[f^{-1}])\otimes \overline{f}^{\vee}=\eta\cdot\partial_\nu [f^{-1}] \otimes \overline{f}^{\vee}= \eta\otimes \overline{f}^{\vee}. \]
\end{proof}
\begin{rmk}
      Under the isomorphisms
      \[H^0(D,\KMW_{-1},\cal O_D(D)) {\ra} H^0(D,\W,\cal O_D(D)) \overset{\phi_l}{\ra} H^0(D,\W,\cal O))\]
      the element $\eta\otimes \bar f^\vee$ is mapped first to $1\otimes \bar f^\vee$ and then to $\langle \lambda_l \rangle\otimes 1$. Here $l$ is a generic section of $
      \cal L|_{D}$, and $\phi_l$ is the isomorphism induced by $l$ and the identification $\cal O_D(D)\simeq \cal L|_{D}$.
\end{rmk}

\section{Geometric Preliminaries} % (fold)
\label{sec:stacks_of_elliptic_curves}
In this Section we first recall some basic facts on the geometry of $\scr B\Gm$ and $\scr B\mu_{n}$, and in particular how to present $\scr B\mu_{n}$ as a $\Gm$-quotient stack. Next we recall some basic facts on the stack $\Mcal_{1,1}$ of elliptic curves and on its compactification $\Mbar_{1,1}$ by means of stable curves. For sake of simplicity we assume to work over a base field $k$ where $n$ is invertible; all schemes are understood as schemes over $k$.

\subsection{Classifying stacks of $\Gm$ and $\mu_{2n}$}
We recall here some basic facts on the classifying stacks of $\Gm$ and $\mu_{2n}$; these already appeared in the literature, see for instance the beginning of Section 7 of \cite{Bro} or Section 2.4 of \cite{TotBook}.
    \subsubsection{}
    \label{sub:GmReps_and_approx}
    Let $\bb G_m$ be the multiplicative group and let $\mu_n$ be the subgroup of $n$-th roots of unity. Along this paper we will deal only with actions of these two groups. We denote by $V_n$ the rank one representation of $\Gm$ of weight $n$, i.e. the representation on which $\Gm$ acts by $ \lambda\cdot v:=\lambda^nv$. We will then denote by $\bb V_n$ the associated $\Gm$-equivariant vector bundle over $\spec(k)$. Similarly we set $V_{m,n}:=V_{m}\oplus V_{n}$, and $\bb V_{m,n}$ will denote the associated $\Gm$-vector bundle.

    The diagonal action of $\bb G_m$ on $\bb V_1^n$ has the origin as the only fixed point. Set $U_n:=\bb V_1^n\smallsetminus\{0\}\subset \bb V_1^n$: then the quotient $U_n/\Gm\simeq \bb P^{n-1}$ is an equivariant scheme approximation of $B\bb G_m$ in codimension $\leq n-1$. It can thus be used to approximate Chow-witt classes of codimension $\leq n-2$ on $\scr B \bb G_m$, and actually on any quotient $[X/\bb G_m]$ of a $\bb G_m$-variety $X$ satisfying \ref{subs:assumptions_on_X}. More precisely, given a $\bb G_m$-equivariant diagram $X_\bullet$ of $k$-varieties satisfying \ref{subs:assumptions_on_X}, the diagram $U_n\times_{\bb G_m} X_\bullet$ is an approximation of the corresponding diagram of stacks $[X_\bullet/\bb G_m]$: it can be thus used to perform any sensible operations on cycle groups that only involves cycles of codimension smaller than or equal to $n-2$. 
    
    \subsubsection{}\label{sub: O(-1)}
    We denote by $\cal U$ the universal line bundle on $\scr B\Gm$: its total space is $[\bb V_{-1}/\Gm]$ and $\bb V_{\bb P^N}(\cal O(-1))$ is an equivariant scheme approximations in codimension $\leq N$. In fact, the total space of the tautological line bundle on $\bb P^n$ is by definition the subscheme in $\bb P^n \times \bb A^{N+1}$ formed by the pairs $([X_0:\ldots :X_N],(Y_0,\ldots, Y_N))$ such that $(Y_0,\ldots, Y_N) = \lambda(X_0,\ldots ,X_N)$ for some scalar $\lambda$. If we look at the pullback of this line bundle along the $\Gm$-torsor $\bb A^{N+1}\smallsetminus\{0\}\ra \bb P^N$, we obtain a subscheme $L\subset \bb A^{N+1}\smallsetminus\{0\} \times \bb A^{N+1}$. There is an isomorphism $L\simeq \bb A^{N+1}\smallsetminus\{0\}\times \bb A^1$ given by sending a point $((X_0:\ldots :X_N),(Y_0,\ldots, Y_N))$ in $L$ to the point $((X_0:\ldots :X_N),\lambda)$, where $\lambda$ is the unique scalar such that $(Y_0,\ldots, Y_N) = \lambda(X_0,\ldots ,X_N)$. Observe that if we act on $\bb A^{N+1}\smallsetminus\{0\} \times \bb A^{N+1}$ with $\Gm$ by multiplication $(X_0,\ldots,X_{N})\longmapsto t(X_0,\ldots,X_N)$ on the first factor and trivially on the second factor, the subscheme $L\subset \bb A^{N+1}\smallsetminus\{0\}\times \bb A^{N+1}$ is $\Gm$-invariant, and if we use the isomorphism $L\simeq \bb A^{N+1}\smallsetminus\{0\}\times \bb A^1$ defined above, the action on $L$ is $(X_0,\ldots, X_N,\lambda)\longmapsto (tX_0,\ldots, tX_n, t^{-1}\lambda)$. In other terms, the pullback of $\bb V(\cal O_{\bb P^N}(-1))$ along $\bb A^{N+1}\smallsetminus\{0\}\ra \bb P^N$ is isomorphic to $\bb A^{N+1}\smallsetminus\{0\}\times \bb V_{-1}$, as claimed.
    
    \subsubsection{}
    The following observations on $\mu_n$-covers are explained in higher generality and more detail in \cite[Section 2]{AV}, but we recall here the main points for the convenience of the reader. 
    
    The objects of the classifying stack $\scr B \mu_n$ are by definition $\mu_n$-torsors $X\to S$. Let $\scr F_n$ denote the stack over the \'{e}tale site of schemes whose objects are triples $(S,\cal L,\varphi:\cal L^{\otimes n}\overset{\simeq}{\rightarrow}\cal O_S)$, where $S$ is a scheme, $\cal L$ is a line bundle over $S$, and $\varphi$ a $\cal O_S$-linear isomorphism; the morphisms between triples $(S',\cal L',\varphi')\to (S,\cal L,\varphi)$ are maps $S'\overset{f}{\to} S$ together with an isomorphism $\psi:f^*\cal L\simeq \cal L'$ such that $f^*\varphi=\varphi'\circ \psi$.
    
    There is a canonical isomorphism $\scr B\mu_n\simeq \scr F_n$, that is the stack of $\mu_n$-torsors is equivalent to the stack of line bundles together with a trivialization of their $n^{\rm th}$-tensor power. To see this, let us first define a morphism $\scr F_n \ra \scr B\mu_n $ by sending an object $(S,\cal L, \varphi)$ to the $\mu_n$-torsor over $S$ defined as follows: set
    \[ \cal A:= \cal O_S\oplus \cal L \oplus \cal L^{\otimes 2}\oplus \cdots \oplus \cal L^{\otimes (n-1)}. \]
    We can use the homomorphism $\varphi:\cal L^{\otimes n}\to\cal O_S$ to give a multiplicative structure to $\cal A$. Indeed, given $i,j<n$, write $i+j=mn+r$, with $r<n$, and consider the map
    \[ \varphi^{\otimes m}\otimes {\rm{id}}_{\cal L^{\otimes r}}:\cal L^{\otimes (i+j)}\simeq \cal L^{\otimes mn}\otimes \cal L^{\otimes r}\longrightarrow\cal O_S^{\otimes m}\otimes \cal L^{\otimes r}\simeq \cal L^{\otimes r}. \]
    Then for $s_i$ a section of $\cal L^{\otimes i}$ and $s_j$ a section of $\cal L^{\otimes j}$, we set 
    \[s_i\cdot s_j:=(\varphi^{\otimes mn}\otimes {\rm id}_{\cal L^{\otimes r}})(s_i\otimes s_j).\]
    Together with the usual sum, this gives $\cal A$ the structure of a sheaf of $\cal O_S$-algebras, so we can define
    \[ X_{(S,\cal L,\varphi)} :=\spec_{\cal O_S}\left(\cal A\right).\]
    There is an action of $\mu_n$ on $X_{(S,\cal L,\varphi)}$ induced by the coaction $\cal A\to\cal A[x]/(x^n-1)$ that sends a section $s_i$ of the factor $\cal L^{\otimes i}$ to $s_ix^i$.  The fact that this action is freely transitive is equivalent to $X_{(S,\cal L,\varphi)}\to S$ being \'{e}tale: this can be checked fiber by fiber, and the fiber over a point $s$ of $S$ is equal to $k(s)[t]/(t^n-f)$, here $f$ being the restriction of the image of $1$ along $\varphi^{\vee}:\cal O_S\ra\cal L^{\otimes (-n)}$. As by construction $f$ never vanishes, we deduce the \'{e}taleness of the morphism.
    
    The morphism $\scr F_n \ra \scr B\mu_n $ given by $(S,\cal L,\varphi)\longmapsto (X_{(S,\cal L,\varphi)}\ra S)$ is an isomorphism of stacks: this follows essentially from \cite[Proposition 2.2]{AV}.
    
  \subsubsection{} 
    \label{sub:line_bundles_on_BGm_and_their_torsors}
    Let us denote the quotients $[\bb V_n/\bb G_m]$ (resp.  $[\bb V_{m,n}/\bb G_m]$) by $ \scr{V}_n$ (resp. $\scr{V}_{m,n}$): these are rank-one  (resp. rank-two) vector bundles on $\scr B \bb G_m$. The substack $[\bb V_n\smallsetminus 0/\bb G_m]\subseteq [\bb V_n/\bb G_m]$ is naturally identified with the complement of the zero-section $s_0$ of $\scr{V}_n$. A map $S\ra \scr V_{n}$ can be regarded as a map $f:S\ra\scr B\bb G_m$ together with a section of $f^*\cal U^{\otimes (-n)}$, because the universal line bundle $\cal U$ coincides by definition with $\scr V_{-1}$. In other terms, we can regard the objects of the stack $\scr V_{n}$ as triples $(S,\cal L,\sigma:\cal O_S\ra\cal L^{\otimes (-n)})$, where $\cal L$ is a line bundle on $S$. Consequently, the objects of $\scr V_n\smallsetminus s_0$ are triples $(S,\cal L,\sigma:\cal O_S\overset{\simeq}{\ra} \cal L^{\otimes (-n)})$, where the morphism of line bundles is an isomorphism because the section $\sigma$ cannot vanish. We immediately deduce that there is an isomorphism $\scr V_n\smallsetminus s_0 \ra \scr F_n$ given by $(S,\cal L,\sigma)\longmapsto (S,\cal L,\sigma^{\vee})$, and therefore an isomorphism 
    \[\phi_n:[\bb V_{n}\smallsetminus 0/\Gm] \simeq \scr V_{n} \smallsetminus s_0 \overset{\simeq}{\ra} \scr B \mu_n\]
    for every $n\geq 1$.
    
    The line bundle obtained by pulling back $\cal U$ along the canonical projection $p_n:\scr V_{n}\smallsetminus s_0 \ra \scr B \bb G_m$ comes with a canonical trivialization of its $n$-th tensor power. We denote this pair of line bundle and trivialization with $(p_n^\ast \cal U, \sigma^{\vee}:p_n^\ast \cal U^{\otimes n} \overset{\simeq}{\ra} \cal O_{\scr V_n\smallsetminus s_0})$. It is clear that this pair is identified via $\phi_{n}$ with the universal $n$-torsion line bundle over $\scr B\mu_n$.
    
    In terms of equivariant scheme approximations, the quotients $\scr V_{n}$ and $\scr V_{n}\smallsetminus s_0$ can be respectively approximated by the total space of the line bundle $\bb V_{\bb P^N}(\cal O(n))$ over $\bb P^N$ and by its open subscheme $\bb V_{\bb P^N}(\cal O(n))\smallsetminus s_0$. This last fact follows from \ref{sub: O(-1)}, because taking scheme approximations of vector bundles commutes with taking tensor powers.

    \subsection{Moduli of elliptic curves}
    \subsubsection{}
    We denote by $\Mcal_{1,1}$ be the stack of elliptic curves, parametrizing $1$-marked smooth relative curves with geometrically integral fibers of genus one. Let $\Mbar_{1.1}$ be the stack of stable $1$-marked curves of genus one, which is a modular compactification of $\Mcal_{1.1}$. Recall that a stable relative $1$-marked curve $(p:C\ra S,\sigma)$ is the datum of a flat and proper morphism $p$ with a section $\sigma$ landing in the smooth locus of $p$; moreover the geometric fibers of $p$ are required to be geometrically integral, to have at worst nodal singularities, and a finite group of automorphism respecting the marking.
    
    \subsubsection{}
    In \cite[Proposition 20]{EG} the authors give presentations $\Mcal_{1,1}\simeq [U/G]$ and $\Mbar_{1,1}\simeq [W/G]$ for a certain group $G$ and certain $G$-schemes $U$ and $W$, assuming that the base field has characteristic $\neq 2,3$. In the Remark following \cite[Proposition 21]{EG}, the authors state an observation of Vistoli that the unipotent radical $H$ of $G$ acts on $U$ and $W$ freely, hence the quotient $[U/H]$ (respectively $[W/H]$) is a scheme, and there is a well defined action of $G/H\simeq\bb G_m$ on the both quotient schemes; from this we deduce
    \[ [U/G]\simeq [(U/H)/(G/H)]\simeq [(U/H)/\bb G_m],\quad [W/G]\simeq [(W/H)/(G/H)]\simeq [(W/H)/\bb G_m].  \]
    In the same Remark it is also observed that $U/H$ (respectively $W/H$) is isomorphic as a $\bb G_m$-scheme to $\bb V_{-4,-6}\smallsetminus\Delta$ (respectively $\bb V_{-4,-6}\smallsetminus\{0\}$), where $\Delta$ is the closed subscheme of equation $4a^3+27b^2=0$. Here $a$ and $b$ are coordinates on $\bb V_{-4,-6}$ of weight respectively $-4$ and $-6$. All together, these observations imply the following.
  \begin{prop}\label{prop:presentation_as_quotients}
      Suppose that the base field $k$ has characteristic $\neq 2,3$. Then we have:
      \begin{enumerate}
          \item $\Mbar_{1,1}\simeq [\bb V_{-4,-6}\smallsetminus\{0\}/\Gm]$.
          \item $\Mcal_{1,1}\simeq [\bb V_{-4,-6}\smallsetminus\Delta/\Gm]$.
      \end{enumerate}
      In other words $\Mbar_{1,1}$ is an open substack of $\scr V_{-4,-6}$, whose complement is the image of the zero section $s_0:\scr B \Gm \ra \scr V_{-4,-6}$.
  \end{prop}
  
  \subsubsection{}
  The universal curve on $\Mbar_{1,1}$ has a relative dualizing line bundle, and its pull-back along the universal section gives a line bundle on $\Mcal_{1,1}$ called the Hodge line bundle, which in this paper we denote $\cal E$. The Hodge line bundle determines a map $\Mbar_{1,1}\ra \scr B\Gm$ which coincides with the morphism $[\bb V_{-4,-6}\smallsetminus\{0\}/\Gm]\ra \cal B\Gm$ induced by the $\Gm$-equivariant projection on the point.
  In particular, the pull-back of the universal line bundle $\cal U$ from $\cal B\Gm$ to $\Mbar_{1,1}$ coincides with the Hodge line bundle.
  
  \subsubsection{}\label{subec:localization sequence Mcal}
  Similarly, the stack $\Mcal_{1,1}$ can be seen as an open substack of $\Mbar_{1,1}$, whose complement is the the quotient stack $[\Delta\smallsetminus\{0\}/\Gm]$. There is a canonical isomorphism $[\Delta\smallsetminus\{0\}/\Gm]\simeq\scr B\mu_2$, the existence of which follows from the fact that the action of $\Gm$ on $\Delta\smallsetminus\{0\}$ is transitive with stabilizer $\mu_2$.

  Another way to prove that $[\Delta\smallsetminus\{0\}/\Gm]\simeq\scr B\mu_2$ is the following: the complement of $\Mcal_{1,1}$ in $\Mbar_{1,1}$ is the divisor of non-smooth stable curves. This stack has only one geometric point, corresponding to the unique  stable nodal elliptic curve $(C,p)$ over an algebraically closed field, hence $[\Delta\smallsetminus\{0\}/\Gm]\simeq\scr B {\rm{Aut}}(C,p)$.
  The nodal elliptic curve $(C,p)$ can be obtained from a marked $\bb P^1$ by gluing two points $p_0$ and $p_1$ together: the automorphism group of  $(C,p)$ can then be identified with the group of automorphisms of $\bb P^1$ that fix the marking $p_\infty$ and the degree $2$ divisor $p_0+p_1$: the only non-trivial automorphism having this property is the involution that switches $p_0$ and $p_1$ and fixes $p_\infty$. We deduce that ${\rm{Aut}}(C,P)\simeq\mu_2$ and $[\Delta\smallsetminus\{0\}/\Gm]\simeq\scr B\mu_2$.

  Finally, observe that $\Mcal_{1,1}$ can also be seen as an open substack of $[V_{-4,-6}/\Gm]\simeq\scr V_{-4,-6}$, whose complement is the (singular) quotient stack $\cal C:=[\Delta/\Gm]$.

  The Chow rings of $\Mbar_{1,1}$ and $\Mcal_{1,1}$ have been computed in \cite[Proposition 21]{EG}.
  \begin{prop}
  \label{prop:chow groups}
  Let $k$ be a field of characteristic $\neq 2,3$. Then
          \begin{enumerate}
          \item $\CH^\ast(\Mbar_{1,1})\simeq\ZZ[\T]/(24\T^2)$,
          \item $\CH^\ast(\Mcal_{1,1})\simeq \ZZ[\T]/(12\T)$,
      \end{enumerate}
      where $\T$ corresponds to the Euler class of the dual of the Hodge line bundle.
  \end{prop}

  The strategy we will be using in \Cref{sec:Chow-Witt_ring_of_Mbar} for computing Chow-Witt rings of these stacks can be used to give a proof of \ref{prop:chow groups}.
  
  \subsubsection{}
  We have shown that the stacks $\scr B\bb G_m$ and $\scr B \mu_{n}$ have a presentation as $\bb G_m$-quotients of smooth $\bb G_m$-schemes. The same is true for $\Mcal_{1,1}$ and $\Mbar_{1,1}$ when the characteristic of the base field is $\neq 2,3$.
  
  In particular, the assumptions of \ref{subs:assumptions_on_X} applies to these stacks, for instance because the group $\bb G_m$ is special (that is, every $\bb G_m$-torsor can be trivialized Zariski-locally). This means that there is a well defined theory of Chow-Witt rings of these stacks, as defined in \ref{sub:chow-witt stack}.
  
\section{The Chow-Witt rings of the classifying stacks of $\bb G_m$ and $\mu_{2n}$}
\label{sec:CW Bmu2n}

In this Section, we first recall a presentation of the Chow-Witt ring of $\scr B\Gm$ (\Cref{prop:CW ring BGm}). We leverage this result to compute the Chow-Witt ring of $\scr B\mu_{2n}$: we begin by determining the additive structure (\Cref{prop:CW groups Bmu2n}) and afterward we focus on the multiplicative structure (\Cref{thm:CW ring of Bmu2n}). The knowledge of the multiplicative structure of $\scr B\mu_{2n}$ will be the key for determining the multiplicative structure of $\Mbar_{1,1}$ and $\Mcal_{1,1}$ in the next Section.

\subsection*{Conventions}
For sake of readability we have decided to hide from our notation most of the references to the "grading" of the twisting graded line bundles. So Chow-Witt groups appear as twisted simply by a line bundle, rather than a graded line bundle. We have however maintained the "grading" explicit in the definition of the generators of the various Chow-Witt groups appearing throughout. In general situations these "grading" would be relevant when comparing products $xy$ with $yx$. In our specific situation the grading turns out to be irrelevant, making our notational choice harmless. We invite in any case the reader to check this harmlessness autonomously.  

Regarding the twisting line bundles we have decided to make a choice of a line bundle for every class in the Picard group modulo $2$ of schemes and stacks appearing in the next sections. This can be done thanks to the fact that homology with coefficients in $\KMW_*$ and $\I^*$ is insensitive to the replacement of the twist $(\cal L,a)$ with a twist $(\cal L,a)\otimes (\cal M,b)^{\otimes 2}$. We have thus expressed all computations in terms of these choices.

Although it would have been more appropriate, we have refrained in the next two sections from using the notation introduced by Feld, following Rost, for Chow-Witt groups with coefficients. On one hand we believe that the notation introduced by Fasel in \cite{FasLectures} is a bit more intuitive; on the other hand keeping track of virtual vector bundles in the twisting coordinate would have added a whole layer of notational complications.

%\textcolor{dark-blue}{
%   Lista della spesa:
%   \begin{enumerate}
%   \item sara' meglio fare prima tutto $\cal M$ barra e poi $\cal M$ oppure per ogni stage fare direttamente entrambi?
%   \item Ingredienti della descrizione come pull-back, cui potremmo far eseguire il remark sul $\rm{Pic}(\cal M_{1,1})$ e sul %\rm{Pic}(\cal M_{1,1})$ che ci servono per twistare le cose. Non serve veramente la descrizione come pull-back, ma credo sia carino %metterla.
%   \item Descrizione dei Chow-Witt
%   \item Nucleo della descrizione come pull-back. Possibilmente dedurla dal punto precedente, senza riaprire la localizzazione.
%   \item descrizione più esplicita possibile delle nuove classi. Cosa fanno su $\bb R$?
%   \item la cuova classe in $\bar{\cal M_{1,1}}$ si può vedere come push forward da un punto che non sia la cuspide? Probabilmente sì. %a nuova classe su $\scr B\mu_2$ è pull-back di qualcosa su $\bar{\cal M_{1,1}}$? Mi par di no.
%   \item Relazione con la coomologia dello stack topologico reale?
%   \item Char 3?
%   \end{enumerate}
%}

%Recall that $\rm{Pic}(\Mbar_{1,1})\simeq \mathbb{Z}\cdot[\Ecal]$, where $\Ecal$ stands for the Hodge line bundle. In particular, %$\Ecal^0=\mathcal{O}$ and $\Ecal^1=\Ecal$.
%\begin{cor}\label{cor:equivariant_presentations}
%   Suppose that the base field $k$ has characteristic $\neq 2,3$, then:
%   \begin{enumerate}
%       \item $\CHW(\Mbar_{1,1},\Ecal^i)\simeq \CHW_{\Gm}(V_{-4,-6}\smallsetminus\{0\},V_i)$, for $i=0,1$.
%       \item $\CHW(\Mcal_{1,1},\Ecal^i)\simeq \CHW_{\Gm}(V_{-4,-6}\smallsetminus\Delta,V_i)$, for $i=0,1$.
%   \end{enumerate}
%\end{cor}

  \subsection{The Chow-Witt ring of $\scr B\Gm$}
    \subsubsection{}
    In what follows, we set 
    \[\T:=e(\cal U^\vee)\in \CHW^1(\scr B\Gm,\cal U,1)\] 
    the Euler class of the dual of the universal line bundle $\cal U^\vee$ over $\scr B\Gm$. %In what follows we will hide the grading, and just refer to $\T$ as a class in codimension $1$ and twist $\cal U$.
    
    To start, we recall the structure of the $\I^*$-cohomology of $\scr B\Gm$. The description of the groups $H^{i,j}(\scr B\Gm,\I^\ast,\bullet)$ follows directly from the computation of $H^{i,j}(\bb P^n,\I^\ast,\bullet)$ contained in \cite[Section 11]{MR3061003}, because $\bb P^n$ is an equivariant scheme approximation of $\scr B\Gm$. The description of the $\W(k)$-algebra $H^\ast(\scr B\Gm,\I^\ast,\bullet)$ can be found in \cite[Proposition 4.3]{WendtGrass}.
            \begin{prop}\label{prop:Ij cohomology BGm}
             Let $\cal U$ be the universal line bundle over $\scr B\Gm$. We have:
             \begin{align*} 
             &
             H^{i,j}(\scr B\Gm,\I^\ast,\cal O)\simeq \left\{
             \begin{tabular}{cl}
                 $\I^j(k)$ & if $i=0$ \\
                 $\Km_{j-i}(k)$ & if $i\geq 2$ even \\
                 $0$ & if $i$ odd
             \end{tabular}
             \right.
             &
             H^{i,j}(\scr B\Gm,\I^\ast,\cal U)\simeq \left\{
             \begin{tabular}{cl}
                 $0$ & if $i$ even \\
                 $\Km_{j-i}(k)$ & if $i$ odd
             \end{tabular}
             \right.
             \end{align*}
           Moreover we have an isomorphism of $\W(k)$-algebras
            \[ H^{\ast,\ast}(\scr B\Gm,\I^*,\bullet) \simeq \W(k) [\T]/\I\cdot\T\]
            induced by mapping $\T\mapsto e(\cal U^\vee) \in H^{1,1}(\scr B\Gm,\I^*,\cal U,1)$, the $\I^*$-theoretic Euler class of $\cal U^\vee$.
          \end{prop}
    Actually, all the cohomology groups of $\scr B\Gm$ with coefficients in $\KMW_\ast$ have been fully described in \cite[Theorem 11.7]{MR3061003}.
    \begin{prop}
    \label{prop:KMW cohomology BGm}
     We have
        \[
        H^{i,j}(\scr B\Gm,\KMW_\ast,\cal O)\simeq \left\{ 
        \begin{tabular}{cl}
        $\KMW_j(k)$  & if $i=0$  \\
        $\KM_{j-i}(k)$ & if $i\geq 2$ is even \\
        $2\KM_{j-i}(k)$ & if $i$ is odd
        \end{tabular}
        \right., \quad
        H^{i,j}(\scr B\Gm,\KMW_\ast,\cal U)\simeq \left\{ 
        \begin{tabular}{cl}
        $2\KM_{j-i}(k)$ & for $i$ even \\
        $\KM_{j-i}(k)$ & for $i$ odd.
        \end{tabular}
        \right.
        \]
    \end{prop}

    \subsubsection{}
    \label{sub:H}
    We denote by
    \[\HH \in \CHW^{0}(\scr B\Gm,\cal U,0)\simeq \Pb^0(\scr B\Gm,\cal U,0) \subseteq H^{0,0}(\scr B\Gm, \I^\ast, \cal U,0)\times \CH^0(\scr B\Gm) .\]
    that corresponds to the pair $(0,2)$. In other terms $\HH$ is a sort of hyperbolic form but with twisted coefficients in $\cal U$. This can be represented by the symbol $h\otimes s \in \KMW_0(k(x_1,\dots ,x_n),\cal U)$ where $s$ is any generic generator of $\cal U$, and $h$ represents the Grothendieck-Witt class of the hyperbolic plane.
    
    The classes $\T$ and $\HH$ are particularly important for our purposes. Our work is fully formulated in terms of $\Gm$-equivariant geometry, and thus the Chow-Witt rings we are interested in are best understood as $\CHW^\ast(\scr B\Gm,\bullet)$-algebras. In particular they all contain the naturally defined pull-backs of the classes $\T$ and $\HH$.
    \begin{prop}\label{prop:CW ring BGm}
        We have an isomorphism of $\GW(k)$-algebras
        \[ \CHW^*(\scr B\Gm,\bullet)\simeq \GW(k)[\T,\HH]/(I\cdot\T,I\cdot\HH,\HH^2-2h)  \]
        given by mapping $\T$ to $e(\cal U^\vee)$ and $\HH$ to the class introduced in \ref{sub:H}.
    \end{prop}
    \begin{proof}
    This is an easy consequence of \cite[Theorem 1.1]{WendtGrass} in the case $n=1$.
    \end{proof}

    \subsubsection{}
      \label{subsub:Eul_class_of_Un}
     In Chow-Witt theory, as much as in $\I^\ast$- and $\W$-theory, Euler classes of the tensor product of line bundles $\cal L \otimes \cal M$ in general is not equal the sum of the Euler classes of $\cal L$ and $\cal M$. 
     
     For instance, the usual formula $e(\cal U^{\otimes -2})=2\T$ cannot hold in this setting, because the element on the left belongs to $\CHW^1(\scr B\Gm,\cal U^{\otimes 2})$ whereas the element on the right is in $\CHW^1(\scr B\Gm,\cal U)$. However on $\scr B\bb G_m$ it is easy to write an explicit formula in terms of the multiplicative structure described in \Cref{prop:CW ring BGm}:
    \[ e(\cal U^{\otimes(-2n)})=n\T\HH,\quad e(\cal U^{\otimes (-2n-1)})=(2n+1)\T. \]

  Observe also that from the relations in \Cref{prop:CW ring BGm} we have $\T^i\HH^{2j}=2^jh^j\T^i$. Actually, for $i>0$ the right hand side of this equality can be further simplified into $4^j\T^i$: indeed, the element $2^jh^j-4^j$ belongs to the fundamental ideal $I$, hence $(2^jh^j-4^j)\cdot \T=0$, that is $2^jh^j\cdot\T=4^j\T$, from which our claim follows. We will frequently take advantage of these formulas in our computations.

%\subsubsection{Construction of an additional generator}\label{subsubsec:a new element}
  \subsection{The Chow-Witt groups of $\scr B\mu_{2n}$}
  For the remainder of the section, we assume that the base field $k$ has characteristic coprime with $2n$.

    As we have seen in \ref{sub:line_bundles_on_BGm_and_their_torsors} the stack $\scr B \mu_{2n}$ is naturally identified with the complement of the zero section of the line bundle $p:\scr V_{-2n} \ra \scr B\bb G_m$. Observe that the line bundle $\scr V_{-2n}$ is the total space of the invertible sheaf $\cal U^{\otimes 2n}$, hence the pull-back of $\cal U^{\otimes 2n}$ to $\scr V_{-2n}$ comes equipped with a canonical section $\sigma$. The open substack $\scr V_{-2n}\smallsetminus s_0$ can be regarded both as the complement of the zero section $s_0:\scr B\Gm\ra \scr V_{-2n}$ and as the non-vanishing locus of $\sigma$. In other terms, the section $\sigma$ trivializes the pull-back of $\cal U^{\otimes 2n}$ over $\scr V_{-2n}\smallsetminus\{0\}\simeq\scr B\mu_{2n}$
    
    The localization sequence associated with the decomposition 
    \begin{equation}
      \label{eqn:loc_seq_Bmu_2n}
      s_0(\scr B\Gm) \subset \scr V_{-2n} \supset \scr V_{-2n}\smallsetminus s_0(\scr B\Gm)
    \end{equation}
    gives a boundary morphism
    \[\partial:\CHW^0(\scr B\mu_{2n}, \cal O) \ra H^0(\scr B\Gm,\KMW_{-1},\cal U^{\otimes 2n})\simeq H^0(\scr B\Gm,\W,\cal U^{\otimes 2n}).\]
  Cohomology groups are well defined because all the stacks involved are smooth.
    We can apply the construction of \Cref{lemma:splitting CW0}, which produces an element $q_{gen}$ in $\CHW^0(\scr B\mu_{2n})$. The latter corresponds to the restriction over the generic point of the quadratic form
    \[q:\cal U^{\otimes n} \otimes \cal U^{\otimes n} \ra \cal O \]
    induced by the trivialization of $\cal U^{\otimes 2n}$ defined by $\sigma$.
    \begin{defin}\label{defin:class u}
        We define $\U$ as the element $q_{gen}-1$ in $\CHW^0(\scr B\mu_{2n},\cal O,0)$ provided by \Cref{lemma:splitting CW0}.
  \end{defin}
  \begin{rmk}\label{rmk:class U}
  The element $\U$ is only visible at the level of Chow-Witt groups and is not captured by the usual Chow theory: by this we mean that the image of $\U$ along the map
  \[ \CHW^0(\scr B\mu_{2n},\cal O) \ra \CH^0(\scr B\mu_{2n}) \]
  is zero.
  \end{rmk}
%\begin{rmk}
%A somewhat more intrinsic way to think of the element $u$ defined above is the following. As already said, the Picard group of the stack $\scr B\mu_2$ is generated by a line bundle $\cal U$ of $2$-torsion, that is
%\[ \cal U\otimes \cal U \xrightarrow{\simeq} \cal O. \]
%We can regard this isomorphism as a quadratic form $q$ on the line bundle $\cal U$, which in turn determines an element in $\CHW^0(\scr B\mu_2)$.

%This element coincides with the quadratic form $\langle x \rangle$ that we constructed above using the equivariant approximation of $\scr B\mu_2$ given by $V_2\smallsetminus\{0\}$. Therefore, the element $\U$ is equal to $q-1$.
%\end{rmk}
The Picard group of $\scr B\mu_{2n}$ is isomorphic to $\ZZ/2n\ZZ$, and a generator for this group is given by the pull-back of $\cal U$, which we will keep indicating with the same symbol. Henceforth, up to isomorphism and up to squares, the Chow-Witt groups of $\scr B\mu_{2n}$ are twisted either by $\cal O$ or by $\cal U$. The same argument used in \cite{MNW} for $\scr B\mu_2$ can be extended to the following Proposition.
        \begin{prop}
        \label{prop:CW groups Bmu2n}
      Suppose that the base field $k$ has characteristic co-prime with $2n$. Let $\T$ be the Euler class of $\cal U^\vee$, let $\HH$ be the class $(0,2)\in \CHW^0(\scr B\mu_{2n},\cal U)$ introduced in \ref{sub:H}, and let $\U$ be the class defined in \Cref{defin:class u}. Then the following description of $\CHW^*(\scr B\mu_{2n},\bullet)$ holds:
      %\begin{table}
      %\caption{Additive structure for the Chow-Witt ring of }
      \[\begin{array}{c@{\hskip 0.16in}c@{\hskip 0.16in}c@{\hskip 0.16in}c@{\hskip 0.16in}c@{\hskip 0.16in}c@{\hskip 0.16in}c}
      \toprule
      Twist  & 0 & 1 & 2 & 3 & 2k & 2k+1\\
      \midrule
       \cal O  & \GW(k)\cdot 1\oplus \W(k)\cdot \U& \bb Z/n\cdot \HH\T & \bb Z/4n\cdot \T^2& \bb Z/n\cdot \HH\T^3 &\bb Z/4n\cdot \T^{2k} & \bb Z/n\cdot \HH\T^{2k+1}\\
       \cal U & \bb Z\cdot\HH & \bb Z/4n\cdot\T & \bb Z/n\cdot\HH\T^2 & \bb Z/4n\cdot \T^3 & \bb Z/n\cdot\HH\T^{2k}& \bb Z/4n\cdot\T^{2k+1}.\\
      \bottomrule
      \end{array}\]
   % \end{table}
    \end{prop}

    \begin{proof}
    %Recall that $\scr B\mu_{2n}$ is an open substack of the line bundle $\cal U_{2n}$ over $\scr B\Gm$, whose complement is the image of the zero section. By [REF] we have the following exact sequences
    %\[ \CHW^{i-1}(\scr B\Gm) \ra \CHW^i(\cal U_{2n}) \ra \CHW^i(\scr B\mu_{2n}) \ra H^i(\scr B\Gm, \KMW_{i-1}) \ra H^{i+1}(\cal U_{2n},\KMW_i) \]
    It is clear that the classes $\T$ and $\HH$ are the pull-backs to $\scr B\mu_{2n}$ of the classes $\T$ and $\HH$ on $\scr B\Gm$. As $\CHW^\ast(\scr B\mu_{2n},\bullet)$ has a natural $\CHW^\ast(\scr B\Gm,\bullet)$-algebra structure, we will adopt the same notation for $\T$ and $\HH$ on both $\scr B\mu_{2n}$ and $\scr B\Gm$.
    
    The localization sequence induced by the decomposition \eqref{eqn:loc_seq_Bmu_2n} of $\scr V_{-2n}$, in light of \ref{sub:can_loc_sequence_on_a_bundle}, gives a long exact sequence
    \[\xymatrix{
    \CHW^{i-1}(\scr B\Gm,\bullet\otimes \cal U^{\otimes 2n}) \ar[r]^(0.6){e\cdot} & \CHW^i(\scr B\Gm,\bullet) \ar[r]^{(j\circ p) ^\ast} & \CHW^i(\scr B\mu_{2n},\bullet) \ar[r]^(0.4){\partial} & H^i(\scr B\Gm, \KMW_{i-1},\bullet\otimes\cal U^{\otimes 2n}),
    } \]
    where $e$ denotes the Euler class $e(\cal U^{\otimes 2n})=-n\T\HH$ (see \ref{subsub:Eul_class_of_Un}).
    By \Cref{prop:CW ring BGm} and \Cref{prop:KMW cohomology BGm}, for $i\geq 1$ and $\bullet \simeq \cal U^{\otimes i+1}$ (modulo squares) we get
  \[\begin{cases}
    \ZZ\cdot\T^{i-1} \xrightarrow{n\T\HH} \ZZ\cdot \T^i\HH \ra \CHW^i(\scr B\mu_{2n},\cal U^{\otimes i+1}) \ra 0, & \text{ if } i \geq 2\\
    \GW\cdot\T^{i-1} \xrightarrow{n\T\HH} \ZZ\cdot \T^i\HH \ra \CHW^i(\scr B\mu_{2n},\cal U^{\otimes i+1}) \ra 0, & \text{ if } i=1.
    \end{cases}\]
    where the zeroes on the right come from the fact that for $i\geq 1$ we have $H^i(\scr B\Gm, \KMW_{i-1},\bullet\otimes\cal U^{\otimes 2n})$ is a subgroup of $\KM_{-1}$, and this last group is zero (see \Cref{prop:KMW cohomology BGm}). This allows us to conclude that $\CHW^i(\scr B\mu_{2n}, \cal U^{\otimes i+1})\simeq \ZZ/n\ZZ\cdot\T^i\HH$. Depending on the parity of $i$, we get the claimed results for either $\CHW^i(\scr B\mu_{2n})$ or $\CHW^i(\scr B\mu_{2n},\cal U)$.
    
    For $i\geq 1$ and $\bullet\simeq \cal U^{\otimes i}$ we have instead exact sequences
    \[ \ZZ\cdot\T^{i-1}\HH \xrightarrow{n\T\HH} \ZZ\cdot\T^i \ra \CHW^i(\scr B\mu_{2n},\cal U^{\otimes i}) \ra 0, \]
    and thus $\CHW^i(\scr B\mu_{2n}, \cal U^{\otimes i})\simeq \ZZ/4n\ZZ\cdot\T^i$, since $\HH^2\T=2h\T=4\T$.

    We are left with the case where $i=0$. When $\bullet\simeq \cal U$ we have
    \[ 0\ra \ZZ\cdot \HH \ra \CHW^0(\scr B\mu_{2n},\cal U) \ra 0. \]
    When instead $\bullet \simeq \cal O$ it is easily deduced from the $\W$-cohomology of $\scr B \bb G_m$ (see \Cref{prop:KMW cohomology BGm}) that we have an extension
    \begin{equation}\label{eq:ses Bmu2} 0 \ra \GW(k)\cdot 1 \ra \CHW^0(\scr B\mu_{2n},\cal O) \overset{\partial}{\ra} \W(k) \ra 0 \end{equation}
    By \Cref{lemma:splitting CW0} we know that $\partial(\U)$ is a unit in $\W(k)$, hence the morphism
    \[ W(k)\ra \CHW^0(\scr B\mu_{2n},\cal O) \]
    given by multiplication by $\partial(U)^{-1}$ and followed by multiplication by $\U$ splits $\partial$ in the extension above. In particular $\CHW^0(\scr B\mu_{2n})\simeq \GW(k)\oplus \W(k)$, and this concludes.
    \end{proof}

%\begin{proof}
%Arguing exactly as in the proof of \Cref{prop:CW groups Bmu2n}, for $i\geq 2$ even we have the following exact sequences:
%\[ \ZZ\cdot\T^{i-1} \xrightarrow{{2n+1}\T} \ZZ\cdot\T^{i} \ra \CHW^i(\scr B\mu_{2n}) \ra 0, \]
%which tells us that $\CHW^i(\scr B\mu_{2n+1})=\ZZ/(2n+1)\ZZ\cdot T^i$. Similarly, for $i$ odd we have
%\[ \ZZ\cdot\T^{i-1}\HH \xrightarrow{{2n+1}\T} \ZZ\cdot\T^{i}\HH \ra \CHW^i(\scr B\mu_{2n}) \ra 0,  \]
%thus $\CHW^i(\scr B\mu_{2n+1})=\ZZ/(2n+1)\ZZ\cdot T^i\HH$.
%\end{proof}

    \subsection{Multiplicative structure}
    \subsubsection{}
      \label{subsub:the_approx_Q}
      
    The line bundle $\bb V(\cal O(-2n))\ra \bb P^1$ is an equivariant approximation of $\scr V_{-2n}\ra \scr B\Gm$, and the complement of the zero section $s_0$ is an equivariant approximation of $\scr B\mu_{2n}$. We set $Q:=\bb V(\cal O(-2n))\smallsetminus s_0$.
    
    Let $X_0$ and $X_1$ be homogeneous coordinates on $\bb P^1$, then we set $x=X_0/X_1$ and $y=X_1/X_0$. The rational function $x$ is a local coordinate for the open subscheme $U_1:=\{ X_1 \neq 0 \}$ and $y$ is a local coordinate for $U_0:=\{ X_0\neq 0 \}$.
    
    We chose a trivialization $\bb V(\cal O(-2n))_{|U_1}\simeq U_1 \times \bb A^1$: we call $t$ the coordinate for the factor $\bb A^1$ appearing in the trivialization.
    Similarly, we chose $\bb V(\cal O(-2n))_{|U_0}\simeq U_0 \times \bb A^1$, and we call $s$ the coordinate for $\bb A^1$.
    Observe that in $(U_0\cap U_1)\times\bb A^1$ we have $x=1/y$ and $t=y^{2n} s$. The open subschemes $U_0\times (\bb A^1\smallsetminus\{0\})$ and $U_1\times (\bb A^1\smallsetminus\{0\})$ are affine charts for $Q$.
    
    \subsection{}
    The map $\bb V(\cal O(-1))\ra \bb V(\cal O(-2n))$ given by the $2n^{\rm{th}}$-power defines an \'{e}tale cover of $Q=\bb V(\cal O(-2n))\smallsetminus s_0$ of degree $2n$; this in turn induces a map $\varphi:Q\ra\scr B\mu_{2n}$. We need the following result.
    %And in fact $\pi$ is canonically identified with the $\mu_2$-torsor obtained by quotienting $\bb A^{2}$ by the diagonal multiplication action of $\mu_2$. Similarly $q$ is the $\bb G_m$-torsor obtained by quotienting $Q$ by the residual diagonal $\bb G_m$-action on $Q$. After embedding $Q \subset \bb A^3\smallsetminus\{0\} \subset \bb P^3\smallsetminus\{0\}$ and $\bb P^1$ in $\bb P^2$ as the conic $X_0X_1=X_2^2$, the map $q$ can be identified as the restriction to the projection at infinity $\bb P^3\smallsetminus 0 \ra \bb P^2$ to the punctured affine cone over the above conic.
    %Consider the closed subscheme $Q\subset\bb A^3\smallsetminus\{0\}$ of equation $X_0X_1=X_2^2$. Alternatively, we can describe this scheme as the complement of the zero section in the total space of the line bundle $\bb V(\cal O(-2))$ over $\bb P^1$. Observe that there is an \'{e}tale cover of degree $2$ given by
    %\[\bb A^{2}\smallsetminus\{0\}\ra Q, (x,y)\longmapsto (x^2,y^2,xy), \]
    %which in turn induces a morphism $\varphi:Q\ra\scr B\mu_2$. Let $L$ be the line bundle on $Q$ obtained by pulling back $\cal O(1)$ along $Q\ra\bb P^1$ or, equivalently, the pullback of $\cal U$ along $\varphi$.

    \begin{lemma}\label{lemma:pullback varphi iso}
        The pull-back homomorphisms
        \[ \varphi^*:\CHW^0(\scr B\mu_{2n})\ra\CHW^0(Q),\quad \varphi^*:\CHW^1(\scr B\mu_{2n},\cal U) \ra \CHW^1(Q,\cal O(-1)) \]
        are both isomorphisms, where $\cal O(-1)$ is pulled back from $\bb P^1$.
    \end{lemma}
    
    \begin{proof}
    Observe that $Q$ is an equivariant scheme approximation of $\scr B\mu_{2n}$ in codimension $0$, hence their Chow-Witt groups in degree zero are isomorphic. We cannot apply the same argument in degree $1$ because $Q$ is not an equivariant scheme approximation of $\scr B\mu_{2n}$ in codimension $1$.
    
    Nevertheless, the open embedding $Q\hookrightarrow \bb V(\cal O(-2n))$ induces a long exact sequence that, after identifying the Chow-Witt groups of $\bb P^1$ and $\bb V(\cal O(-2n))$, looks like
     \[\xymatrix{
                \CHW^{0}(\bb P^1,\cal O(-1)\otimes \cal U^{\otimes 2n}) \ar[r]^(0.6){e\cdot} & \CHW^1(\bb P^1,\cal O(-1)) \ar[r]^{(j\circ p) ^\ast} & \CHW^1(Q,\cal O(-1)) \ar[r]^(0.4){\partial} & H^1(\bb P^1, \KMW_{0},\cal O(-2n-1)).
                } \]
    We can use \cite[Corollary 11.8]{MR3061003} to write down explicitly the Chow-Witt groups above. We obtain
    \[\xymatrix{ 2\bb Z \ar[r]^{-2n\cdot \T } & \bb Z\cdot \T \ar[r] & \CHW^1(Q,\cal O(-1)) \ar[r] & H^1(\bb P^1, \KMW_{0},\cal O(-2n-1)). }\]
    From \cite[Theorem 11.7]{MR3061003} we get that the last term is zero, so $\CHW^1(Q,\cal O(-1))\simeq \bb Z /4n \cdot \T$. 
    
    As the pullback map $ \varphi^*:\CHW^1(\scr B\mu_{2n},\cal U) \ra \CHW^1(Q,\cal O(-1))$ sends $\T$ to $\T$, by looking at the formula for $\CHW^1(\scr B\mu_{2n},\cal U)$ given in \Cref{prop:CW groups Bmu2n} we get the desired conclusion.
    \end{proof}
    
        We are ready to prove the main Theorem of this section.
    \begin{thm}\label{thm:CW ring of Bmu2n}
        Let $k$ be a field whose characteristic is co-prime with $2n$. Then we have an isomorphism of $\GW(k)$-algebras 
        \[ \CHW^*(\scr B\mu_{2n},\bullet)\simeq\GW(k)[\T,\HH,\U]/(I\cdot\T,I\cdot\HH,h\U,\HH\U, n\T\HH, \HH^2-2h,\U^2+2\U,\T\U- 2n\T)\]
        given by mapping $\T\mapsto e(\cal U^{\vee})$, $\HH\mapsto (0,2)$, and $\U$ to the element introduced in \Cref{defin:class u}.
    \end{thm}

    \begin{proof}
    From \Cref{prop:CW groups Bmu2n} we know what are the generators of $\CHW^*(\scr B\mu_{2n},\bullet)$ as $\GW(k)$-algebra: the Euler class of the dual of the universal line bundle $\cal U^{\vee}$ on $\scr B\mu_{2n}$, the hyperbolic form $\HH$ and the element $\U$ introduced in \Cref{defin:class u}.

    The first two generators are pulled back from the Chow-Witt ring of $\scr B\Gm$, described in \Cref{prop:CW ring BGm}, thus all the relations in this last ring hold  also in the Chow-Witt ring of $\scr B\mu_{2n}$. This shows that the only thing left to determine is the product of $\U$ with itself and the other generators. 

    First note that by construction $h\U$ is zero, while $n\T\HH$ is the Euler class of $\cal U^{\otimes -2n}$,and thus $n\T\HH$ vanishes as well.
    
    From the additive structure we know that $\HH \U=m\HH$ for some $m\in\bb Z$. We observed in \Cref{rmk:class U} that $\U$ is sent to zero by the map $\CHW^0(\scr B\mu_{2n})\ra \CH^0(\scr B\mu_{2n})$, hence also $\HH\U$ is sent to zero.
    We already mentioned that $Q$ is an equivariant scheme approximation of $\scr B\mu_{2n}$ in codimension $0$, hence $\CH^0(\scr B\mu_{2n})\simeq\CH^0(Q)\simeq\bb Z$, which is torsion free.
    As the map $\CHW^0(\scr B\mu_{2n})\ra \CH^0(\scr B\mu_{2n})$ maps $m\HH$ to $2m$, we deduce that $m=0$.
    
    Thanks to \Cref{lemma:pullback varphi iso}, the relation $\U^2+2\U=0$ can be checked on the equivariant approximation $Q$ introduced in \ref{subsub:the_approx_Q}, from which we adopt the notation from now on.
    The function field of $Q$ is $k(x,t)$ and the pull-back of the element $\U$, regarded as an element in $\GW(k(x,t))$ is $\langle t^{-1}\rangle-1$, hence we have 
    \[ (\langle t^{-1} \rangle -1)^2 = \langle t^{-2} \rangle -2\langle t^{-1} \rangle + 1 = 2(1-\langle t^{-1} \rangle), \]
    from which the sought relation follows.
    
    We are left with the computation of $\T\U$. By Lemma \ref{lemma:pullback varphi iso} we can check the relation on the approximation $Q$. Hereafter we denote by $\U$, $\HH$, $\T$ and $\cal U$ the restriction to $Q$ of the corresponding elements on $\scr B\mu_{2n}$. 
    %For this, consider the closed subscheme $Q\subset\bb A^3\smallsetminus\{0\}$ of equation $X_0X_1=X_2^2$, and recall that we have a map $\varphi:Q\ra\scr B\mu_2$: by \Cref{lemma:pullback varphi iso} the pullback along this map is an isomorphism of Chow-Witt groups in degree $0$ and in degree $1$ twisted by $L=\varphi^*\cal U$.
    %This implies that to compute explicitly $\U^2$, $\U\HH$ and $\U\T$ we can equivalently make the computations for their pullbacks in the Chow-Witt ring of $Q$. 

    %To study the multiplicative structure we need a more down-to-earth description of the elements involved. In first place we observe that the projection at infinity
    %\[\bb A^3\smallsetminus\{0\}\ra\bb P^2,(X_0,X_1,X_2)\longmapsto [X_0:X_1:X_2] \]
        %allows to identify $Q$ with the punctured affine cone of the conic of equation $X_0X_1=X_2^2$ in $\bb P^2$.
    %restricted to $Q$ coincides with the map $p:Q\ra\bb P^1$ obtained by regarding $Q$ as the total space of $\cal O(-2)$ minus the zero section: indeed, the image of $Q$ in $\bb P^2$ is the conic of equation $X_0X_1=X_2^2$, which is isomorphic to $\bb P^1$.

    %The multiplication by $X_0$, regarded as a section of $p^*\cal O(2)$, induces a trivialization $\cal O_Q\simeq p^*\cal O(2)$. We can apply \Cref{lemma:splitting CW0}, which tells us that a generator for the Witt part of $\CHW^0(Q)\simeq\GW(k)\oplus\W(k)$ is given by $\langle X_0 \rangle-1=\eta[X_0]$.
        The element $\U$ in $\CHW^0(Q)$, regarded as an element of $\GW(k(Q))$, is $\langle t^{-1} \rangle -1 = \eta[t^{-1}]$.
    
      Let $C$ be the Cartier divisor $\{(U_0\times (\bb A^1\smallsetminus\{0\}),1),(U_1\times(\bb A^1\smallsetminus\{0\}),x)\}$ and let $D$ be the Cartier divisor $\{(U_0\times (\bb A^1\smallsetminus\{0\}),y),(U_1\times (\bb A^1\smallsetminus\{0\}),1)\}$. Observe that the class of both these divisors in $\CHW^1(Q,\cal O(-1))$ is $\T$. Therefore to compute $\U\T$ we can leverage the description contained in \cite[Lemma 3.6]{FasLectures} for the product of an element with a Cartier divisor, using either $C$ or $D$ as model for $\T$.

    The divisor $D$ is contained in $U_0\times (\bb A^1\smallsetminus\{0\})$ and its associated ideal is generated by $y$. From \cite[Lemma 3.6]{FasLectures} we obtain:
    \[e(\cal O(D))\cap \eta[t^{-1}] = d(y)^0([y]\eta[t^{-1}]\otimes y)=(\eta \partial^{y}_\nu[y,t^{-1}] )\otimes y=\eta[t^{-1}]\otimes \overline{y}^{\vee}\otimes y \]
    In the second equality we used the commutativity of the product with $\eta$ (which follows from the definition of the boundary, see \cite[Theorem 1.7]{FasLectures}) and the definition of $d(y)^0$ as the part of the differential supported on $\{y =0 \}$ (see \cite[pg. 109]{FasLectures}. The third equality is obtained by combining the fact that $t^{-1}$ is a unity in $U_0$ together with the very definition of $\partial^{y}_\nu$. 

    The symbol produced in this way gives an element in $\CHW^1(Q,\cal O(-D))$. This last group is obviously isomorphic to $\CHW^1(Q,\cal O(-C))$, and such isomorphism sends $\eta[t^{-1}]\otimes \overline{y}^{\vee}\otimes y$ to $\eta[t^{-1}]\otimes \overline{y}^{\vee}\otimes 1$. We conclude that $\U\T$, as an element in $\CHW^1(Q,\cal O(-C))$, is equal to the equivalence class of $(\eta[t^{-1}]\otimes\overline{y}^\vee\otimes 1)\cdot [D]$, where the writing $(\alpha)\cdot [C]$ stands for the element in $\oplus_{x\in Q^{(1)}} \GW(k(x))$ that is $\alpha$ in $\GW(k(C))$ and zero otherwise.

    Consider now the boundary morphism in the Gersten complex
    \[ \KMW_1(k(Q),\cal O(-C)) \ra \bigoplus_{x\in Q^{(1)}} \KMW_0(k(x),(\fk m_x/\fk m_x^2)^{\vee}\otimes \cal O(-C)_x). \]
    We claim that the boundary of $[x^{2n}t]\otimes x$ is 
    \[ (nh\langle t\rangle \otimes 1)\cdot [C] - \langle -1 \rangle (\eta[t^{-1}]\otimes\overline{y}^\vee\otimes 1)\cdot [D]. \]
    Observe that in $\CHW^1(Q,\cal O(-C))$ we have that the first term is equal to $2n\T$, whereas the second one coincides with $\U\T$. As the boundary of an element determines a relation in the corresponding Chow-Witt group, this would imply that $\U\T=2n\T$, as expected.

    To prove the claim we only have to apply the rules for the computation of residues (\cite[Theorem 1.7]{FasLectures}. First observe that the residue of $[x^{2n}t]\otimes x$ along the points of codimension $1$ contained in $Q\smallsetminus (C\cup D)$ is always zero. Therefore we have
    \[\partial([x^{2n}t]\otimes x)=\partial_C([x^{2n}t]\otimes x)\cdot [C]+\partial_D([x^{2n}t]\otimes x)\cdot [D].\]
    For the first term we have:
    \begin{align*}
    \partial_C([x^{2n}t]\otimes x) = \partial^{x}_C([x^{2n}t])\otimes \overline{x}^{\vee}\otimes x = nh\langle t\rangle\otimes 1,
    \end{align*}
    because $C$ is contained in the affine open subset $U_1\times (\bb A^1\smallsetminus\{0\})$ where $t$ is invertible and $x$ is a local generator for $\cal O(-C)$. The last equality comes from combining the formula \cite[Lemma 1.3.(3)]{FasLectures} with the usual rules for the computation of residues (\cite[Theorem 1.7]{FasLectures}). For computing the second term, first observe that $D\subset U_0\times (\bb A^1\smallsetminus\{0\})$, and in this affine open subset $\cal O(-C)$ is generated by $1$. In particular, an element $\alpha\otimes\ell$ is equal to $\langle \ell \rangle\alpha\otimes 1$. Therefore
    \begin{align*}
    \partial_D([x^{2n}t]\otimes x) &= \partial_D(\langle x \rangle [x^{2n}t]\otimes 1) = \partial_D^{y}([x^{2n}t]+\eta[x,x^{2n}t])\otimes 1\\
        &=\partial_D^{y}([s])\otimes 1+\eta\partial_D^{y}([y^{-1},s])\otimes 1 = -\langle -1 \rangle \eta[s]\otimes \overline{y}^{\vee} \otimes 1.
    \end{align*}
    We have $\eta[s]=\langle s \rangle - 1 =\langle x^{2n}t \rangle - 1 = \eta[t]$ because $x^{2n}$ is obviously a square. For the same reason, we have $\eta[t]=\eta[t^{-1}]$, hence the last term above is equal in $\CHW^1(Q,\cal O(-D))$ to $-\langle -1 \rangle \U\T$, as expected. This concludes the proof.
    \end{proof}

\section{The Chow-Witt rings of $\Mbar_{1,1}$ and $\Mcal_{1,1}$}
\label{sec:Chow-Witt_ring_of_Mbar}
In this section we compute the Chow-Witt rings of $\Mbar_{1,1}$ (\Cref{thm:CW ring of Mbar}) and $\Mcal_{1,1}$ (\Cref{thm:CW ring of Mcal}). As in the previous section, we first determine the additive structure (\Cref{thm:additive structure Mbar} and \Cref{thm:additive structure Mcal}) and afterwards the multiplicative structure. The latter turns out to be basically determined by that of $\scr B\mu_{2}$ and $\scr B\mu_{12}$.
\subsection{The $\I^j$-cohomology of $\Mbar_{1,1}$}\label{sec:Ij cohomology of Mbar}
  We determine here the $\I^j$-cohomology of $\Mbar_{1,1}$. The result itself (\Cref{prop:Ij cohomology Mbar}) has some independent interest, which we highlight after the main proof. Moreover, this computation is useful for determining Chow-Witt groups.

  Recall that $\cal E$ denotes the Hodge bundle on $\Mbar_{1,1}$.
  \begin{prop}\label{prop:Ij cohomology Mbar}
     Suppose ${\rm{char}}(k)\neq 2,3$. Then we have:
     \begin{align*}
     &
     H^{i,j}(\Mbar_{1,1},\I^\ast, \cal O)\simeq \left\{
     \begin{tabular}{cl}
      $\I^j(k)$    & if $i=0$, \\
      $\I^{j-2}(k)$    & if $i=1$, \\
      $\Km_{j-i}(k)$ & if $i\geq 2$ even, \\
      $\Km_{j-i-1}(k)$ & if $i\geq 3$ odd.
     \end{tabular}
     \right.
     &
     H^{i,j}(\Mbar_{1,1},\I^\ast,\cal E)\simeq \left\{
     \begin{tabular}{cl}
      $\Km_{j-i-1(k)}$ & if $i$ even, \\
      $\Km_{j-i}(k)$ & if $i$ odd.
     \end{tabular}
     \right.
     \end{align*}
    Moreover we have an isomorphism of $\W(k)$-algebras
    \[ H^{\ast,\ast}(\Mbar_{1,1},\I^*,\bullet)\simeq \W(k)[\T,\E]/(\I\cdot \T, \E^2, \E\T) \]
    induced by mapping $\T$ to the Euler class of the dual of the Hodge line bundle in $H^{1,1}(\Mbar_{1,1},\I^*,\cal U,1)$, and $\E$ to a generator of the $\W(k)$-module $H^1(\Mbar_{1,1},\I^1,\cal O,0)$.    
  \end{prop}

  \begin{proof}
  Recall from (\ref{prop:presentation_as_quotients}) that $\Mbar_{1,1}$ is the complement of the zero section of the vector bundle $p:\scr V_{-4,-6}\ra\scr B\Gm$; in addition the restriction of the line bundle $p^*\cal U\ra\scr V_{-4,-6}$ to $\Mbar_{1,1}$ is isomorphic to the Hodge line bundle $\cal E$.
  
  As we have recalled in \ref{sub:can_loc_sequence_on_a_bundle}\ the localization sequence associated to the decomposition 
  \[s_0(\scr B\bb G_m)\hra \scr V_{-4,-6}\hla \scr V_{-4,-6}\smallsetminus s_0(\scr B\bb G_m)\]
  reads as 
  \begin{equation*} 
  \begin{tikzcd}
    \ra  H^{i-2,j-2}(\scr B\Gm,\I^\ast,\cal U^{\otimes 10}\otimes \bullet) \rar{e\cdot} &  H^{i,j}(\scr B\Gm,\I^\ast, \bullet) \rar
               \ar[draw=none]{d}[name=X, anchor=center]{}
      & H^{i,j}(\Mbar_{1,1},\I^\ast,j^\ast p^\ast\bullet) \ar[rounded corners,
              to path={ -- ([xshift=2ex]\tikztostart.east)
                        |- (X.center) \tikztonodes
                        -| ([xshift=-2ex]\tikztotarget.west)
                        -- (\tikztotarget)}]{dll}[at end]{} \\      
      \rar{e\cdot} H^{i-1,j-2}(\scr B\Gm,\I^\ast,\cal U^{\otimes 10}\otimes \bullet) & H^{i+1,j}(\scr B\Gm,\I^\ast,\bullet) \rar & H^{i+1,j}(\Mbar_{1,1},\I^\ast,j^\ast p^\ast\bullet) \ra
  \end{tikzcd}
  \end{equation*}

  where $e\cdot$ is the multiplication by the $\I^\ast$-theoretic Euler class of $\scr V_{-4,-6}$. Observe that by definition $\scr V_{-4,-6}\simeq\cal U^{\otimes 4}\oplus \cal U^{\otimes 6}$ and that Euler class of a direct sum is the product of Euler classes (\cite[pg. 14]{MR3061003}). Using the formulas for the Euler class given in \ref{subsub:Eul_class_of_Un}, we deduce
  \[e(\scr V_{-4,-6})=e(\cal U^{\otimes 4})\cdot e(\cal U^{\otimes 6}) = -2\T\HH\cdot (-3\T\HH) = 24\T^2, \]
  which is zero in $H^2(\scr B\Gm,\I^2)\simeq \ZZ/2\cdot\T^2$ (see \Cref{prop:Ij cohomology BGm}).
  
  Our computation of the Euler class implies that for each $i,j$ we have a short exact sequence
  \[0\ra H^{i,j}(\scr B\Gm,\I^\ast, \bullet) \ra H^{i,j}(\Mbar_{1,1},\I^\ast,j^*p^*\bullet) \ra  H^{i-1,j-2}(\scr B\Gm,\I^\ast,\cal U^{\otimes 10}\otimes\bullet) \ra 0.   \]
  From \Cref{prop:Ij cohomology BGm} we see that for every choice of $i,j$ and for every choice of the twist, one of the two external terms in the sequence above is zero. The additive part of the statement can then be easily deduced from the $\I^\ast$-cohomology of $\scr B\Gm$.
 
  The multiplicative description is the only thing left. Let $\E$ be a generator of $H^1(\Mbar_{1,1},\I^1)$. Then $\E\cdot\T=0$ because $H^2(\Mbar_{1,1},\I^2,\cal U)=0$. This also implies that $\E^2=0$: if this was not the case, we would have $\E^2=\T^2$, hence $\T^3=\E^2\T=0$. This gives a complete description of the multiplicative structure of $H^*(\Mbar_{1,1},\I^*,\bullet)$ and concludes the proof of the Proposition.
  \end{proof}

%\textcolor{red}{TBA:multiplicative structure}

\subsection{The Chow-Witt groups of $\Mbar_{1,1}$}
\subsection{}
Here we deal with the additive structure of $\CHW^\ast(\Mbar_{1,1},\bullet)$, using the usual description of this stack as a $\Gm$-quotient.
\begin{prop}\label{thm:additive structure Mbar}
    Let $\T$ be the Euler class of the line bundle $\cal E^\vee$, and let $\HH$ be the pull-back to $\Mbar_{1,1}$ of the element $(0,2)$ in $\CHW^0(\scr B\Gm,\cal U)$ introduced in \ref{sub:H}. 
    Moreover, let $\E'$ be any element whose image in $\CH^1(\Mbar_{1,1})$ is zero and that is sent to a generator by the boundary morphism 
    \[\partial:\CHW^1(\Mbar_{1,1},\cal O)\ra H^{0,-1}(\scr B\Gm,\KMW_\ast,\cal O)\simeq \W(k).\]
    Then the following description of $\CHW^*(\Mbar_{1,1},\bullet)$ holds:
    
    \[\begin{array}{c@{\hskip 0.17in}c@{\hskip 0.17in}c@{\hskip 0.17in}c@{\hskip 0.17in}c@{\hskip 0.17in}c@{\hskip 0.17in}c}
      \toprule
       Twist & 0 & 1 & 2 & 3 & 2k & 2k+1\\
      \midrule
       \cal O  & \GW(k)\cdot 1   & \bb Z\cdot \HH\T \oplus\W(k)\cdot\E' & \bb Z/24\cdot \T^2& \bb Z/24\cdot \HH\T^3 &\bb Z/24\cdot \T^{2k} & \bb Z/24\cdot \HH\T^{2k+1}\\
       \cal E & \bb Z\cdot\HH & \bb Z \cdot\T & \bb Z/24\cdot\HH\T^2 & \bb Z/24\cdot \T^3 & \bb Z/24\cdot\HH\T^{2k}& \bb Z/24\cdot\T^{2k+1}.\\
      \bottomrule
     \end{array}
     \]
\end{prop}
\begin{proof}

As in the proof of \Cref{prop:Ij cohomology Mbar} we have an open-closed decomposition
\[s_0(\scr B\bb G_m)\hra \scr V_{-4,-6}\hla \scr V_{-4,-6}\smallsetminus s_0(\scr B\bb G_m)\simeq \Mbar_{1,1}.\]
The induced localization sequence, combined with what we observed in \ref{sub:can_loc_sequence_on_a_bundle}, determines almost completely the additive structure. More in detail: since by definition $\scr V_{-4,-6}\simeq \cal U^{\otimes 4} \oplus \cal U^{\otimes 6}$, and because Euler class of a direct sum of vector bundles is the product of Euler classes (see \cite[Proposition 13.3.2]{FGCW}), we deduce 
\[e(\scr V_{-4,-6})=e(\cal U^{\otimes 4})e(\cal U^{\otimes 6})=-2\T\HH\cdot (-3\T\HH)=24\T^2,\]
where in the second equality we used the formulas for the Euler class given in \ref{subsub:Eul_class_of_Un}.
%As remarked in \Cref{prop:presentation_as_quotients}, we can regard $\Mbar_{1,1}$ as an open substack of vector bundle $\scr V_{-4,-6}$ over $\scr B\Gm$, with complement isomorphic to the image of the zero section. This induces the following exact sequences:
%\[ \CHW^{i-2}(\scr B\Gm) \ra \CHW^i(\scr V_{-4,-6}) \ra \CHW^i(\Mbar_{1,1}) \ra H^{i-1}(\scr B\Gm, \KMW_{i-2}) \ra H^{i+1}(\scr V_{-4,-6},\KMW_i), \]
%\[ \CHW^{i-2}(\scr B\Gm,\cal U) \to \CHW^i(\scr V_{-4,-6},p^*\cal U) \to \CHW^i(\Mbar_{1,1},\cal E) \to H^{i-1}(\scr B\Gm, \KMW_{i-2},\cal U) \to H^{i+1}(\scr V_{-4,-6},\KMW_i,p^*\cal U), \]
%where $\cal U$ is the universal line bundle on $\scr B\Gm$.
%By homotopy invariance we can identify the cohomology groups of $\scr V_{-4,-6}$ with those of $\scr B\Gm$: in this way, the pushforward along the zero section is identified with $e(\scr V_{-4,-6})$, which is
%\[ e(\scr V_{-4,-6})=e(\cal U^{\otimes (-4)})e(\cal U^{\otimes(-6)})=-2\T\HH\cdot (-3\T\HH)=24\T^2 \]
%where $\T$ is the Euler class of the universal line bundle on $\scr B\Gm$.
%Applying
Using \Cref{prop:CW ring BGm} and \Cref{prop:KMW cohomology BGm}, for $i\geq 2$ we get exact sequences:
\begin{equation*}
    \left \{
    \begin{tabular}{cl}
    $\ZZ\cdot\T^{i-2} \xrightarrow{24\T^2} \ZZ\cdot\T^i \ra \CHW^i(\Mbar_{1,1},\bullet) \ra 0$ & for $\bullet\simeq \cal E^{\otimes i}$ (modulo squares);\\
    $\ZZ \cdot \T^{i-2}\HH \xrightarrow{24\T^2} \ZZ\cdot \T^i\HH \ra \CHW^i(\Mbar_{1,1},\bullet) \ra 0$ &  for $\bullet\simeq \cal E^{\otimes i+1}$ (modulo squares).
    \end{tabular}\right.
\end{equation*}
where the zeroes on the right come from the fact that for $i\geq 1$ we have that $H^i(\scr B\Gm, \KMW_{i-1},\bullet\otimes\cal U^{\otimes 2n})$ is a subgroup of $\KM_{-1}$: this last group is zero (see \Cref{prop:KMW cohomology BGm}).
%\[ \ZZ\cdot\T^i \xrightarrow{24\T^2} \ZZ\cdot\T^i \ra \CHW^i(\Mbar_{1,1}) \ra 0, \]
%\[ \ZZ \cdot \T^i\HH \xrightarrow{24\T^2} \ZZ\cdot \T^i\HH \ra \CHW^i(\Mbar_{1,1},\cal E) \ra 0. \]
%which allows us to conclude that in this range we have
%\[\CHW^i(\Mbar_{1,1})\simeq\ZZ/24\ZZ\cdot T^i,\quad \CHW^i(\Mbar_{1,1},\cal E)\simeq\ZZ/24\ZZ\cdot\T^i\HH. \]
%A similar picture holds for $i\geq 3$ odd, giving the claimed conclusion.
%For $i=0$ we easily get
%\[\CHW^0(\Mbar_{1,1})\simeq\GW(k)\cdot 1,\quad\CHW^0(\Mbar_{1,1},\cal E)\simeq\ZZ\cdot\HH.\]
%Also the following computation is pretty straightforward:
%\[\CHW^1(\Mbar_{1,1},\cal E)\simeq\ZZ\cdot T. \]
Similarly we get 
\begin{equation*}
    \CHW^0(\Mbar_{1,1},\cal O)\simeq\GW(k)\cdot 1,\quad\CHW^0(\Mbar_{1,1},\cal E)\simeq\ZZ\cdot\HH, \quad \CHW^1(\Mbar_{1,1},\cal E)\simeq\ZZ\cdot \T.
\end{equation*}
%%On the other hand \Cref{prop:chow groups} and \Cref{prop:Ij cohomology Mbar}
We are left with identifying $\CHW^1(\Mbar_{1,1},\cal O)$. 
From the localization sequence we extract a short exact sequence
\[ 0\ra \ZZ\cdot\T\HH \ra \CHW^1(\Mbar_{1,1},\cal O)\overset{\partial}{\ra}\W(k)\ra 0.\]
We claim that this sequence splits. To prove this, observe that the Picard group of $\Mbar_{1,1}$ has no torsion. Therefore \Cref{prop:fundamental diagram} ensures that $ \CHW^1(\Mbar_{1,1},\cal O)$ fits in the pull-back square
\begin{equation}\label{eq:diagram}
\xymatrix{
\CHW^1(\Mbar_{1,1},\cal O)\ar[r] \ar[d] & \CH^1(\Mbar_{1,1}) \ar@{->>}[d] \\
H^1(\Mbar_{1,1},\I^1,\cal O) \ar[r]^{z} & \Ch^1(\Mbar_{1,1})
}
\end{equation}
%canonical map
%\[ \CHW^1(\Mbar_{1,1})\ra\Pb^1(\Mbar_{1,1})\]
%is an isomorphism. On the other hand \Cref{prop:chow groups} and \Cref{prop:Ij cohomology Mbar} imply that $\Pb^1(\Mbar_{1,1})$ fits in the cartesian diagram 
%\begin{equation}\label{eq:diagram}
%\xymatrix{
%\Pb^1(\Mbar_{1,1})\ar[r] \ar[d] & \CH^1(\Mbar_{1,1}) \ar@{->>}[d] \\
%H^1(\Mbar_{1,1},\I^1) \ar[r]^{z} & \Ch^1(\Mbar_{1,1})
%}
%\end{equation}
where $z$ is the map induced on $H^1$ by the Milnor map $\I^1 \ra \Km_1$. We claim that $z=0$.
The exact sequence 
\begin{equation}\label{eq:ses} 0\ra \I^2 \ra \I^1 \ra \Km_1 \ra 0. \end{equation}
induces a commutative ladder in cohomology

\begin{equation}\label{eq:ladder} \xymatrix{
H^{1,1}(\scr V_{-4,-6},\I^\ast,\cal O)\ar[r] \ar[d]^{j^\ast} & H^{1,1}(\scr V_{-4,-6},\Km_\ast)\ar[r] \ar[d]^{j^*}_{\simeq} & H^{2,2}(\scr V_{-4,-6},\I^\ast,\cal O) \ar[d]^{j^*}_{\simeq} \\
H^{1,1}(\Mbar_{1,1},\I^\ast,\cal O)\ar[r]^{z} & H^{1,1}(\Mbar_{1,1},\Km_\ast)\ar[r] & H^{2,2}(\Mbar_{1,1},\I^\ast,\cal O).
} \end{equation}
Observe that the central and right vertical maps are isomorphisms. Indeed, injectivity follows from the vanishing of $e(\scr V_{-4,-6})$ in $\I^\ast$-cohomology (and thus also on mod two Chow groups by \ref{sub:can_loc_sequence_on_a_bundle}), while surjectivity follows from \Cref{prop:Ij cohomology BGm}. As $\scr V_{-4,-6}$ is a vector bundle over the smooth stack $\scr B\Gm$, the top-left term is isomorphic to $H^{1,1}(\scr B\bb G_m,\I^\ast,\cal O)$ and by \Cref{prop:Ij cohomology BGm} we know that this is zero. A diagram chase shows that this implies that both horizontal arrows on the right in \eqref{eq:ladder} are injective, hence $z=0$.

From \eqref{eq:diagram} we know that $\CHW^1(\Mbar_{1,1},\cal O) \simeq \Pb^1(\Mbar_{1,1},\cal O)$. As $z=0$, the latter by definition coincides with
\[ H^1(\Mbar_{1,1},\I^\ast,\cal O)\oplus\Ker\left(\CH^1(\Mbar_{1,1})\ra\Ch^1(\Mbar_{1,1})\right), \]
which is in turn isomorphic to $\W(k)\oplus 2\T\cdot \ZZ$. This proves the splitting.

Now let $\E''$ be an element that is sent to a generator of $\W(k)$ as $\W(k)$-module. Observe that by what we proved before the element $\E''$ must be sent to an element of $\Ker\left(\CH^1(\Mbar_{1,1})\ra\Ch^1(\Mbar_{1,1})\right)$ by the morphism induced by the rank. As $\CH^1(\Mbar_{1,1})\simeq \ZZ$, this implies that the rank of $\E''$ is even, say $2d$. Therefore, if we define $\E':=:\E''-dH$, we obtain an element having the properties listed in the statement, which by construction generates the Witt factor. We are left to show that any generator of the Witt summand will be mapped along the boundary
$\partial:\CHW^1(\Mbar_{1,1},\cal O)\ra H^{0,-1}(\scr B\Gm,\KMW_\ast,\cal O)\simeq \W(k)$ to a generator. For this we use the commutative square
\begin{equation*}
  \begin{tikzcd}
     \CHW^1(\Mbar_{1,1},\cal O) \arrow[r,"\partial"] \arrow[d] & H^{0,-1}(\scr B \bb G_m, \KMW_*,\cal O) \arrow[d,"\simeq"] \\
     H^{1,1}(\Mbar_{1,1},\I^*,\cal O) \arrow[r,"\partial","\simeq"'] & H^{0,-1}(\scr B \bb G_m, \I^*,\cal O)
  \end{tikzcd}
\end{equation*}
The right vertical map is an isomorphism since it fits in the long exact sequence associated to
\begin{equation*}
\begin{tikzcd}
  0 \arrow[r] & 2\KM_\ast \arrow[r] & \KMW_\ast \arrow[r] & \I^\ast \arrow[r] & 0.
\end{tikzcd}
\end{equation*}
\end{proof}

\subsection{The Chow-Witt groups of $\Mcal_{1,1}$}\label{subsec:CW groups Mcal}
%\textcolor{red}{Nota: vengono usati gruppi di coomologia di fasci $\KMW$. Bisogna pero' stare attenti, in quanti $\cal C$ e' singolare. Sono definiti questi gruppi per varieta' singolari? Oppure bisogna utilizzare gruppi di Chow omologici e omologia del complesso C(X,-i,L)?}
%\textcolor{dark-blue}{Qui ci mettiamo la cohomologia dei complessi di Rost-Schmid . Tra l'altro sto riguardando un attimo perche' mi piacerebbe dire che e' la borel-Moore di qualcosa, ma nonsono sicuro di cosa.}

\subsubsection{}
Recall from (\ref{subec:localization sequence Mcal}) that the stack $\Mcal_{1,1}$ can be regarded as an open substack of $\scr V_{-4,-6}$, whose complement is the quotient stack $\scr C:=[C/\Gm]$, where $C$ denotes the closed subscheme in $\bb V_{-4,-6}$ of equation $4a^3+27b^2=0$.

Observe that the normalization of $\scr C$ is isomorphic to $\scr V_{-2}$: indeed, the normalization of $C$ is isomorphic to $\bb V_{-2}$, and the proper morphism $\nu:\bb V_{-2}\ra C$ is $\Gm$-equivariant, hence it descends to a well defined morphism of quotient stacks.
\begin{lemma}\label{lemma:normalization C}
    The normalization $\nu:\scr V_{-2}\ra \scr C$ induces an isomorphism of homology groups
    \[ \nu_\ast:H_{i,j}(\scr V_{-2},\op K_\ast,\nu^*\Lcal) \xrightarrow{\simeq} H_{i,j}(\scr C,\op K_\ast,\Lcal) \]
    where $\op K_\ast \in \{\KMW_\ast,\I^\ast, \KM_\ast,\Km_\ast\}$.
\end{lemma}
\begin{proof}
    The restriction of $\nu$ to the open substack $\scr B\mu_2=[\bb V_{-2}\smallsetminus\{0\}/\Gm]$ is an isomorphism; the same holds for the restriction to the closed complement $\scr B\Gm=[\{0\}/\Gm]$. The conclusion then follows from functoriality of localization sequences along proper maps.
\end{proof}

The Picard group of $\Mcal_{1,1}$ is isomorphic to $\ZZ/12$ and it is generated by the Hodge line bundle $\cal E$. In fact the discriminant is a never-vanishing section of $\cal E^{\otimes 12}$, hence it canonically induces a map
\[ \Delta: \Mcal_{1,1} \ra \scr B\mu_{12}.\]
The map $\Delta$ can also be constructed as the quotient of the $\Gm$-equivariant morphism
\[ \Delta: \bb V_{-4,-6}\smallsetminus C \ra \bb V_{-12}\smallsetminus\{0\},\quad (a,b)\longmapsto 4a^3+27b^2. \]
As the map $\Delta$ is flat between smooth quotient stacks, there is a well defined pullback $\Delta^*$ at the level of Chow-Witt rings (see \Cref{thm:equiv prop_verison2_stacks}).

\begin{defin}\label{defin:class d}
    We define $\D$ in $\CHW^0(\Mcal_{1,1},\cal O,0)$ as $\Delta^*\U$, where $\U$ is the element of $\CHW^0(\scr B\mu_{12})$ introduced in \Cref{defin:class u}.
\end{defin}

\subsubsection{}\label{sub:new class Mcal}
If we denote by $\bf{GW}$ the unramified sheaf of Grothendieck-Witt rings on the category of smooth $k$-schemes, we can regard the element $\D$ as a map $\Mcal_{1,1}\ra \bf{GW}$ of Zariski stacks over smooth $k$-schemes associating with every elliptic curve $p:E\ra S$ the Grothendieck-Witt class of the pair $(\cal E_p, \Delta)$: here we see the discriminant $\Delta$ as a quadratic form $\cal E_p^{\otimes 6}\otimes \cal E_p^{\otimes 6} \ra \cal O$ on the sixth power of the Hodge bundle $\cal E_p$. From this point of view we actually have that $\D$ is the following composition of maps of Zariski stacks
\[\xymatrix{\Mcal_{1,1} \ar[r]^{\Delta} & \scr B\mu_{12} \ar[r]^{(-)^6} & \scr B\mu_2 \ar[r]^{\U} & \bf{GW}.}\]
\begin{prop}\label{thm:additive structure Mcal}
Suppose that the characteristic of the base field is $\neq 2,3$. Let $\T$ be the Euler class of $\cal E^{\vee}$, let $\HH$ be the (pull-back of the) class introduced in from \ref{sub:H} and let $\D$ be the class introduced in \Cref{defin:class d}. Then the following description of $\CHW^\ast(\Mcal_{1,1},\bullet)$ holds:

    \[\begin{array}{c@{\hskip 0.17in}c@{\hskip 0.17in}c@{\hskip 0.17in}c@{\hskip 0.17in}c@{\hskip 0.17in}c@{\hskip 0.17in}c}
      \toprule
      Twist  & 0 & 1 & 2 & 3 & 2k & 2k+1\\
      \midrule
       \cal O  & \GW(k)\cdot 1\oplus \W(k)\cdot \D& \bb Z/6\cdot \HH\T & \bb Z/24\cdot \T^2& \bb Z/6\cdot \HH\T^3 &\bb Z/24\cdot \T^{2k} & \bb Z/6\cdot \HH\T^{2k+1}\\
       \cal E & \bb Z\cdot\HH & \bb Z/24\cdot\T & \bb Z/6\cdot\HH\T^2 & \bb Z/24\cdot \T^3 & \bb Z/6\cdot\HH\T^{2k}& \bb Z/24\cdot\T^{2k+1}.\\
      \bottomrule
     \end{array}
     \]
\end{prop}

\begin{proof}
We use the localization sequence associated to $\scr C \overset{i}{\hra} \scr V_{-4,-6} \overset{j}{\hla} \Mcal_{1,1} $  for the homology with coefficients in Milnor-Witt K-theory: %we get the following exact sequence
\begin{equation}
\label{eq:long M11} 
\xymatrix{\CHW_{i}(\scr C,\bullet)\ar[r]^(0.45){i_\ast} & \CHW_i(\scr V_{-4,-6},\bullet) \ar[r]^{j^\ast} & \CHW_i(\Mcal_{1,1},\bullet)\ar[r] & H_{i-1,i}(\scr C,\KMW_\ast,\bullet) }
\end{equation}
%\[ \CHW_{i}(\cal C,\iota^*\cal U)\ra \CHW^i(\scr V_{-4,-6},\cal U) \ra \CHW^i(\Mcal_{1,1},\cal E)\ra H^i(\cal C,\KMW_{i-1},\iota^*\cal U) \ra H^{i+1}(\cal U_{-2},\KMW_i,\cal U). \]
From \Cref{lemma:normalization C}, combined with the usual homotopy invariance and with the computations contained in \Cref{prop:KMW cohomology BGm}, we know that $H_{i-1,i}(\scr C,\KMW_\ast,\bullet)=0$ for $i\neq 1$.
 
We are left with determining the image of $i_\ast$: we claim that $\rm{Im}(i_\ast)$ is the graded ideal generated by $e(\cal O(12))$. By \Cref{lemma:normalization C} we know that $\rm{Im}(i_\ast)$ coincides with the image of the puhforward along the composition
\[\scr V_{-2} \overset{\nu}{\ra} \scr C \overset{i}{\hra} \scr V_{-4,-6}.\]
 
Since $(i\circ \nu)^\ast$ is a surjective map of $\CHW^\ast(\scr B\Gm,\bullet)$-modules, the projection formula ensures that $\rm{Im}(i\circ \nu)_\ast=((i\circ \nu)_\ast(1))$. Moreover it is clear that $(i\circ \nu)_\ast(1)=[\scr C]$ where $[\scr{C}]$ denotes the fundamental class of the Cartier divisor $\scr C \subset \scr V_{-4,-6}$. It is known that
\[ [\scr C]=e(\cal O(\cal C))=e(\cal U^{\otimes 12})=-6\T\HH, \]
and this concludes the computation of the Chow-Witt groups of $\Mcal_{1,1}$ in cohomological degree $i>0$. Actually, the same argument also shows that $\CHW^0(\Mcal_{1,1},\cal E)=\ZZ\cdot\HH$.

In order to determine $\CHW^0(\Mcal_{1,1})$ we use again the localization sequence \eqref{eq:long M11}, from which we deduce the following short exact sequence:
\begin{equation}\label{eq:ses CW groups Mcal} 0\ra \GW(k)\cdot 1 \ra \CHW^0(\cal M_{1,1},\cal O) \ra \W(k) \ra 0. \end{equation}
By construction $h\cdot\D=\Delta^*h \cdot \Delta^*U=\Delta^*(h\cdot\U)=0$, hence the multiplication by $\D$ defines a map $ \W(k) \ra \CHW^0(\Mcal_{1,1})$. We claim that this map splits the sequence \eqref{eq:ses CW groups Mcal}: for this we need to show that the boundary of $\D$ in  $H_{0,1}(\scr C,\KMW_\ast,\cal U^{\otimes 12})\simeq \W(k)$
is a unit.

Consider the equivariant approximation of $\scr V_{-4,-6}\ra\scr B\Gm$ given by the vector bundle 
\[V:=\mathbb V(\cal O(-4)\oplus\cal O(-6))\ra\bb P^2.\]
An equivariant approximation of $\scr C$ is given by the divisor $C\subset V$ defined as the vanishing locus of the map
\[ \bb V \big (\cal O(-4)\oplus\cal O(-6)\big ) \ra \bb V \big (\cal O(-12)\big ),\quad (a,b) \longmapsto 4a^{\otimes 3} + 27b^{\otimes 2}.  \]
Therefore, the open subscheme $V\smallsetminus C$ is an equivariant approximation of $\Mcal_{1,1}$. The class $\D$ can then be regarded as the element of $\CHW^0(V\smallsetminus C,\cal O)$, or even more conveniently as an element of $\GW(k(V))$. Let $\cal L$ denote the pull-back of $\cal O(-6)$ to $V$, then we have the quadratic form
\[ q:\cal L \otimes \cal L \ra \cal O_V,\quad t\otimes t' \longmapsto \frac{t\otimes t'}{4a^{\otimes 3}+27b^{\otimes 2}}. \]
If we write $k(V)=k(x,y,a,b)$, then $q_{gen}\in \GW(k(V))$ corresponds to the symbol $1+\eta[\Delta]$, where $\Delta=4a^3+27b^2$.
By construction $\D=(q_{gen}-1)$, and its image along the boundary morphism
\[ \CHW_1 (V\smallsetminus C,\cal O)\ra H_{0,1}(C,\KMW_\ast,\cal O(C)) \]
is equal to the residue of this element at the codimension one point corresponding to the generic point of $C$. The discriminant $\Delta=4a^3+27b^2$ is a local parameter for the valuation $\nu:k(V)\ra \ZZ\cup\{\infty\}$ induced by the Cartier divisor $C\subset V$. Applying the formula for the residues (\cite[Theorem 1.7]{FasLectures}) we obtain
\[ \partial_\nu(\eta[\Delta^{-1}])=\partial_\nu^{\Delta}(\eta[\Delta^{-1}])\otimes \overline{\Delta}^{\vee}=\eta\cdot\partial_\nu [\Delta^{-1}] \otimes \overline{\Delta}^{\vee}= \eta\otimes \overline{\Delta}^{\vee}.\]
The computation above shows that, after identifying $H_{0,1}(C,\KMW_\ast,\cal O(C))$ with $\W(k)$, the boundary of $\D$ is a unit and hence the multiplication by this class splits \eqref{eq:ses CW groups Mcal}. This concludes the proof.

\end{proof}

\subsection{Multiplicative structures}\label{sec:multiplicative structures}
\subsubsection{}
Let us take a look at the localization sequence of $\KMW_\ast$-cohomology groups induced by the open immersion $\Mcal_{1,1}\hra\Mbar_{1,1}$. Recall that $\Delta_0$, the divisor of singular curves, is isomorphic to $\scr B\mu_2$, so that we have
\[ \CHW^0(\scr B\mu_2,\cal O)\ra \CHW^1(\Mbar_{1,1},\cal E^{-12}) \ra \CHW^1(\Mcal_{1,1}) \ra 0. \]
The surjectivity of the last map is due to the vanishing of $H^1(\scr B\mu_2,\KMW_0)$, which is in turn an easy consequence of the vanishing of $H^{1,j}(\scr B\Gm,\KMW_\ast)$ for $j=0,-1$.

We can rewrite the exact sequence above using the information we gathered so far on the Chow-Witt groups of these stacks (see \Cref{prop:CW groups Bmu2n}, \Cref{thm:additive structure Mbar} and \Cref{thm:additive structure Mcal}). We obtain:
\[ \GW(k)\cdot 1 \oplus \W(k)\cdot \U \ra \ZZ\cdot\T\HH\oplus \W(k)\cdot\E' \ra \ZZ/6\cdot \T\HH \ra 0.  \]
As $\cal O(\Delta_0)\simeq \cal E^{\otimes 12}$ we get
\begin{equation}\label{eq:6TH} i_*1=e(\cal O(\Delta_0))=e(\cal E^{\otimes 12})=-6\T\HH, \end{equation}
which implies that $i_*\U\neq 0$. Write $i_*\U=a\T\HH+b\E'$. The commutative square
\[\xymatrix{
\CHW^0(\scr B\mu_2)\ar[r]^{i_*} \ar[d]^{ch} & \CHW^1(\Mbar_{1,1}) \ar[d]^{ch} \\
\CH^0(\scr B\mu_2) \ar[r]^{i_*} & \CH^1(\Mbar_{1,1})\simeq\ZZ
} \]
tells us that $ch(i_*\U)=i_*(ch(\U))$, which is zero because of how we defined $\U$ (see \Cref{defin:class u}). Henceforth $i_*\U$ must be of the form $b\E'$. Consider now the commutative square 
\[\xymatrix{
\CHW^1(\Mbar_{1,1})\ar[r]^{j^*\quad\quad} \ar[d]^{ch} & \CHW^1(\Mcal_{1,1})\simeq\ZZ/6\cdot\T\HH \ar[d]^{ch} \\
\CH^1(\Mbar_{1,1}) \ar[r]^{j^*\quad\quad} & \CH^1(\Mcal_{1,1})\simeq\ZZ/12\cdot \T.
} \]
Observe that the right vertical arrow sends $\T\HH$ to $2\T$, hence the map is injective. As $\E'$ is by definition sent to zero by the left vertical arrow, we deduce that $\E'$ is sent to zero by the top horizontal arrow, i.e. $\W(k)\cdot \E'$ is contained in the kernel of $ \ZZ\cdot\T\HH\oplus \W(k)\cdot\E' \ra \ZZ/6\cdot \T\HH$.

On the other hand, this kernel coincides with the image of  $\GW(k)\cdot 1 \oplus \W(k)\cdot \U \ra \ZZ\cdot\T\HH\oplus \W(k)\cdot\E'$, which is generated by $i_*(1)=6\T\HH$ and $i_*(U)=b\E'$. This implies that $b$ must be invertible in $\W(k)$, as the $\GW(k)$-module generated by $b\E'$ must be equal to the whole $\W(k)\cdot\E'$.
\begin{defin}\label{defin:class E}
    We denote by $\E$ the class $i_*\U$.
\end{defin}
We are ready to state our first main result.
\begin{thm}\label{thm:CW ring of Mbar}
Let $k$ be a perfect field of characteristic different from $2$ and $3$. Then we have an isomorphism of $\GW(k)$-algebras
\[ \CHW^*(\Mbar_{1,1},\bullet)\simeq\GW(k)[\T,\E,\HH]/(I\cdot\T,I\cdot\HH,\HH^2-2h,h\E,\HH\E,\E^2,24\T^2,12\HH\T^2+\E\T), \]
where $\T$ is mapped to the Euler class of $\cal E^{\vee}$, the element $\HH$ is mapped to the (pull-back of the) element introduced in \ref{sub:H}, and $\E$ is mapped to the class introduced in \Cref{defin:class E}.
\end{thm}
\begin{proof}
The relations that only involve $\T$ and $\HH$ comes from the Chow-Witt ring of $\scr B\Gm$ (see \Cref{prop:CW ring BGm}). The relation $24\T^2=0$ has already been discussed in the proof of \Cref{thm:additive structure Mbar}. 

By definition $\E=i_*\U$, and we know from \Cref{thm:CW ring of Bmu2n} that $\U^2=0$, $h\U=0$ and $\U\T=2\T$. The push-forward morphism $i_*$ is a morphism of $\CHW^*(\scr B\Gm,\bullet)$-modules and we have already verified in (\ref{eq:6TH}) that $i_*(1)=6\T\HH$, thus by the projection formula (cf. Theorem 3.19 of \cite{FasLectures}) we obtain
\begin{align*}
\E\T=i_*(\U)\cdot\T=i_*(\U\cdot i^*\T)=i_*(\U\T)=i_*(2\T)=2\T\cdot i_*(1)=-12\T^2\HH
\end{align*}
In a similar way we deduce that $h\E=0$ and $\HH\E=0$. To prove that $\E^2=0$, is enough to observe that by construction the image of $\E^2$ in $\CH^2(\Mbar_{1,1})$ is zero. On the other hand the group $\CHW^2(\Mbar_{1,1})$ is generated by $\T^2$ and the morphism $\CHW^2(\Mbar_{1,1})\ra \CH^2(\Mbar_{1,1})$ is an isomorphism (see \Cref{thm:additive structure Mbar}), hence $\E^2=0$. 

The $\GW(k)$-algebra generated by $\T$,$\HH$ and $\E$ modulo the relations that we have found maps by construction onto the Chow-Witt ring of $\Mbar_{1,1}$. By looking at the additive structure of these two rings, we deduce that this map must be an isomorphism, thus concluding the proof.
\end{proof}

\begin{thm}\label{thm:CW ring of Mcal}
Suppose that the characteristic of the base field is $\neq 2,3$. Then we have an isomorphism of $\GW(k)$-algebras
\[ \CHW^*(\Mcal_{1,1},\bullet)\simeq\GW(k)[\T,\D,\HH]/(I\cdot\T,I\cdot\HH,\HH^2-2h,h\D,\HH\D,\D^2+2\D,6\T\HH,12\T-\D\T), \]
where $\T$ is the Euler class of $\cal E^{\vee}$, the element $\HH$ is the pull-back of the class introduced in \ref{sub:H}, and $\D$ is the class introduced in \Cref{defin:class d}.
\end{thm}
\begin{proof}
The discriminant morphism $\Mcal_{1,1}\ra\scr B\mu_{12}$ induces a pull-back morphism at the level of Chow-Witt rings. By looking at the additive structure of both rings, which we know from \Cref{thm:CW ring of Bmu2n} and \Cref{thm:additive structure Mcal}, we can conclude that the pull-back morphism is actually an isomorphism of rings. This concludes the proof.
\end{proof}

\subsection{Geometric interpretation of the new classes}
\label{subsec:inter}
\subsubsection{}
Recall that for a smooth $k$-variety $X$ and a line bundle $\cal L$ on $X$, the Chow-Witt group $\CHW^i(X,\cal L)$ is a subquotient of
\[\bigoplus_{x\in X^{(i)}} \GW(k(x),\det(\fk m_x/\fk m_x^2)^\vee\otimes_{k(x)} \cal L),\]
where $\fk m_x/\fk m_x^2$ denotes the cotangent space of $X$ at $x$. Each element of $\CHW^i(X,\cal L)$ thus writes as a linear combination \[\sum_{x\in X^{(i)}} q_x\otimes l_x [x],\] where $q_x$ is a class in $\GW(k(x))$, $l_x$ is a local generator of $\det (\fk m_x/\fk m_x^2)^\vee \otimes_{k(x)}\cal L$. Sometimes, with abuse, we index the sum on integral subvarieties of $X$.

\subsubsection{}
\label{subsub:torsor_of_nodes}
Observe that $\E \in \CHW^1(\Mbar_{1,1},\cal O)$ belongs to, and in fact generates multiplicatively, the kernel of the natural morphism
\[\CHW^\ast(\Mbar_{1,1},\bullet) \ra \CH^\ast(\Mbar_{1,1})\]
which, we remind, sends an element $\sum_i q_i\otimes l_i [V_i]$ to $\sum_i {\rm{rk}}(q_i)[V_i]$.

The construction of the element $\E$ is somehow indirect. We have first identified $\phi: \Delta_0\overset{\simeq}{\ra} \scr B\mu_2$ to get a class in $\CHW^0(\Delta_0,\cal O)$ corresponding to $\U$, and in a second step we have used the composition of the push-forward map with a non-canonical identification $g$
\begin{equation*}
    \xymatrix{
    \CHW^0(\Delta_0,\cal O) \ar[r]_(0.4){\iota_\ast} &  \CHW^1(\Mbar_{1,1},\cal E^{\otimes-12}) \ar[r]^{\simeq}_{g} & \CHW^1(\Mbar_{1,1},\cal O),\\
    }
\end{equation*}
where $\iota: \Delta_0 \hra \Mbar_{1,1}$ is the obvious closed embedding, and where $g$ corresponds to a family of choices of local generators $g_x$ for $\cal E^{\otimes -6}$. The isomorphism $\phi$ is actually the composition 
\[\xymatrix{\Delta_0   \ar[r]^(0.4){\simeq}_(0.4)\psi &  \scr V_{-2}\setminus s_0 \ar[r]^{\simeq}_{\phi_2}&  \scr B\mu_2:}\]
here $\psi$ is induced by the equivariant isomorphism
\[\spec k[a,b]/(4a^3+27b^2)\setminus \{(0,0)\} \ra \spec k[t]\setminus \{0\} \; \textrm{ defined by }\; t \mapsto b/a\]
(where $\Gm$ acts on $a, b, t$ with weights $-4,-6,-2$ respectively), while $\phi_2$ is the tautological isomorphism introduced in \ref{sub:line_bundles_on_BGm_and_their_torsors}. In this picture $\cal E_{|\Delta_0}$ is easily seen to be induced by the pull-back of the universal line bundle $\cal U$ on $\scr B\mu_2$ via $\phi$. In particular, since the coordinate $t$ is the canonical trivialization of the pull-back of $\phi_2^\ast\cal U^{\otimes 2}$, the section $b/a$ trivializes $\phi^\ast \cal U^{\otimes 2}$, and thus induces a trivialization 
\[b/a:\cal O \overset{\simeq}{\ra} \cal E_{|\Delta_0}^{\otimes 2}.\]
This induces a quadratic form on $\cal E_{|\Delta_0}$, and thus an element $q=\phi^\ast(q_{gen})\in \CHW^0(\Delta_0,\cal O)$. By construction, the element $\E$ coincides with $(g\circ \iota_\ast)(q-1)=g\big (\kc{b/a}-1)[\Delta_0]\big )$.

More concretely we can think of the element $\E$ as a sort of characteristic class, functorially associated with families $p: C\ra S$ of genus $1$ curves over smooth bases, marked with a section $\sigma: S\ra C$ (landing in the smooth locus of $p$), where the non-smoothness locus of $p$ is a smooth Cartier divisor $S_\infty\subseteq S$. For instance, assume $S$ to be a smooth curve and let $f:S\ra \Mbar_{1,1}$ be the classifying map of the family $p$. Then $S_\infty=s_1+\dots+s_m$ is a reduced Cartier divisor arising as pull-back of $\Delta_0$, and the induced map $f_\infty: S_\infty \ra \Delta_0$ classifies the induced family of nodal curves $p_\infty: C_\infty \ra S_\infty$. In this setting
\begin{equation}\label{eq:class_of_nodal_locus}
  f_\infty^\ast\U\in \CHW^0(S_\infty,\cal O)\simeq \bigoplus _i \GW(k(s_i))
\end{equation}
and the above arguments tautologically give 
\[f_\infty^\ast\U=\sum_i \left(\left\langle \frac{b_i}{a_i}\right\rangle - 1\right):\]
here $a_i,b_i$ are coefficients of any Weierstrass equation $y^2=x^3+a_ix+b_i$ for the fiber $C_{s_i}$. As a sanity check we note that another Weierstrass equation for $C_{s_i}$ would change the ratio $b_i/a_i\in k(s_i)^\times$ by a square, leaving the cycle $f_\infty^\ast\U$ unchanged. Observe that $f_\infty^\ast\U$ is nothing else than the "new" characteristic class of the $\mu_2$-torsor
\begin{equation}\label{eq:torsor_of_a_nodal_curve}
\underline{\rm{Aut}}_{S_\infty}(\cal E_{p|S_\infty},f_\infty^\ast(b/a))
\end{equation} 
of automorphisms of the Hodge bundle $\cal E_{p_\infty}(\simeq \cal E_{p|S_\infty})$ of $p_\infty$ respecting the induced trivialization $f_\infty^\ast b/a$ of its tensor square. 

A side note: nodal curves are completely classified by their $\mu_2$-torsor \eqref{eq:torsor_of_a_nodal_curve}; $\mu_2$-torsors over a smooth base are, on their turn, completely classified by their discriminant $\U$ in the $\cal O$-twisted Chow-Witt ring of their base. Thus the invariant $\U$ is a full invariant of families of nodal curves.

Going back to the main topic we see that, by base change (cf. Theorem 3.18 of \cite{FasLectures}), formula \eqref{eq:class_of_nodal_locus} gives that $f^*\E\in \CHW^1(S,\cal E_{p}^{\otimes -12})$ can be expressed as
  \[f^*\E = \sum_i \left(\left\langle \frac{b_i}{a_i}\right\rangle - 1\right)[s_i],\]
where $\cal E_p$ denotes the Hodge bundle of the family $p$.

\subsubsection{}
\label{subsub:torsors_of_tangents}
We can relate the element $f^*\E$ to the geometry of the tangent lines around the nodes of the singular fibres as follows. Let $\nu_i:\overline{C}_{s_i}\ra C_{s_i}$ be the normalization map, and let $P_i$ be the fiber of $\nu_i$ over the singular point of $C_i$. It is easy to check that $P_i\ra s_i$ is an \'{e}tale cover (of degree $2$) isomorphic to $k(s_i)(\sqrt{b_i/2a_i})$: this follows from writing down explicitly the equation for the normalized curve. In other words, the tangent lines at the node of $C_{s_i}$ are rational if and only if $b_i/2a_i$ is a square in $k(s_i)$.

We can further reformulate of this phenomenon via a chain of natural isomorphisms of $\mu_2$-torsors over $s_i$ hereby displayed:
\begin{equation}
    \label{eq:torsor_of_tangents}
    \xymatrix{P_i   & \underline{\rm{Isom}}_{s_i}\left(\{0,1\}, P_i\right)  \ar[l]_(0.7)\simeq  & \underline{\rm{Isom}}_{s_i}\big( (\bb P^1_{s_i},\{0,1\},\infty),(\overline{C}_{s_i}, P_i, \sigma(s_i)\big). \ar[l]_(0.65)\simeq}
\end{equation}
The source of the left map is the the $\mu_2$-torsor of isomorphisms of $\mu_2$-torsors over $s_i$, and isomorphically maps to $P_i$ by evaluating at $0$. The source of the right map is the $\mu_2$-torsor of isomorphisms of curves marked with a Cartier divisor of degree two and with a Cartier divisor of degree one; the right hand side map is simply the restriction to the divisor of degree two. The ratio $b_i/2a_i$ being a square is thus equivalent to the triviality of the $\mu_2$-torsor 
\[\underline{\rm{Isom}}_{s_i}\big( (\bb P^1_{s_i},\{0,1\},\infty),(\overline{C}_{s_i}, P_i, \sigma(s_i)\big)\ra s_i.\]
Since we further have
\[\underline{\rm{Isom}}_{s_i}\big( (\bb P^1_{s_i},\{0,1\},\infty),(\overline{C}_{s_i}, P_i, \sigma(s_i)\big)\simeq \underline{\rm{Isom}}_{s_i}\big( (\bb P^1_{s_i}/0\sim 1, 2\{0\sim 1\},\infty),(C_{s_i}, P_i, \sigma(s_i)\big),\]
the invariant $\langle b_i/2a_i \rangle$ checks whether the curve $C_{s_i}$ is isomorphic (over $\spec (k(s_i))$) to a projective line with $0$ and $1$ glued together or to a quadratic twist of it. Obviously when $2$ is a square in $k$, the torsors \eqref{eq:torsor_of_a_nodal_curve} and \eqref{eq:torsor_of_tangents} are isomorphic, thus they have the same associated Chow-Witt class in $\CHW^1(S,\cal E^{\otimes -12})$. In particular, we have
        \[ \langle 2 \rangle f^*\E= \sum_i \left(\left\langle \frac{b_i}{2a_i} \right\rangle - \left\langle 2 \right\rangle \right) [s_i]. \]
When the ground field has a root of $2$, the expression above is equal to $f^*\E$.

\appendix
  \section{A gentle introduction to Chow-Witt theory}\label{sec:poor}
The goal of this Section is to give an elementary motivation and guide to Chow-Witt groups to a reader who is acquainted with classical intersection theory but not with the machinery of motivic cohomologies. Therefore, in this Section we sketch an intuitive construction of Chow-Witt groups that resembles the one of Chow groups. A rigorous fully detailed treatment of foundations is given in \cite{FasLectures} and in cite \cite{FGCW}. 

We have selected a few key facts without claiming or hoping to be exhaustive. The interested reader should refer to \cite{FasLectures} for further explanations and proofs.

\subsection{Recap on Chow groups}
\subsubsection{} The standard reference for the definition of Chow groups is \cite[Chapter 1]{Ful}.
Given a separated scheme $X$ over a field $k$, the group of cycles on $X$ of dimension $i$ is defined as
\[ Z_i(X):=\langle [V] | V\subset X \text{ is a closed subvariety of dimension }i\rangle. \]
Therefore, a cycle of dimension $i$ on $X$ is a finite formal sum $\sum_j n_j[V_j]$, where each $V_j$ is a subvariety of dimension $i$, and $n_j$ is an integer.

Given a normal subvariety $W\subset X$ and a rational function $\varphi:W\ra\bb P^1$ on $W$, one can define a cycle in $X$ as follows:
\[\rm{div}(\varphi):=\varphi^{-1}(0) - \varphi^{-1}(\infty). \]
When $W$ is not normal, more carefulness is needed and one should rather work with the normalization of $W$. For more details, see \cite[Chapter 1]{Ful}.

The definition of $\rm{div}(\varphi)$ above is necessary in order to introduce the notion of rational equivalence of cycles.
Indeed, two cycles $[V]$ and $[V']$ of dimension $i$ are called rationally equivalent if there exist a finite number of subvarieties $W_1,\dots,W_m \subset X$ of dimension $i+1$ and rational functions $\varphi_1,\dots,\varphi_m$ with $\varphi_i\in k(W_i)$ such that
\[ [V]-[V'] = \sum_j \rm{div}(\varphi_j). \]
The intuition behind the notion of rational equivalence is that two cycles are rationally equivalent if one can be deformed into the other along a chain of projective lines.

The Chow group $\CH_i(X)$ of cycles of dimension $i$ is then defined as
\[ \CH_i(X):=Z_i(X)/\sim_{rat}. \]

\subsubsection{}
The Chow group $\CH_i(X)$ sits in the exact sequence
\begin{equation}\label{eq:ex seq for chow} \bigoplus_{W\in X_{(i+1)}} k(W)^* \ra \bigoplus_{V \in X_{(i)}} \ZZ\cdot [V] \ra \CH_i(X) \ra 0, \end{equation}
where the first arrow sends rational functions $\varphi$ to $\rm{div}(\varphi)$. This is just a reformulation of the construction presented in the previous paragraph.

Observe that we can further reformulate the exact sequence above in terms of Milnor K-theory: recall that for a field $F$ we have
$\KM_0(F)=\ZZ$ and $\KM_1(F)=F^*$, and therefore we can rewrite (\ref{eq:ex seq for chow}) as
\[ \bigoplus_{W\in X_{(i+1)}} \KM_1(k(W)) \ra \bigoplus_{V \in X_{(i)}} \KM_0(k(V)) \ra \CH_i(X) \ra 0,  \]
where the first arrow is the total residue homomorphism. This way of looking at Chow groups is the starting point to get to the definition of Chow-Witt groups; more on that later.

When $X$ is smooth, Chow groups inherit a multiplicative structure given by the intersection product. In particular, such multiplicative structure is well behaved with respect to the natural multiplicative structure of $\KM_0(F)\simeq\ZZ$, i.e. we have
\[ n[V] \cdot m[V] = nm([V]\cdot [V']). \]
An analogue statement will be true for Chow-Witt groups, as we will see in the next Subsection.
\subsection{Intuition for Chow-Witt groups}
\subsubsection{}
A nice reference for the contents of this subsection is \cite[Section 1]{FasLectures}.
We have just seen that the generators of classical Chow groups are cycles $\sum_j n_j[V_j]$, where the $n_j$ are \emph{integers} and the $V_j$ are subvarieties of $X$. The idea behind Chow-Witt groups is to consider cycles $\sum_j q_j[V_j]$ where we would like the coefficients $q_j$ to be \emph{quadratic forms}. 

We should also be able to subtract and multiply two coefficients together, just as in the classical case. It would also be useful to have an extension of the rank function $q\mapsto\rm{rk}(q)$ to these groups of cycles, so to be able to get back to the classical cycles by considering
\[ \sum_j q_j[V_j] \longmapsto \sum_j \rm{rk}(q_j)[V_j]. \]
Given a cycle $q[V]$, an obvious choice for field of definition of the quadratic form $q$ is given by the field $k(V)$. Quite unsurprisingly at this point, the ring of quadratic forms we are looking for turns out to be the Grothendieck-Witt ring $\GW(k(V))$. Recall indeed that given a field $F$, we can construct a monoid of isometry classes of symmetric bilinear spaces $(V,q)$,  i.e. pairs where $V$ is a finite dimensional $F$-vector space and $q$ is a symmetric bilinear form on $V$. Operations are defined to be orthogonal direct sum and tensor product. The Grothendieck-Witt ring $\GW(F)$ is then the group completion of this monoid.

One would be tempted to define the Chow-Witt group of $i$-dimensional cycles on a variety $X$ as a suitable quotient of
\[ \bigoplus_{V\in X_{(i)}} \GW(k(V)). \]
However for several reasons this is a bit naive: for instance this definition does not allow to canonically define an analogue of rational equivalence, as it will be clear soon. This group has to be slightly modified in order to give a reasonable theory. 

\subsubsection{}
Another strong evidence that the Grothendieck-Witt ring is the right object to consider is given by the fact that $\GW(F)$ is isomorphic to the degree $0$ piece of the so called Milnor-Witt K-theory of $F$. The Milnor-Witt K-theory graded ring $\KMW_\ast(F)$ is generated as a graded ring by the degree $1$ symbols $[a]$, where $a\in F^{*}$, and the degree $-1$ symbol $\eta$. The relations are the following:
\begin{itemize}
    \item $[a]\cdot [1-a]=0$ for $a\neq 1$,
    \item $[ab]=[a]+[b]+\eta[a][b]$,
    \item $\eta[a]=[a]\eta$ for $a\in F^{*}$,
    \item $\eta(\eta[-1]+2)$.
\end{itemize}
An element $[a_1][a_2]\cdots [a_n]$ is usually denoted $[a_1,a_2,\dots,a_n]$. There is a natural isomorphism
\[ \GW(F) \ra \KMW_0(F) \]
defined by
\[ \langle a \rangle \longmapsto 1+\eta[a]. \]

Another interesting fact about $\KMW_\ast(F)$ is that the pieces of negative degree are all canonically isomorphic to $\W(F)$, the Witt ring of $F$. Observe that in $\KMW_0(F)$ the hyperbolic form $h$ is written as $2+\eta[-1]$, and the last relation assures us that $\eta h=0$. Then the aforementioned isomorphism is given by
\[\W(F)\ra \KMW_{-i}(F),\quad \langle a\rangle\longmapsto \eta^i(1+\eta[a]). \]
\subsubsection{}
The next step is to introduce an equivalence relation on the cycles. Ideally, we would like this relation to be well behaved with respect to the rank function, so that
\[ q[V] \sim q'[V'] \Rightarrow \rm{rk}(q)[V] \sim_{rat} \rm{rk}(q')[V']. \]
Recall that in the classical case, the cycles that are rationally equivalent to zero belongs to the image of the total residue homomorphism
\[\bigoplus_{W\in X_{(i+1)}} \KM_1(k(W)) \ra \bigoplus_{V \in X_{(i)}} \KM_0(k(V)).\]
Let $V\subset W$ be a codimension $1$ subscheme (we are assuming $W$ normal). Then the field $k(W)$ has a discrete valuation $\nu$ with residue field $k(V)$. Pick a generator $\pi$ for the kernel of $\nu$, and let $\Ocal_{W,V}\subset k(W)$ be the local ring of $W$ at $V$. Then the corresponding residue homomorphism is determined by the properties
\[ \partial^{\pi}_{\nu} [\pi] = 1, \quad \partial^{\pi}_{\nu} [u]= 0\text{ for }u\in \Ocal_{W,V}^{*}. \]

To define an equivalence relation for cycles with coefficients in quadratic forms it is quite natural to look at $\KMW_1(-)$ and try to replicate the construction above, i.e. to try to define an homomorphism
\[\bigoplus_{W\in X_{(i+1)}} \KMW_1(k(W)) \ra \bigoplus_{V \in X_{(i)}} \KMW_0(k(V)). \] 
The group $\KMW_1(F)$ is generated by the symbols $\eta^{j-1}[a_1,\dots,a_j]$, where $a_i$ is in $F^*$. Given $W\subset V$ of codimension $1$, pick a parameter $\pi$ for the kernel of the valuation $\nu:k(W)\ra \ZZ\cup\{-\infty\}$. Then there is a unique homomorphism
\[ \partial^{\pi}_{\nu}: \KMW_1(k(W)) \ra \KMW_0(k(V)) \]
which satisfies
\[ \partial^{\pi}_{\nu} \eta^{j-1}[\pi,u_2,\dots,u_j] = \eta^{j-1} [\overline{u_2},\dots,\overline{u_j}]. \]
\subsubsection{}
There is a subtlety here that must not be overlooked: the homomorphism $\partial_{\nu}^{\pi}$ is not independent of the choice of the parameter $\pi$, i.e. a different choice of generator would determine a different residue homomorphism. This is a completely new feature of Milnor-Witt K-theory which is not present in Milnor K-theory. In order to make the residue homomorphism canonical, one needs to keep track of the choices of uniformizers made for each specialization. This leads to the introduction of twisted Milnor-Witt K-theory.

\subsection{Twisted Chow-Witt groups}
\subsubsection{}
If $\cal L$ is a line bundle on $\spec{F}$, we denote $\cal L^*$ the set of non-zero elements of $\cal L$. Both $\KMW_\ast(F)$ and $\ZZ[\cal L^*]$ are naturally $\ZZ[F^*]$-modules. A unit $u\in F^\ast$ acts on $\KMW_\ast(F)$ by multiplication by $\kc{u}$ and on $\cal L^\ast$ by multiplication by $u$. We can thus define the $\cal L$-twisted Milnor-Witt K-theory of $F$ as
\[ \KMW_\ast(F,\cal L):=\KMW_\ast(F)\otimes_{\ZZ[F^*]} \ZZ[\cal L^*]. \]
Therefore, given a discrete valuation $\nu:k(W)\ra\ZZ\cup\{\infty\}$ with uniformizer $\pi$, the twisted residue defined as
\[ \KMW_1(k(W)) \ra \KMW_0(k(W),(\fk m_V/\fk m_V^2)^{\vee}),\quad [a] \longmapsto \partial^{\pi}_{\nu}([a])\otimes \overline{\pi}^{\vee} \]
does not depend on the choice of uniformizer. We can also consider twists on the source of the residue map by acting as follows. Consider a line bundle $\cal L$ defined over the local scheme $\spec{\Ocal_{W,V}}$. We then define a homomorphism
\[ \KMW_1(k(W),\cal L_{k(W)}) \ra \KMW_0(k(V),(\fk m_V/\fk m_V^2)^{\vee}\otimes \cal L_{k(V)}) \]
as follows. Consider an element $[a]\otimes l$, find $u\in k(W)^*$ such that $l=ul'$ where $l'$ is a generator of the $\cal O_{W,V}$-module $\cal L$, so that we can rewrite $[a]\otimes l=[a]\kc{u}\otimes l'$. Then we set
\[ [a]\otimes l \longmapsto  \partial^{\pi}_{\nu}([a]\kc{u})\otimes\overline{\pi}^{\vee}\otimes\overline{l'},\]
where $\overline{l'}$ denotes the image of $l'$ along $\cal L\ra \cal L\otimes k(V)$.

In particular, for every line bundle $\cal L$ on $X$, we can define a residue homomorphism
\[ \KMW_1(k(W),\det(\Omega_{k(W)/k})\otimes \cal L_{k(W)}) \ra \KMW_0(k(V),\det(\Omega_{k(V)/k})\otimes \cal L_{k(V)}), \]
where on the right we are implicitly using the canonical isomorphism 
\[\det(\Omega_{\Ocal_{W,V}/k})\otimes_{\cal O_{W,V}} k(V) \otimes (\fk m_V/\fk m_V^2)^{\vee} \simeq \det(\Omega_{k(V)/k}).  \]
Henceforth, we can define a total residue homomorphism
\begin{equation}\label{eq:coker} \bigoplus_{W\in X_{(i+1)}}\KMW_1(k(W),\det(\Omega_{k(W)/k})\otimes \cal L_{k(W)}) \ra \bigoplus_{V\in X_{(i)}} \KMW_0(k(V),\det(\Omega_{k(V)/k})\otimes \cal L_{k(V)}). \end{equation}
The image of this total residue homomorphism is by definition the group of cycles with coefficients in quadratic forms that are rationally equivalent to zero \emph{in twist} $\cal L$.

The total residue homomorphism is actually part of the so called Rost-Schmid complex: the next map in the sequence is
\begin{equation}\label{eq:ker} \bigoplus_{V\in X_{(i)}}\KMW_0(k(V),\det(\Omega_{k(V)/k})\otimes \cal L_{k(V)}) \ra \bigoplus_{U\in X_{(i-1)}} \KMW_{-1}(k(U),\det(\Omega_{k(U)/k})\otimes \cal L_{k(U)}), \end{equation}
whose definition is completely analogous to the one given before.
\subsubsection{}\label{sub:chow-witt app}
The Chow-Witt group $\CHW_i(X,\cal L)$ of dimension $i$ in twist $\cal L$ is defined as the kernel of (\ref{eq:ker}) modulo the image of (\ref{eq:coker}). For usual Chow groups we don't have to consider the kernel of (\ref{eq:ker}) because $\KM_{-1}=0$.
As a matter of facts, we can always (non-canonically) identify
\[ \bigoplus_{V\in X_{(i)}}\KMW_0(k(V),\cal L) \simeq \bigoplus_{V\in X_{(i)}}\KMW_0(k(V)) \]
by choosing a generator for each $\cal L_{k(V)}$. Therefore, we could still think of Chow-Witt groups as subquotients of cycles with coefficients in quadratic forms, but this should be avoided: changing twists radically changes the rational equivalence we are using, by affecting the definition of residue morphism.

If $X$ is smooth of dimension $n$, then it is also possible to define the (cohomological) Chow-Witt group $\CHW^i(X,\cal L)$ of codimension (or degree) $i$ in twist $\cal L$. The relation between the two groups is as follows:
\[ \CHW^i(X,\cal L) \simeq \CHW_{n-i}(X,\det(\Omega_{X})^{\vee}\otimes\cal L) \]

\printbibliography
\end{document}